\documentclass[11pt]{amsart}

\usepackage{a4wide, amsmath, amsfonts, amssymb, mathrsfs, amsthm, bbm, bm, appendix, dsfont}
\usepackage{lmodern}
\usepackage[margin=1in]{geometry}
\usepackage[hyperfootnotes=false]{hyperref}
\usepackage{comment}

\usepackage{xargs}
\usepackage[pdftex,dvipsnames]{xcolor}  
\usepackage[shortlabels]{enumitem}

\numberwithin{equation}{section}

\newtheorem{theorem}{Theorem}[section]
\newtheorem{lemma}[theorem]{Lemma}
\newtheorem{corollary}[theorem]{Corollary}
\newtheorem{remark}[theorem]{Remark}
\newtheorem{proposition}[theorem]{Proposition}
\newtheorem{definition}[theorem]{Definition}

\newtheorem{assumption}[theorem]{Assumption}

\allowdisplaybreaks

\newcommand{\dd}{\,\mathrm{d}}

\renewcommand{\epsilon}{\varepsilon}
\newcommand{\R}{\mathbb{R}}
\newcommand{\C}{\mathbb{C}}
\newcommand{\T}{\mathbb{T}}
\newcommand{\Z}{\mathbb{Z}}
\newcommand{\N}{\mathbb{N}}
\renewcommand{\P}{\mathbb{P}}
\newcommand{\E}{\mathbb{E}}
\newcommand{\indicator}[1]{\mathbbm{1}_{#1}}
\newcommand{\Id}{\,\mathrm{d}}
\newcommand{\norm}[1]{\left\lVert#1\right\rVert}
\newcommand{\qv}[1]{\langle #1 \rangle}

\newcommand{\cF}{\mathcal{F}}

\newcommand{\Tr}{\mathrm{Tr}}
 
\def \red {\textcolor{red}}

\renewcommand{\hat}{\widehat}

\begin{document} 

\title[Convergence rate for Fluctuations of mean field diffusion]{Convergence rate for Fluctuations of mean field interacting diffusion and application to 2D Viscous Vortex Model and Coulomb potential}

\author{Alekos Cecchin}

\author{Paul Nikolaev}

\address[A Cecchin and P. Nikolaev]
{\newline \indent University of Padova, Department of Mathematics ``Tullio Levi-Civita'', 
 \newline \indent Via Trieste 63, 35121 Padova, Italy}
\email{alekos.cecchin@unipd.it, paul.nikolaev@unipd.it}

\date{August 31, 2025}

\thanks{The authors are supported by the project MeCoGa ``Mean field control and games'' of the University of Padova through the program STARS@UNIPD. A.C. also acknowledge support  from 
the PRIN 2022 Project 2022BEMMLZ ``Stochastic control and games and the role of information'', 
 the INdAM-GNAMPA Project 2025 ``Stochastic control and MFG under asymmetric information: methods and applications'' and the PRIN 2022 PNRR Project P20224TM7Z  ``Probabilistic methods for energy transition''.}



\begin{abstract}
For a system of mean field interacting diffusion on $\T^d$, the empirical measure $\mu^N_{\bm{X}}$ converges to the solution $\mu$ of the Fokker-Planck equation. Refining this mean field limit as a Central Limit Theorem, the fluctuation process $\rho^N_t= \sqrt{N}( \mu^N_{\bm{X}} -\mu_t)$ convergences to the solution $\rho$ of a linear stochastic PDE on the negative Sobolev space $H^{-\lambda-2}(\T^d)$. The main result of the paper is to establish a rate for such convergence: we show that  $|\E[\Phi(\rho_t^N) - \Phi(\rho_t)]| = \mathcal{O}(\tfrac{1}{\sqrt{N}})$, for smooth functions on $H^{-\lambda-2}(\T^d)$. The strategy relies on studying the generators of the processes $\rho^N$ and $\rho$ on $H^{-\lambda-2}(\T^d)$, and thus estimating their difference. Among others, this requires to approximate in probability $\rho$  with solutions to  stochastic diffential equations on the Hilbert space $H^{-\lambda-2}(\T^d)$. The flexibility of the approach permits to establish a rate for the fluctuations, not only in case of a regular drift, but also for the the 2D viscous Vortex model, governed by the Biot-Savart kernel, and  for the repulsive Coulomb potential. 
\end{abstract}

\maketitle

\tableofcontents

\section{Introduction}

In this work, we analyze the asymptotic fluctuation as \(N \to \infty\) of the particle system on the torus \(\T^d\) given by
\begin{equation} 
\label{eq: interacting_particle_system}
    \Id X_t^{i} = b(t,X_t^{i}, \mu_t^N) \Id t + \sigma \Id B_t^{i}, \quad i =1, \ldots , N,
\end{equation}
for $t\in (0,T]$, where $\sigma>0$ and $\mu_t^N = \frac1N\sum\limits_{i=1}^N \delta_{X_t^{i}}$
denotes the empirical measure of the system \((X^{i}, i \in \N)\). 
Mean field interacting particle systems of the form~\eqref{eq: interacting_particle_system} are widely used in various contexts. In biology, they appear in models describing the collective behavior of animals and micro-organisms, such as flocking, swarming, and chemotaxis processes~\cite{Carrillo2014}, as well as opinion dynamics~\cite{hegselmann2002,nikolaev2025}. In physics, such systems are employed to model large-scale structures like galaxies~\cite{Jeans15}, the dynamics of ions and electrons in plasma physics~\cite{Dobrushin1979}, phenomena in fluid dynamics~\cite{Onsager1949}, and problems in statistical mechanics~\cite{Serfaty2020}. More recently, they have been applied to modern fields such as neural networks~\cite{Moynot2002}, neuroscience~\cite{Baladron2012}, and optimization~\cite{Carrillo2018}. Another wide rang of applications is to the recent theories of mean field games and mean field control problems, started in the papers \cite{huang2006LargePopulationStochastic, lasry2007mean}.

A central problem in the study of mean field interacting particle systems is to understand the asymptotic behavior of the empirical measure \(\mu_t^N\) as \(N \to \infty\), a regime referred to as the mean field limit. This concept was first introduced by Boltzmann~\cite{Boltzmann1970} in the context of kinetic theory to describe the phenomenon of molecular chaos. 
If the law of the system is asymptotically i.i.d., then Propagation of Chaos is said to hold for such systems. 
Under suitable regularity assumptions on the drift coefficient \(b\), it can be shown that the empirical measure \(\mu_t^N\) converges, as \(N \to \infty\), to a deterministic measure-valued solution \((\mu_t, 0 \leq t \leq T)\) of the Fokker--Planck equation
\begin{equation} \label{eq: fokker--planck}
    \partial_t \mu_t = \frac{\sigma^2}{2} \Delta \mu_t - \nabla \cdot (b(t,\cdot,\mu_t) \mu_t),
\end{equation}
which describes the evolution of the law of a typical particle whose trajectory  is governed by the McKean--Vlasov equation
\begin{equation} \label{eq: mckean_vlasov}
    \Id Y_t = b(t,Y_t, \mathbb{P}_{Y_t}) \Id t + \sigma \Id B_t.
\end{equation} 
There is a huge literature on Propagation of Chaos for mean field model: see e.g. ~\cite{Sznitman1991} for a classical proof in the case of regular mean-field interaction, and \cite{Carmona2018vol1, cardaliaguet2019master} for the related convergence problem in mean field games. 
Embedded in the scaling of the empirical measure is the factor \(\tfrac{1}{N}\), which reflects both the weak interaction between particles and a law of large numbers type behavior. Therefore, the convergence result described above can be interpreted as a law of large numbers for the interaction particle system~\eqref{eq: interacting_particle_system}.

In this article we are interested in the asymptotic behavior of the system~\eqref{eq: interacting_particle_system} in the fluctuation regime, which plays an important role for the applications. That is, we are interested in the asymptotic behavior of the fluctuation process
\begin{equation} 
\label{eq: def_fluctuation_process}
    \rho_t^N := \sqrt{N}(\mu_t^N - \mu_t),
\end{equation}
which provides a sort of Central Limit Theorem for the mean field limit. 
As shown for instance in~\cite{Tanaka1981} it is expected that in the limit the process converges towards the SPDE
\begin{equation} 
\label{eq: limiting_spde}
\partial_t \rho_t = -  \nabla \cdot (\rho_t b(t,\mu_t)) - \bigg\langle \nabla_x \cdot  \bigg(  \mu_t  \frac{\delta b}{\delta m} (t,\cdot,\mu_t , v ) \bigg), \rho_t(v)  \bigg\rangle 
+  \frac{\sigma^2}{2}  \Delta \rho_t   \Id t
 - \nabla \cdot \big( \sigma^{\mathrm{T}} \sqrt{\mu_t} \xi \big) , 
\end{equation}
where \(\xi\) denotes a Gaussian noise (sometimes called white noise); for a precise definition see \S \ref{subsec: Gaussian_noise}.
Notice that, while the Fokker--Planck equation~\eqref{eq: fokker--planck} is non-linear, the SPDE~\eqref{eq: limiting_spde} is a linear parabolic SPDE with additive noise. It holds in the sense of distribution and is set on a negative Sobolev space $H^{-\lambda-2}(\T^d)$, where the constant $\lambda$ depends on the dimension $d$. Such dependence indeed guarantees that convergence of $\mu^N_t$ to $\mu_t$ on that space is of order $1/\sqrt{N}$, thus allowing to study fluctuations, while it is known that the convergence rate in e.g. Wasserstein metrics suffers form the dimension; see~\cite{Fournier2015}. The need to embed the signed measure $\rho^N_t$ into a Hilbert space comes also in order to have a good structure and apply known results.

Our primary motivation for this work is the analysis of the point Vortex model, which corresponds to system~\eqref{eq: interacting_particle_system} with \(b(t,x,m)= K*m\) and \(K\) given by the Biot–Savart kernel
\begin{equation} \label{eq: introd_biot_savart_kernel}
    K(x_1,x_2) = \frac{1}{2\pi}\frac{(-x_2,x_1)}{|x|^2}.
\end{equation}
The well posedness of~\eqref{eq: interacting_particle_system} was established for instance by Osada~\cite{Osada1985} under sufficient regularity on the initial data, which will be recalled in Section~\ref{sec: application}. 
This model is closely connected to the two-dimensional incompressible Navier–Stokes equations, given by
\begin{align*}
    \partial_t u = \Delta u - u \cdot \nabla u - \nabla p,
\end{align*}
where \(p\) denotes the local pressure. Taking the curl of this equation, it is a classical result~\cite{Marchioro1994} that the vorticity \(\omega(t,x) = \nabla \times u(t,x)\) satisfies the Fokker–Planck equation~\eqref{eq: fokker--planck} with \(K\) given by~\eqref{eq: introd_biot_savart_kernel}. This establishes a relationship between the Navier–Stokes dynamics and the corresponding Fokker–Planck formulation.

The point Vortex approximation to the two-dimensional Navier--Stokes and Euler equations has attracted considerable attention since the 1980s. A first result on Propagation of Chaos was provided by Osada~\cite{Osada1986}, who established convergence for bounded initial data under the assumption of large viscosity \(\sigma\). Using a different approach based on compactness methods,~\cite{Fournier2014} proved entropic Propagation of Chaos without any restriction on \(\sigma\), as long as it remains strictly positive. Later, Wang and Jabin~\cite{JabinWang2018} derived a quantitative version of Propagation of Chaos by analyzing the evolution of relative entropy and employing an exponential law of large numbers. 
Building on the same method, uniform-in-time Propagation of Chaos is established in~\cite{Guillin2025}.

Besides the Vortex model, we are also interested in further particle systems with singular interaction kernels \(K\). In recent years, significant progress has been made in the study of Propagation of Chaos for such singular kernels. We refer to the non-exhaustive list of seminal works~\cite{Fournier2014,JabinWang2018,Serfaty2020,Bresch2023}. 
In particular, the Coulomb kernel
\begin{equation*}
 K(x) = \begin{cases}
     \nabla \log(|x|) , \quad &d=2 \\
        \nabla |x|^{2-d}  , \quad &d \ge 3 
 \end{cases}   
\end{equation*}
is of central importance in the fields of statistical and quantum mechanics.

The aim of this paper is to establish a rate for the convergence of the weak error 
\begin{equation}
\label{eq:uno}
|\E[\Phi(\rho_t^N)] - \E[\Phi(\rho_t)]|
\end{equation} 
for smooth $\Phi: H^{-\lambda-2}(\T^d) \rightarrow \R$. We study the convergence rate both for regular interaction and for the Vortex and Coulomb models.  

\subsection{Related works on Fluctuations}

The convergence of $\rho^N$ to $\rho$ in distribution on a negative Sobolev space has been studied in several papers, always by proving tightness of the sequence and then convergence in law by means of a martingale central limit theorem. 
The first results in this direction are provided in \cite{Tanaka1981} and  \cite{HITSUDA1986311} in case of linear interaction; see also  \cite{Szn85}. The case of nonlinear and smooth interaction is first studied first  \cite{KurtzXiong2004}, where a common noise is also present, but the analysis is restricted to dimension one. In higher dimension, a suitable weighted negative Sobolev space for the Hilbertian approach to fluctuations for dynamics in $\R^d$ is introduced in \cite{Meleard1996} and further developed in \cite{Fernandez1997}, but the analysis therein is restricted to linear interaction. The case of moderate interaction is studied in~\cite{Jourdain1998}.  Fluctuations for the case of smooth and nonlinear interaction in arbitrary dimension, and also in presence of common noise,  are studied in \cite{Delarue2019} by exploiting the weighted spaces introduced in \cite{Meleard1996}, and their result is applied to the convergence problem in mean field games; see also \cite{delarue2021mean}. 

Recently, the compactness method in negative Sobolev spaces has been applied to obtain fluctuation results, i.e. convergence in law of $\rho^N$ to $\rho$, also for dynamics with irregular interaction. 
Building on the relative entropy bounds obtained uniformly in \(N \in \N\) in~\cite{JabinWang2018},  Wang, Zhao, and Zhu~\cite{Xianliang2023} obtained a Gaussian fluctuation result for the Vortex model. Based on the same idea, Shao and Zhao~\cite{shao2024} and the second author~\cite{Nikolaev2025_fluc} extended the Gaussian fluctuation result to the stochastic version of the Navier--Stokes equation and the Langevin dynamics, where a common source of randomness is introduced into the particle system~\eqref{eq: interacting_particle_system}. 
For the Coulomb kernel, the main fluctuation result has been obtained by Serfaty~\cite{Serfaty2023}.

Let us mention some recent works related some particular cases of quantitative fluctuations.
The papers 
\cite{JourdainTse21, flenghi2022} show that, for any smooth $U:\mathcal{P}_2(\R^d)\rightarrow\R$, the process 
$\sqrt{N}(U(\mu^N_t)-U(\mu_t))_t$ converges to a Gaussian process (dependent on $U$). Their method is to compare the PDEs satisfied by $U(\mu^N_t)$ and $U(\mu_t)$ on $\mathcal{P}_2(\R^d)$, by using the differential calculus on the space of measures 
(see \cite{Carmona2018vol1, cardaliaguet2019master}) and thus improving the quantitative weak Propagation of Chaos analyzed in \cite{CST22}. 
The recent preprint \cite{bernou2025} establishes a rate for the weak convergence of $\langle \rho^N_t, \varphi \rangle$ to $\langle \rho_t, \varphi \rangle$, for any smooth $\varphi: \R^d\rightarrow \R$, that is, they study a version of \eqref{eq:uno} for a linear $\Phi$, but they treat a Langevin dynamics and obtain a uniform in time rate.  The preprint \cite{gu2024} considers a specific irregular Langevin dynamics and establishes quantitative fluctuations for a particular functional of the empirical measure. 

A common and powerful tool to study the convergence rate in Central Limit Theorems is Stein's method. The basic idea is to view the limiting measure as an invariant measure of some process: it is widely used and has been applied in many problems, also for stochastic processes, and also in combination with Malliavin calculus; we refer to \cite{azmoodeh_peccati_yang_2021} for a recent survey. However, is seems to not have employed so far for stochastic processes valued in Hilbert spaces, which is the case we treat here. 
We finally mention that the convergence rate for \eqref{eq:uno} has been analyzed for mean field interacting systems of finite state Markov chains, obtaining an estimate similar to ours. In that case the state space of both $\rho^N$ and $\rho$ is a subset of $\R^d$, which  is finite dimensional and simplifies a lot the analysis. We refer to 
\cite[Prop. 5.11.3]{kolokoltsov2011markov} for the estimate via an analysis of the generators of the processes, which partially motivated the present work; see also \cite{Kolokoltsov2010} for the application of that technique to other type of processes and
\cite{CECCHIN2019} for an application to finite state mean field games.

\subsection{Our contribution}

The main result of the paper is to establish a rate for the weak convergence of $\rho^N$ to $\rho$: we show that 
\begin{equation}
\label{eq:main_intro}
  \sup\limits_{0 \le t \le T} |\E[\Phi(\rho_t^N)] - \E[\Phi(\rho_t)]| \le C [\Phi]_{C^2(H^{-\lambda-2}(\T^d))}  \bigg( \frac{1}{\sqrt N } +   
      W_{1, H^{-\lambda-2}(\T^d)}\Big( \P_{\rho^N_0}, \P_{\rho_0}\Big) 
    \bigg) 
\end{equation}
for a smooth function $\Phi: H^{-\lambda-2}(\T^d) \rightarrow \R$. This is a first quantitative estimate for mean field fluctuations and may permit to better understand the relation between the fluctuation process and its limit.  Compactness methods clearly can not provide a convergence rate, thus we employ a different approach in order to obtain the result. Roughly speaking, we study the generator of the fluctuation process and of the SPDE on $H^{-\lambda-2}(\T^d)$, and estimate the difference. The convergence rate or order $1/\sqrt{N}$ is common to all quantitative Central Limit Theorems, as given e.g. by Berry-Esseen Theorem or by Stein method. The flexibility of our approach permits to treat, in the same way, not only a smooth drift, but also  the case of irregular interaction kernel, in particular the Vortex and Coulomb model, hence improving with a convergence rate the fluctuation results recently obtained by compactness methods in~\cite{Xianliang2023} and \cite{Serfaty2023}.

The idea of this work is to study the generators of the the processes $\rho^N$ in 
\eqref{eq: def_fluctuation_process} and $\rho$ in \eqref{eq: limiting_spde}, and then to estimate their difference to otbain the  convergence rate. The method in \cite{kolokoltsov2011markov}, however, is not directly applicable, because $\rho^N$ and $\rho$ take values is a infinite dimensional space. More importantly, the process $\rho$ solves a SPDE in a space of distribution and is a kind of Ornstein-Uhlenbeck process in a Hilbert space, which is not strongly continuous. Thus we can not talk about generators of the processes in the usual sense and, moreover, we can not apply the usual technique to pass estimates from generators to semigroups in the proof of Trotter-Kato Theorem; see \cite[eq. I,(6.1)]{Ethier1986}.    Although there are techniques to deal with non strongly continuous semigroups in infinite-dimensional Hilbert spaces (see e.g.~\cite[\S B.5]{Gozzi2017}), they seem to not allow in general to prove a product formula as in \cite[eq. I,(6.1)]{Ethier1986}. Therefore, we have to proceed in another way to prove the main estimate \eqref{eq:main_intro}.  

As another main result, we introduce a mollification procedure which allows to approximate the limiting process $\rho$ with solutions to SPDEs with smoother coefficients which, importantly, can be written as solutions of Hilbert space valued stochastic differential equation. This fact, on one hand, allows to apply It\^o formula for It\^o processes on Hilbert spaces and hence to write a sort of generator of the approximating process $\rho^n$ on $H^{-\lambda-2}(\T^d)$, although the corresponding semigroup is not strongly continuous.  On the other hand, by employing techniques used for nonlinear SPDEs (see \cite{LiuWei2015SPDE,Rozovsky2018}), as another novelty of the paper, we show the crucial fact that the approximation $\rho^n$ converges to $\rho$ in probability on the path space.
Indeed, the usual weak convergence (in the sense of functional analysis) commonly obtained for linear SPDEs, e.g. for the finite dimensional approximation, would not be sufficient here, since the function $\Phi$ in \eqref{eq:main_intro} is non-linear. We refer to Definition \ref{def: spde_solution} for the notion of solution to \eqref{eq: limiting_spde} within the theory of SPDEs and to Section \ref{sec: SPDE_existence} for the details on the approximation procedure, which also shows existence and uniqueness of solution to \eqref{eq: limiting_spde}.


As far as the process $\rho^N$ is concerned, we use the standard It\^o formula in $\R^d$ in order to write its generator, and exploit the calculus rules for derivatives of functions along probability measures and their restriction to empirical measures detailed in \cite{Carmona2018, cardaliaguet2019master}. We do not require tightness of $\rho^N$ in order to show the main result \eqref{eq:main_intro}. 
On a technical side, we find that the second order term 
\[
\frac{1}{N^2} \partial^2_{\mu\mu} \Phi(\mu^N_{\bm{x}}; x_i,x_i), 
\]
appearing in the expansion of the second order pure derivative $\partial_{x_i x_i}\Phi(\mu^N_{\bm{x}})$ for the restriction of functions along empirical measures, is exactly, after rescaling and integration, what is given by It\^o formula in the Hilbert space $H^{-\lambda-2}(\T^d)$ as the trace of a linear operator containing the second Fréchet derivative of $\Phi$. 
To this end, we need to consider second derivatives on $H^{-\lambda-2}(\T^d)$ as functions  of two variables and study their regularity. 
Finally, to deal with the lack of strong continuity of the semigroups, we develop a technique partially presented in \cite{gess2024}. 
Our main result holds for $\Phi$ of linear growth, which in particular includes the linear functionals considered in \cite{bernou2025}. As another result, differently from other quantitative estimates, our main estimate \eqref{eq:main_intro} may imply convergence in law of the fluctuation process to its limit.


In this work, we will also make use of the relative entropy estimates from~\cite{JabinWang2018} to derive our main inequality \eqref{eq:main_intro}. Our method of proof enables to obtain the main estimate \eqref{eq:main_intro} also for the Vortex model: see \S \ref{sec:vortex}. 
 For the Coulomb kernel, we also derive our main estimate \eqref{eq:main_intro} in dimensions \(d = 2,3\), but with a convergence rate depending on the dimension: see \S \ref{sec:couulomb}.
The dimensional restriction is consistent with the work by Serfaty~\cite{Serfaty2023} on Gaussian fluctuations and, in our case, follows from the sharp estimates on the modulated free energy in ~\cite{Serfaty2025}.

\subsection{Perspectives}
The same result on the convergence rate for the fluctuations could be obtained for some more general dynamics than \eqref{eq: interacting_particle_system}. For instance, we could include a volatility $\sigma(t,x)$ depending on time and space, and non-degenerate. We prefer not to consider this case, as it would complicate more the notations. Instead, we prefer to present the main ideas and to show that our approach permits to treat cases of irregular drift. Several other more difficult generalizations could be considered. First of all, we consider dynamics on the torus and not on $\R^d$; one of the main reason for this restriction is that is is not clear to us how to approximate the solution to \eqref{eq: limiting_spde} with operators on weighted Sobolev spaces on $\R^d$ which have all the properties requires in Section  \ref{sec: SPDE_existence}. Another interesting question would be to consider a non-degenerate and measure-dependent volatility, as in the fluctuation results in \cite{Meleard1996, Fernandez1997}. We remark, however, that in case of linear functionals, the convergence rate established in \cite{bernou2025} gets worse if $\sigma$ is degenerate. Other possible generalizations may include a common noise and a rate uniform in time. Is is not clear how to treat these questions with the techniques presented in the paper, and are left to future work.

\subsection{Organization of the paper}
In Section~\ref{sec: setting}, we fist provide the 
notation and the assumptions. 
Then we introduce the Gaussian noise and the Definition of solution to \eqref{eq: limiting_spde}, stating its well-posedness in Theorem \ref{theorem: existence_SPDE}. We thus provide the main result Theorem \ref{theorem: main_result}, which establishes the convergence rate in case of drift with bounded measure derivatives. We also show how this result may imply convergence in law of $\rho^N$ to $\rho$. In Section \ref{sec: SPDE_existence} we introduce a mollification $\rho^n$ of the SPDE \eqref{eq: limiting_spde}, which writes as SDE in $H^{-\lambda-2}(\T^d)$, and show that it is well-posed.  The crucial result on convergence in probability of $\rho^n$ to $\rho$ is Theorem \ref{thm: conv_prob}. We then study the regularity of the flow of $\rho^n$. In section \ref{sec:generators}, we first study the regularity and stability of the semigroup related to $\rho^n$, and also the regularity of its derivatives viewed as Sobolev functions. Then we write its (backward) generator and also in Proposition \ref{prop:4.9}. Thus we show how derivatives in  $H^{-\lambda-2}(\T^d)$ are related to flat derivatives in $\mathcal{P}(\T^d)$ and write the generator of $\rho^N$ in Proposition \ref{prop: generator_n_fluctuation}. In Section \ref{sec: comparison} we compare the generators of $\rho^N$ and $\rho^n$ and estimate the remainder to prove the main result Theorem \ref{theorem: main_result}. 
In Section~\ref{sec: application} we derive the main estimate in Theorem~\ref{theorem: main_result} for the Biot-Savart kernel and the repulsive Coulomb kernel, by exploiting the specific features of these models in order to deal with irregular kernels; see Theorems \ref{thm:main:vortex} and \ref{theorem: repulsive_fluc}.  
Finally, in Appendix \ref{sec:appendix} we first recall some properties of Sobolev and Besov spaces, and state a regularity result for Sobolev functions on the diagonal. Then we recall the Gy\"ongy--Krylov criterion for convergence in probability and state a result on approximation by cylindrical functions on Hilbert spaces.

\section{Setting and Main Result} 
\label{sec: setting}
In this section we introduce the main framework for our result as well as provide the main result. 

\subsection{Notation}
\label{sec:notation}
We write a vector in \(\T^{d}\) as \(x = (x_1, \ldots, x_d) \in \T^{d} \).  
Throughout the entire paper, we use the generic constant \( C \) for inequalities, which may change from line to line at may depend on the dimension \(d\) and final time \(T\). For \(z \in \C\) we write \(\overline z \) for the complex conjugate of \(z\). Given a probability space $(\Omega, \mathcal{A}, \mathbb{P})$, we denode by $\P_{X}$ the law of a random variable $X: \Omega \rightarrow \R^m$. 

Given a linear operator between (separable) Hilbert spaces $L: H_1\rightarrow H_2$, the adjoint operator is denoted by $L^*:H_2\rightarrow H_1$.
For a Hilbert space $H$, we adopt the standard identification of the first and second Fréchet derivatives \(\nabla\) and \(\nabla^2\) as follows. For a real-valued twice Fréchet differentiable \(\Phi \in C^2(H)\), the first derivative takes the form
\[
\nabla \Phi \colon H \to H,
\]
where the identification of the space \(L(H,\R)\) with \(H\) is made via the Riesz representation theorem. The second derivative is then given by
\[
\nabla^2 \Phi \colon H \to L(H, H),
\]
where \(L(H, H)\) denotes the space of bounded linear operators from \(H\) to \(H\). 
We denote the linear-growth norm 
\[
\|\Phi\|_{C_\ell(H)}:= \sup_{z\in H} \frac{|\Phi(z)|}{1+ \|z\|_H} 
\]
and the seminorms 
\begin{equation*}
    [\Phi]_{C^1(H)}:= \sup_{z \in H}\|\nabla \Phi(z)\|_{H}, \qquad 
     [\Phi]_{C^2(H)}:= \sup_{z \in H}\|\nabla^2 \Phi(z)\|_{L(H,H)} + [\Phi]_{C^1(H)}, 
\end{equation*}
and let $C^2_\ell(H)$ be the subset of $\Phi\in C^2(H)$ such that $\|\Phi\|_{C_\ell(H)}$ and $[\Phi]_{C^2(H)}$ are finite, that is, functions of linear growth with bounded first and second derivative. We also denote by $C_b(H)$ the set of bounded continuous functions with the uniform norm and by $C^2_b$ its subset of $C^2(H)$ functions  with bounded first and second derivatives. The Trace of a (Trace Class) operator $L\in L(H,H)$ is denoted by $\mathrm{Tr}(L)$.

\noindent In contrast, for scalar functions \(u\) defined on \(\R^d\) or \(\T^d\), we employ classical differential notation: \(D u\) denotes the gradient, \(D^2 u\) the Hessian, \(\partial_{x_i} u\) the partial derivative in the \(i\)-th coordinate, \(\nabla \cdot u\) the divergence, and \(\Delta u\) the Laplacian.


We introduce the space of Schwarz distributions \(\mathcal{S}'(\T^d)\). We denote dual parings by \(\qv{\cdot, \cdot}\). For instance, for \(f \in \mathcal{S}', \; u \in C^\infty(\T^d) \) we have \(\qv{f,u}  = \qv{u,f}= f[u]\) and for a probability measure~\(\mu\) we have \(\qv{u,\mu}  = \int u \Id \mu\). The correct interpretation will be clear from the context but should not be confused with the scalar product \(\qv{\cdot, \cdot}_{H}\) for some arbitrary Hilbert space \(H\). If we mean the scalar product, we  write the corresponding space \(H\) as subscript onto the scalar product. 
For the space \(L^2(\T^d)\) we define the following orthonormal basis \((e_k, k \in \Z^d)\) given by 
\begin{equation*}
    e_k(x) := \exp(2\pi i k \cdot x) , 
\end{equation*}
where \(i\) denotes the imaginary unit in the complex numbers \(\C\). 
For simplicity, we denote by \( \langle k \rangle := (1+|k|^2)^{\tfrac{1}{2}} \) for \(k \in \Z^d\). Then, for \(s \in \R\) we can define the following Sobolev space 
\begin{equation*}
    H^{s}(\T^d) := \{ f \in \mathcal S'(\T^d) : \norm{f}_{H^{s}(\T^d)} < \infty\},
\end{equation*}
where \(\norm{\cdot}_{H^{s}(\T^d)}\) is induced by the scalar product 
\begin{equation*}
    \langle f, g \rangle_{H^{s}(\T^d)} := \sum\limits_{k \in \Z^d } \langle k \rangle^{2s} \langle f, e_k \rangle \overline{\langle g,e_k \rangle}. 
\end{equation*}
Note that $f$ is a function if $s\geq 0$ and a distribution if $s<0$. Given a smooth function $\varphi$,  it will often be seen as a distribution, without changing the notation, which will be clear from the context. For instance, given $s>0$, $f\in H^{-s}(\T^d)$ and $\varphi\in H^s(\T^d)$, when we write $\langle f, \varphi \rangle_{H^{-s}(\T^d)}$, we mean the scalar product with the distribution $\varphi\in \mathcal{S}'(\T^d)$ defined by $\varphi(\psi)= \langle \varphi, \psi \rangle_{L^2(\T^d)}$, for any $\psi\in C^{\infty}(\T^d)$.

A dyadic partition of unity $(\tilde \chi,\chi)$ in dimension $d$ is given by two smooth symmetric functions on $\R^d$ satisfying $\mathrm{supp}\;  \tilde \chi\subseteq \{x \in \R^d \, : \, |x| \le 2 \}$, $\mathrm{supp}\; \chi \subseteq\{x \in\R^d: 1\le|x|\le 4 \}$ and $\tilde \chi(z)+\sum_{j \ge 0}\chi(2^{-j}z)=1$ for all $z\in\R^d$. We set
\begin{equation*}
  \chi_{0}:=\tilde \chi\quad\text{and}\quad\chi_{j}:=\chi(2^{-j-1}\cdot)\quad\text{for }j\ge1.
\end{equation*}
For \(1 \le p \le \infty\), \(1 \le q \le \infty\) we definie the Besov spaces 
\begin{equation*}
    B_{p,q}^s(\T^d) := \big \{  f \in \mathcal S'(\T^d) : \norm{f}_{B_{p,q}^s(\T^d)} < \infty    \big \},
\end{equation*}
where \(\norm{f}_{B_{p,q}^s(\T^d)}\) is given by 
\begin{equation*}
    \norm{f}_{B_{p,q}^s(\T^d)}
:= \bigg( \sum\limits_{j=0}^\infty 2^{sjq} \norm{\sum\limits_{k \in \Z^d }  \chi_j(k) \langle f, e_k \rangle \exp(-2\pi i k \cdot x )}_{L^p(\T^d)}^q   \bigg)^{\frac{1}{q}}
\end{equation*}
with the usual convention if \(q= \infty\) or \(p = \infty\). It is well-known fact that \(B_{\infty,\infty}^s(\T^d)\) is equal to periodic Hölder spaces for positive non-integer \(s\)~\cite[Theorem~3.5.4]{Schmeisser1987}.

For a Banach space \((E, \norm{\cdot}_{E})\), some filtration \((\cF_t, 0 \leq t \leq T)\), \(1 \le p < \infty\) and \(0 \le s <t \le T\) we denote by \(S^p_{\cF}([s,t];E )\) the set of \(E\)-valued \((\mathcal{F}_t)\)-adapted continuous processes \((Z_u, u \in [s,t])\) such that 
%
\begin{equation*}
  \norm{Z}_{S^p_{\cF}([s,t];E )}:=
  \bigg( \E \bigg[ \sup\limits_{u \in [s,t]} \norm{Z_u}_E^p\bigg] \bigg)^{\frac{1}{p}}
\end{equation*}
is finite. Similar, \(L^p_{\cF}([s,t];E)\) denotes the set of \(E\)-valued predictable processes \((Z_u, u \in [s,t])\) such that
\begin{equation*}
  \norm{Z}_{L^p_{\cF}([s,t];E )}:=
  \bigg( \E \bigg[ \int\limits_{s}^t  \norm{Z_u}_E^p \Id u \bigg] \bigg)^{\frac{1}{p}}
\end{equation*}
is finite. 
Similar definitions hold for the deterministic counterpart \(L^p([s,t];E)\) and time independent counterpart \(L^p_{\cF}(E)\).

The set of probability measures on $\T^d$ is denoted by \(\mathcal P (\T^d) \) and is endowed with the 1-Wasserstein distance, denoted by $W_1$. The relative entropy is denoted by $\mathcal{H}(\mu | \nu)$, also for probability measures on $\T^{ld}$. We denote by $W_{1,H}$ the 1-Wasserstein distance between probability measures on a Hilbert space $H$. 
We also require derivatives on \(\mathcal P (\T^d) \): we refer to \cite{cardaliaguet2019master} for the details. We say that \(U \colon \mathcal P (\T^d) \to \R \) is differentiable, if for every \(m , \tilde m \in \mathcal P(\T^d) \) we have 
\begin{equation*}
    U(m) - U(\tilde m) = \int\limits_0^1 \int_{\T^d} \frac{\delta U}{\delta m } ((1-r)\tilde m + t m, v ) \Id (m -\tilde m )(v) \Id  r, 
\end{equation*}
where $\frac{\delta U}{\delta m }$ is called the \emph{flat} derivative. 
  Similar, \(U\) is twice differentiable if for every \(v\) the map \(m \mapsto \tfrac{\delta U}{\delta m} (m,v) \) is differentiable and we denote its second flat derivative by \(\tfrac{\delta^2 U }{\delta m^2}\), i.e. 
\begin{equation*}
    \frac{\delta^2 U}{\delta m^2}(m,v,v') =  \frac{\delta }{\delta m}\left(\frac{\delta U}{\delta m}(\cdot ,v) \right) (m,v') . 
\end{equation*}
Continuous differentiability and properties of the derivatives when computed along empirial measures are recalled when needed in \S \ref{sec:generator_fluctuation}.

Let us finally fix the constants $\lambda, \lambda' \in \R$  
\begin{equation} 
\label{eq: const}
\begin{split}
    \lambda &> \frac32 d, \\
    \lambda' &> \lambda +1.
    \end{split}
\end{equation}

\subsection{Assumptions}
\label{sec:assumptions}

We state three sets of assumptions for the several results we prove. The first imply existence of the particle system and the Fokker-Planck equation as well existence and stability and approximation results for the SPDE \eqref{eq: limiting_spde}. 

\begin{assumption}[Initial condition] \label{ass: inital_cond}
Suppose \(\rho_0 \in L^2_{\cF_0}(H^{-\lambda-2}(\T^d))  \). 


\end{assumption}
\begin{assumption}[Existence] \label{ass: existence}
The coefficient \(b(t, \cdot, \mu_t) : [0,T] \times \T^d \times \mathcal P(\T^d) \to \T^d\), interacting particle system \((X^i , i \in \N)\) and Fokker--Planck solution \((\mu_t, 0 \leq t \leq T)\) satisfy:
    \begin{enumerate}
    \item There exists a probabilistic weak solution of the interacting particle system~\eqref{eq: interacting_particle_system}, with initial condition \((X^i_0 , i \in \N)\). 
    \item There exists a non-negative solution \((\mu_t, 0 \leq t \leq T) \) in $L^1(\T^d)\cap \mathcal{P}(\T^d)$ to the Fokker--Planck equation~\eqref{eq: fokker--planck} in the sense of distribution, 
    with initial condition \(\mu_0\).  
        \item The following stability estimates hold: 
        \begin{align}
    \norm{\nabla \cdot (b(t,\cdot,\mu_t) f) }_{ H^{-\lambda-2}(\T^d)}^2 & \le  C  \norm{f}_{ H^{-\lambda-1}(\T^d)}^2, \label{eq: b_ineq}\\
    \norm{\bigg\langle \nabla_x \cdot  \bigg(  \mu_t(x)  \frac{\delta b}{\delta m} (t,x,\mu_t , v ) \bigg), f(v)  \bigg\rangle}_{ H^{-\lambda-2}(\T^d)}^2
    &\le C \norm{f}_{H^{-\lambda-1}(\T^d)}^2
    \label{eq: deriv_prob_inequ}
\end{align}
    \end{enumerate}
\end{assumption}

The next Assumption gives the regularity in time in order to apply It\^o formula and compute the generator of the limiting SPDE as well as the fluctuation process. 
\begin{assumption}[It\^o formula]\label{ass: coef_fokker} 
Let \((j_n , n \in \N)\) be a mollifier on \(\T^d\). Suppose the coefficient \(b(t, \cdot, \mu_t) : [0,T] \times \T^d \times \mathcal P(\T^d) \to \T^d\) and Fokker--Planck solution \((\mu_t, 0 \leq t \leq T)\) satisfy: 
\begin{enumerate}
    \item The function \(t \mapsto b(t, \cdot, \mu_t)\) lies in \(C([0,T];H^{\lambda'}(\T^d))\).
    \item  \(\lim\limits_{s \to t} \norm{j_n*  \bigg\langle \nabla_x \cdot  \bigg(  \mu_t(\cdot )  \frac{\delta b}{\delta m} (t,\cdot ,\mu_t , v )- \mu_s(\cdot)  \frac{\delta b}{\delta m} (s,\cdot,\mu_s , v ) \bigg), j_n* f(v)  \bigg\rangle  \bigg\rangle}_{H^{-\lambda-2}(\T^d)} =0. \)
    \item The function \(t \mapsto \mu_t\) satisfies \(\lim\limits_{s \to t } \norm{\sqrt{\mu_t}-\sqrt{\mu_s}}_{L^2(\T^d)}^2 = 0. \)
\end{enumerate}
\end{assumption}
These three assumptions are quite general and hold for both a general smooth non-linear interaction as well as for specific linear irregular interaction. We shall show in Section \ref{sec: application} that they are verified for the Vortex and  Coulomb models.

The last Assumption enables to estimate the reminder in the difference of the semigroups and hence to prove the main result. 
Here and throughout the paper, for $t\in[0,T]$,
 \(\mu_t^{\otimes l}= \mu_t \otimes\cdots \otimes \mu_t\) is the \(l\)-the tensorized version of \(\mu_t\) on \(\T^{dl}\), and  
    \(\bar \mu_t^{N} = \P_{(X^1_t, \dots, X^N_t)} \) is the law of the whole interacting particle system~\eqref{eq: interacting_particle_system} on 
    \(\T^{dN}\).

\begin{assumption}[Mean-field limit]
\label{ass: convergence_ass}
 \textcolor{white}{a}
\begin{enumerate}
        \item \(\sup\limits_{0 \le t \le T}  \sup\limits_{m \in \mathcal P(\T^d)} \norm{\frac{\delta b}{\delta m} (t, \cdot,m , \cdot )}_{L^\infty(\T^d \times \T^d) } < \infty\). 
        \item \(\sup\limits_{0 \le t \le T}  \sup\limits_{m \in \mathcal P(\T^d)} \norm{\frac{\delta^2 b}{\delta m^2} (t, \cdot,m , \cdot ,\cdot )}_{L^\infty(\T^d \times \T^d \times \T^d ) } < \infty\). 
        \item 
        \(\sup\limits_{N \in \N} \sup\limits_{0\leq t\leq T} \mathcal{H}\big(\bar \mu_t^N \vert \mu_t^{\otimes N}\big) < \infty\). 
\end{enumerate}
\end{assumption}
This assumption is clearly concerned with the case of regular interaction which we treat in the main result below; see  Theorem \ref{theorem: main_result}. We show in Section \ref{sec: application} how to prove the main result for cases of irregular interaction, by avoiding this assumption and using instead specific features of the Vortex and Coulomb model.  

\begin{remark}
We observe the following on the assumptions:
\begin{itemize}
\item We do not assume that the initial conditions of the particle system $(X^i_0:i\in\N)$ are i.i.d., neither exchangeable, nor we assume specific integrability of the initial condition $\mu_0$ of the Fokker-Planck equation. However, something is implicitly assumed on the initial conditions in order for the assumptions to hold, in particular Assumption \ref{ass: existence}-(1)-(2) and \ref{ass: convergence_ass}-(3); 
\item We do not assume strong existence neither of the interacting particle system nor of the McKean-Vlasov SDE  \eqref{eq: mckean_vlasov}, and we do not assume uniqueness of solutions neither of \eqref{eq: interacting_particle_system} nor of ~\eqref{eq: fokker--planck};
\item Boundedness of the flat derivative $\tfrac{\delta b}{\delta m}$ gives Lipschitz continuity of $b$ for the total variation norm; such condition is implied by (and thus weaker than) Lipschitz continuity for the  $W_1$ distance, since we are on the torus; 
\item We consider a drift $b$ depending on $t$ in order to consider also controlled dynamics arising from the theory of mean field games and mean field control problems.
\item We do not impose explicit assumptions so that Propagation of Chaos holds; instead, we impose condition (3) above, which is a form of Propagation of Chaos and fits to many models. 
\end{itemize}
\end{remark} 

The Assumptions~\ref{ass: existence},\ref{ass: coef_fokker} may seem technical. Hence, let us reformulate them in the liner case 
\[
b(t,x,m) = K*m(x)
\] for some interaction kernel \(K\).

\begin{lemma} 
\label{rem: convolution_estimate}
 In the linear case \(b(t,x,m) = K*m\), Assumptions~\ref{ass: existence}-~\ref{ass: coef_fokker} are verified if the following conditions hold: 
     \begin{itemize}
     \item[(i)] There exists a probabilistic weak solution of the interacting particle system~\eqref{eq: interacting_particle_system}. 
         \item[(ii)] \(K \in L^1(\T^d) \). 
         \item[(iii)]  \(\mu \in C([0,T];H^{\lambda'}(\T^d))\) solves the Fokker--Planck equation~\eqref{eq: fokker--planck}.
     \end{itemize}
Moreover, Assumption \ref{ass: convergence_ass} holds if 
\begin{itemize}
\item[(iv)] $K\in L^{\infty}(\T^d)$, 
\item[(v)] the initial condition is exchangeable and $\sup_{N \in \N} \mathcal{H}\big(\bar\mu^N_0 |\mu_0^{\otimes N}\big)$ is finite.  
\end{itemize}
These conditions also imply (i) and (ii) above
\end{lemma}
Clearly the latter condition is satisfied if $(X^1_0,\dots, X^N_0)$ is i.i.d. with law $\mu_0$. 
See Assumption~(A4) in~\cite{Xianliang2023} for a comparison in the case of the Vortex model with \(K\) given by the Biot--Savart law. As a consequence of the computations below, we recover the structure of the SPDE provided by~\cite[Equation~(1.5)]{Xianliang2023}. 
 
\begin{proof}
 We have
    \(\frac{\delta (K*\mu_t)}{\delta m}(x,v) = K(x-v)\),
    which implies
    \begin{equation*}
        \bigg\langle f,  \bigg \langle     \mu_t(x)  \frac{\delta b}{\delta m} (t,x,\mu_t , \cdot ) ,D \varphi(x) \bigg \rangle\bigg \rangle
        =  \langle f, \tilde K*(    \mu_t D \varphi) \rangle,
     \end{equation*}
    where we set \(\tilde K(x) = K(-x)\). The expression corresponds exactly to the Schwarz distribution \(\nabla \cdot (\mu_t K*f)\), which offers valuable insight into inequality~\eqref{eq: deriv_prob_inequ}. In our case it is a priori not possible to connect the integrating variable \(x\) and the variable \(v\) corresponding to the action of the Schwartz distribution.  
    But with the convolution structure we can absorbs much of the regularity in the Fokker--Planck equation \(\mu_t\). We obtain  
    %
    \begin{equation} \label{eq: linear_case_1}
           \norm{\nabla \cdot (\mu_t K*f) }_{ H^{-\lambda-2}(\T^d)}^2
           \le C\norm{\mu_t K*f }_{ H^{-\lambda-1}(\T^d)}^2
           \le C\norm{\mu_t}_{H^{\lambda'}(\T^d)}^2 \norm{f}_{H^{-\lambda-1}(\T^d)}^2
           \norm{K}_{L^1(\T^d)}^2
    \end{equation} 
    by an application of Lemma~\ref{lemma: product_distr} and Lemma~\ref{lemma: young}. This proves inequality~\eqref{eq: deriv_prob_inequ}. Similar computation demonstrate 
     \begin{equation} \label{eq: linear_case_2}
          \norm{\nabla \cdot (f K* \mu_t ) }_{ H^{-\lambda-2}(\T^d)}^2 \le 
          C\norm{f K* \mu_t }_{ H^{-\lambda-1}(\T^d)}^2
          \le C\norm{\mu_t}_{H^{\lambda'}(\T^d)}^2 \norm{f}_{H^{-\lambda-1}(\T^d)}^2
           \norm{K}_{L^1(\T^d)}^2,
     \end{equation}
     which implies inequality~\eqref{eq: b_ineq}. 
     For Assumption~\ref{ass: coef_fokker} we notice 
    \begin{equation*}
        \norm{b(t,\cdot,\mu_t)- b(s,\cdot,\mu_s)}_{H^{\lambda'}(\T^d)}
        = \norm{K*(\mu_t-\mu_s)}_{H^{\lambda'}(\T^d)}
        \le \norm{K}_{L^1(\T^d)} \norm{\mu_t-\mu_s}_{H^{\lambda'}(\T^d)} 
    \end{equation*}
    and 
    \begin{align*}
        &\norm{j_n*  \bigg\langle \nabla_x \cdot  \bigg(  \mu_r(\cdot )  \frac{\delta b}{\delta m} (r,\cdot ,\mu_r , v )- \mu_s(\cdot)  \frac{\delta b}{\delta m} (s,\cdot,\mu_s , v ) \bigg), j_n* f(v)  \bigg\rangle  \bigg\rangle}_{H^{-\lambda-2}(\T^d)} \\
        &\quad \le \norm{j_n*(D (\mu_t-\mu_s) \cdot K*j_n*f) }_{H^{-\lambda-2}(\T^d)}
       \le C \norm{ (\mu_t-\mu_s)K*j_n*f }_{H^{-\lambda-1}(\T^d)}\\
        &\quad \le C(n) \norm{K}_{L^1(\T^d)} \norm{\mu_t-\mu_s}_{H^{\lambda'}(\T^d)}\norm{f}_{H^{-\lambda-2}(\T^d)}. 
    \end{align*}
    Therefore, it is enough to require \(\mu \in C([0,T];H^{\lambda'}(\T^d))\). This also implies 
    \begin{equation*}
        \lim\limits_{s \to t } \norm{\sqrt{\mu_t}-\sqrt{\mu_s}}_{L^2(\T^d)}^2 = 0
    \end{equation*}
     by the continuity of \((t,x) \mapsto \mu_t(x)\), which is implied by the Sobolev embedding and \(\lambda' > d/2\).
     Therefore the first claim is proved. 

Regarding Assumptions~\ref{ass: convergence_ass}, we note that the first conditions is equivalent to \(K \in L^\infty(\T^d)\),
while, for bounded kernels, a bound on the relative entropy is provided by~\cite[Corollary~2.6]{Lacker2023} under the assumption on the initial condition in the statement; that result also provides existence and uniqueness of an exchangeable weak solution to the particle system.  
\end{proof}


\subsection{Gaussian noise} 
\label{subsec: Gaussian_noise}

The aim of this section is to construct the Gaussian noise \(\xi\) appearing in SPDE~\eqref{eq: limiting_spde} with the needed covariance structure
and to demonstrate that it coincides with the martingale formulation used in~\cite{Fernandez1997,Delarue2019,Xianliang2023}. Sometimes it is called white noise since it is valued in a space of distributions. 
In light of Assumption \ref{ass: existence}, there exist a probability space where a weak solution to the particle system is defined. Throughout the paper, we fix an enlargement of that space, that is, a filtered probability space $(\Omega, (\mathcal{F}_t)_{0\leq t\leq T}, \P)$ where all the processes are defined and the following noise exists.

In order to construct the Gaussian noise, we require cylindrical Brownian motions. Let us recall the construction provided by Liu, Röckner~\cite{LiuWei2015SPDE}. 
For two Hilbert spaces \(U,H\) denote by \(L_2(U,H)\) the Hilbert--Schmidt operators from \(U \) to \(H\).  
For \(j=1,\ldots,d\) let \(W_j\) be a Gaussian noise on the Hilbert space \(L^2(\T^d)\). More precisely, let \((U_1, \langle \cdot, \cdot \rangle ) \) be another Hilbert space and \(J \colon L^2(\T^d) \mapsto U_1\) be a Hilbert--Schmidt embedding. For the existence of such embeddings we refer to~\cite[Chapter~2.5.1]{LiuWei2015SPDE}. Then, we define 
\begin{equation}
    W_j(t) : = \sum\limits_{n=1}^\infty \beta_{n,j}(t) J(e_n),
\end{equation}
where \((\beta_{n,j}, j=1,\ldots, d, n  \in \N)\) are a family of independent Brownian motions. 
This series converges in the space of square integrable martingale and defines a \(Q_1 := JJ^*\) Wiener process on \(U_1\)~\cite[Chapter 2.5]{LiuWei2015SPDE}.
Later the Hilbert space \(U_1\) will not play a role. 

For \(j=1,\ldots,d\) we introduce the operator \(B_j\colon [0,T] \mapsto L_2(L^2(\T^d), H^{-\lambda-2}(\T^d))\) given by 
\begin{equation}
    B_j(t)(u) =  \partial_{x_j} ( \sigma u \sqrt{\mu_t} ) , 
 \end{equation}
where the right hand side is to be understood in the sense that for smooth \(\varphi \) we have 
\begin{equation*}
    \langle  \partial_{x_j} ( \sigma u \sqrt{\mu_t} ) , \varphi \rangle  
    = - \langle \sigma u \sqrt{\mu_t}  ,  \partial_{x_j} \varphi \rangle _{L^2(\T^d)}.  
\end{equation*}
We want to show that \(B_j\) is well-defined. 

\begin{lemma}
Assuming $||\mu_t||_{L^1(\T^d)}=1$ for any $t$, we have that $B_j(t)$ is Hilbert-Schmidt, its adjoint operator is given by \(B_j(t)^* \colon H^{-\lambda-2}(\T^d)) \mapsto L^2(\T^d) \), 
\begin{equation}
    B_j(t)^*(f) :=  - \sqrt{\mu_t} e_k  \sum\limits_{k \in \Z^d} \langle    k \rangle^{-2(\lambda+2)}  \overline{ \langle \sigma \partial_{x_j} f, e_k \rangle}.
\end{equation}
and 
\begin{equation}
\label{eq:est_B}
\|B_j(t) \|_{L_2(L^2(\T^d), H^{-\lambda-2}(\T^d))} \leq C.
\end{equation}
\end{lemma} 
\begin{proof}
To get the formula for the adjoint operator, 
let \(u \in L^2(\T^d)\) and \(f \in H^{-\lambda-2}(\T^d)\). Then,  
\begin{align*}
    \langle B_j(t) u, f \rangle_{ H^{-\lambda-2}(\T^d)}
&=- \sum\limits_{k \in \Z^d} \langle k \rangle^{-2(\lambda+2)} \sigma \langle  u \sqrt{\mu_t }, \partial_{x_j} e_k \rangle_{L^2(\T^d)} \overline{ \langle f, e_k \rangle } \\
&=  - \sum\limits_{k \in \Z^d} \langle k \rangle^{-2(\lambda+2)} 2\pi k_j i \sigma \langle u  \sqrt{\mu_t },  e_k \rangle_{L^2(\T^d)} \overline{ \langle f, e_k \rangle} \\
&= - \sum\limits_{k \in \Z^d} \langle k \rangle^{-2(\lambda+2)}  \sigma \langle  u \sqrt{\mu_t },  e_k \rangle_{L^2(\T^d)} \overline{\langle \partial_{x_j} f, e_k \rangle} \\
&= - \bigg\langle    u,  \sqrt{\mu_t} e_k  \sum\limits_{k \in \Z^d} \langle    k \rangle^{-2(\lambda+2)} \overline{  \langle \sigma \partial_{x_j} f, e_k \rangle}   \bigg\rangle_{L^2(\T^d)} .
\end{align*}
We estimate
\begin{align*}
     \norm{\sqrt{\mu_t}  e_k  \sum\limits_{k \in \Z^d}  \langle  k \rangle^{-2(\lambda+2)}   \langle \sigma \partial_{x_j} f, e_k \rangle    }_{L^2(\T^d)}^2
     &\le \sigma^2  \sum\limits_{k \in \Z^d}  \langle k \rangle^{-4(\lambda+2)}   |\langle \partial_{x_j} f, e_k \rangle|^2 |\langle \sqrt{\mu_t},e_k\rangle|^2 \\
     & \le C   \sum\limits_{k \in \Z^d}  \langle k \rangle^{-4\lambda-6}   |\langle  f, e_k \rangle|^2    |\langle \sqrt{\mu_t},e_k\rangle|^2 \\
     & \le C  \sum\limits_{k \in \Z^d}  \langle k \rangle^{-2(2\lambda+3)}   |\langle  f, e_k \rangle|^2  \sum\limits_{k \in \Z^d} |\langle  \sqrt{\mu_t},e_k\rangle|^2 \\
     &\le C  \norm{f}_{H^{-(2\lambda+3)}(\T^d)}^2 \norm{\sqrt{\mu_t}}_{L^2(\T^d)}^2 \\
     &\le C \norm{f}_{H^{-\lambda-2}(\T^d)}^2 \norm{\sqrt{\mu_t}}_{L^2(\T^d)}^2
\end{align*}
We show that the adjoint operator is Hilbert--Schmidt. 
Let \((f_l, l \in \N)\) be an orthonormal basis of \(H^{-\lambda-2}(\T^d)\). Repeating the above computation with \(f_l\) instead of \(f\), we find 
\begin{equation} \label{eq: hilbert_schmidt_b}
     \sum\limits_{l=1}^\infty \norm{\sqrt{\mu_t}  e_k  \sum\limits_{k \in \Z^d}  \langle  k \rangle^{-2(\lambda+2)}   \langle \sigma \partial_{x_j} f_l , e_k \rangle }_{L^2(\T^d)}^2
     \le C \norm{\sqrt{\mu_t}}_{L^2(\T^d)}^2 \sum\limits_{l=1}^\infty \norm{f_l}_{H^{-(2\lambda+3)}(\T^d)}^2
\end{equation}
But the embedding \(H^{-\lambda-2}(\T^d) \hookrightarrow H^{-(2\lambda+3)}(\T^d)\) is Hilbert--Schmidt because \(-\lambda-2+2\lambda+3  > d/2\); see Lemma \ref{lem:HS_embedding}.  
Since the adjoint operator of a Hilbert--Schmidt operator is Hilbert--Schmidt~\cite[Remark~B.06 (i)]{LiuWei2015SPDE}, it follows that \(B_j\) is well-defined and \eqref{eq:est_B} holds.
\end{proof} 

We define the stochastic integral  
\begin{equation*}
\int\limits_0^t B_j(s) \Id W_j(s) := \int\limits_0^t B_j(s) \circ J^{-1} \Id W_j(s),
\end{equation*}
where on the right hand side it is the classical stochastic integral with respect to a \(Q_1\)-Wiener process, where \(Q_1\) is of trace class. Since \(B_j\) is a deterministic operator, the stochastic integral is Gaussian. This follows by the standard approximation argument with respect to elementary functions. 
Let us define the Gaussian process \(\zeta\) with values in \( H^{-\lambda-2}(\T^d)\) by the sum of the integrals, which themselves are Gaussian, i.e. 
\begin{equation}
    \zeta(t) := \sum\limits_{j=1}^d  \int\limits_0^t B_j(s) \Id W_j(s) . 
\end{equation}

Let us demonstrate that the constructed stochastic integral coincides with the martingale term given by~\cite{Fernandez1997,Delarue2019,Xianliang2023}. 

\begin{lemma}
Assuming $||\mu_t||_{L^1(\T^d)}=1$ for any $t$, we have the following formula for the correlation: 
\begin{equation*}
\E[\zeta(t)(\varphi_1)\zeta(s)(\varphi_2)]
= \int\limits_0^{\min(s,t)} \langle \mu_r  ,  \sigma^2 D  \varphi_1  \cdot D \varphi_2 \rangle_{L^2(\T^d)} \Id r  
\end{equation*}
for any $0\leq t,s\leq T$ and \(\varphi_1, \varphi_2 \in C^\infty(\T^d)\). 
\end{lemma} 
\begin{proof}
Define the linear functional on \( H^{-\lambda-2}(\T^d)\) by
$L_{i}(f) : = \langle f, \varphi_i \rangle $
for \(i=1,2\).
Applying~\cite[Lemma 2.4.1]{LiuWei2015SPDE} we obtain  
\begin{align*}
    L_{i}\bigg( \int\limits_0^t B_j(s) \Id W_j(s) \bigg) 
    =  \int\limits_0^t L_{i} \circ  B_j(s) \circ J^{-1} \Id W_j(s)
\end{align*}
and the process is real-valued. Consequently, we can compute 
\begin{align*}
    &\E \bigg[  L_{1}\bigg( \int\limits_0^t B_j(r) \Id W_j(r) \bigg)  L_{2}\bigg( \int\limits_0^s B_j(r) \Id W_j(r) \bigg) \bigg]  \\
    &\quad = \E \bigg[ \int\limits_0^t L_{1} \circ  B_j(r) \circ J^{-1} \Id W_j(r) \int\limits_0^s L_{2} \circ  B_j(r) \circ J^{-1} \Id W_j(r) \bigg]  \\
    &\quad = \int\limits_0^{\min(s,t)}  \big( L_{2} \circ  B_j(r) \circ J^{-1} \circ Q_1^{1/2} \big)  \circ \big( L_{1} \circ  B_j(r) \circ J^{-1} \circ  Q_1^{1/2} \big)^{*}  \Id r
\end{align*}
where the last equality follows by~\cite[Proposition 4.28]{Prato1992}. 
Recall, that \(Q_1\) is a non-negative, symmetric operator and, therefore, the square root \(Q_1^{1/2}\) is also symmetric, i.e. \(Q_1^{1/2} = (Q_1^{1/2})^*\). Hence, utilizing \((T^*)^{-1} = (T^{-1})^*\) for arbitrary linear operator \(T\), we obtain 
\begin{align*}
    &\big( L_{2} \circ  B_j(r) \circ J^{-1} \circ Q_1^{1/2} \big)  \circ \big( L_{1} \circ  B_j(r) \circ J^{-1} \circ  Q_1^{1/2} \big)^{*} \\
    &\quad = \big( L_{2} \circ  B_j(r) \circ J^{-1} \circ Q_1 \circ (J^{-1})^* \circ \big(L_{1} \circ  B_j(r)\big)^*  \\
    &\quad = \big( L_{2} \circ  B_j(r) \circ J^* \circ (J^{-1})^* \circ \big(L_{1} \circ  B_j(r)\big)^*  \\
     &\quad = \big( L_{2} \circ  B_j(r) \circ \big(L_{1} \circ  B_j(r)\big)^*  .
\end{align*}
Let us compute the adjoint operator. Let \(a \in \R\), \(u \in L^2(\T^d)\) , then 
\begin{align*}
    \langle \big(L_{1} \circ  B_j(r)\big)^*(a), u \rangle_{L^2(\T^d)}
    &= \langle a,  \big(L_{1} \circ  B_j(r)\big) (u )\rangle_{\R} \\
    &= \langle a, \langle \partial_{x_j} ( \sigma u \sqrt{\mu_t} ), \varphi_1 \rangle \rangle_{\R} \\
     &= -  \langle    u ,   a \sigma \sqrt{\mu_t} \partial_{x_j}  \varphi_1 \rangle_{L^2(\T^d)} ,
\end{align*}
which implies 
\begin{align*}
    \big( L_{2} \circ  B(r)  \big) \circ \big(L_{1} \circ  B(r)\big)^*(a)
    &=  -a L_{2} \circ B(r) \circ \big( \sigma \sqrt{\mu_t} \partial_{x_j} 
 \varphi_1 \big)\\
    &= -a L_{2}  \bigg(\partial_{x_j}  \big( \sigma^2 \mu_t\partial_{x_j} \varphi_1 \big) \bigg)\\
     &= a \langle  \sigma^2 \mu_t \partial_{x_j}  \varphi_1 , \partial_{x_j} 
 \varphi_2 \rangle_{L^2(\T^d)} . 
\end{align*}
Consequently, we find 
\begin{equation*}
    \E \bigg[  L_{1}\bigg( \int\limits_0^t B_j(r) \Id W_j(r) \bigg)  L_{2}\bigg( \int\limits_0^s B_j(r) \Id W_j(r) \bigg) \bigg] 
    = \int\limits_0^{\min(s,t)} \langle \mu_r  ,  \sigma^2 \partial_{x_j}  \varphi_1 \partial_{x_j} \varphi_2 \rangle_{L^2(\T^d)} \Id r . 
\end{equation*}
By the independence of \((W_j, j =1,\ldots,d)\), the claim follows by summing the above equality.
\end{proof}

\subsection{Solution to SPDE}


The goal of this section is to give a definition of solution to the SPDE~\eqref{eq: limiting_spde} and state its well-posedness. We treat it as an infinite dimensional SDE in the triple 
\[(H^{-\lambda-1}(\T^d), H^{-\lambda-2}(\T^d), H^{-\lambda-3}(\T^d)).\]  
    Notice that the sequence \(H^{-\lambda-1}(\T^d)\subset H^{-\lambda-2}(\T^d)\subset H^{-\lambda-3}(\T^d)\) is a normal triple, in the sense that 
    \begin{equation*}
        \langle f_1,f_2 \rangle_{H^{-\lambda-2}(\T^d) }
        \le \norm{f}_{H^{-\lambda-1}(\T^d) } \norm{f_2}_{H^{-\lambda-3}(\T^d) }
    \end{equation*}
    for all \(f_1 \in H^{-\lambda-1}(\T^d)\) and \(f_2 \in H^{-\lambda-2}(\T^d)\). 
    Hence, by~\cite[Section~2.5.2]{Rozovsky2018} there exists a canonical bilinear functional (CBF) given by 
    \begin{equation*}
        [\cdot,\cdot]_{-\lambda-2} \colon  H^{-\lambda-1}(\T^d) \times  H^{-\lambda-3}(\T^d) \to \R 
    \end{equation*}
    such that 
    \begin{equation}
    \label{eq: cbf_relation}
    [f_1,f_2]_{-\lambda-2} =\langle f_1,f_2 \rangle_{H^{-\lambda-2}(\T^d) }
    \end{equation}
    for \(f_1 \in H^{-\lambda-1}(\T^d)\) and \(f_2 \in H^{-\lambda-2}(\T^d)\).

Define the linear operator \(A \colon [0,T] \times   H^{-\lambda-1}(\T^d) \mapsto  H^{-\lambda-3}(\T^d)  \)
by 
\begin{equation} \label{eq: def_A}
    A(t,f) : =  -  \nabla \cdot (f b(t,\cdot,\mu_t) ) - \bigg\langle \nabla_x \cdot  \bigg(  \mu_t(x)  \frac{\delta b}{\delta m} (t,x,\mu_t , v ) \bigg), f(v)  \bigg\rangle  +  \frac{\sigma^2}{2}  \Delta  f  .
\end{equation}
As previously, the right hand side is defined by 
\begin{align*}
    \langle f, A'(t)(\varphi) \rangle , 
\end{align*}
where $A'(t,b,\mu_t,\sigma): H^{\lambda+3}(\T^d) \rightarrow H^{\lambda+1}(\T^d)$ is given by 
\begin{equation} \label{eq: definition_Aprime}
    A'(t)(\varphi):= b(t,\cdot,\mu_t) \cdot D \varphi(\cdot)+ \bigg\langle     \mu_t(x)  \frac{\delta b}{\delta m} (t,x,\mu_t , \cdot )  , D \varphi(x) \bigg \rangle  +   \frac{\sigma^2}{2} \Delta \varphi (\cdot),
\end{equation}
for \(\varphi \in C^\infty (\T^d)\). By the Assumption~\ref{ass: existence} the operator is well defined. Notice a similar results is stated for weighted fractional Sobolev spaces in~\cite[Inequality(5.12)]{Delarue2019} and~\cite{Jourdain1998}, which is partially proven in~\cite[p.~77,~Lemma~5.6]{Meleard1996}. 

The filtration $(\mathcal{F}_t,  0\leq t \leq T )$ and the noise are fixed in \S \ref{subsec: Gaussian_noise} above, thus the following is a  probabilistically strong definition of solution. 
    
\begin{definition} \label{def: spde_solution}
    Let \( \rho = (\rho_t, 0\leq t \leq T )\) be an $\mathcal{F}$-adapted process on \(H^{-\lambda-2}(\T^d)\). We say \(\rho\) is a solution to equation~\eqref{eq: limiting_spde} if $\rho\in S^2_{\cF}([0,T]; H^{-\lambda-2}(\T^d))$ and there exists a set \(\tilde \Omega \) of full measure such that  for all \(\omega \in \tilde \Omega \) the map \(t \mapsto \rho(t,\omega)\) is continuous with values in \(H^{-\lambda-2}(\T^d)\) 
     and for each \(\varphi \in C^\infty(\T^d)\) and \(t \in [0,T]\) it holds
    \begin{equation}
        \langle \rho_t,\varphi \rangle_{H^{-\lambda-2}(\T^d) }  = \langle \rho_0,\varphi \rangle_{H^{-\lambda-2}(\T^d) }   + \int\limits_0^t \langle  \rho_s, A'(s) (\varphi) \rangle_{H^{-\lambda-2}(\T^d) }  \Id s + \langle \zeta_t, \varphi \rangle_{H^{-\lambda-2}(\T^d) }  , \quad \P\text{-a.s.} 
    \end{equation}
\end{definition}

We choose to write the equation with respect to scalar product in $H^{-\lambda-2}(\T^d)$, instead of duality with test functions, mainly to use several results which are stated for SPDEs on Hilbert spaces in normal triplets. We show the following well-posedness result:

\begin{theorem}\label{theorem: existence_SPDE}
Under Assumptions~\ref{ass: inital_cond}and~\ref{ass: existence}, the equation~\eqref{eq: limiting_spde} admits a unique solution \(\rho\) in the sense of Definition~\ref{def: spde_solution}. Moreover, $\rho\in L^2_{\mathcal{F}}([0,T]; H^{-\lambda-1}(\T^d))$ and we have the bound
    \begin{equation}
    \label{eq:esti_spde}
        \E\Big[\sup\limits_{0 \le t \le T } \norm{\rho_t}^2_{H^{-\lambda-2}(\T^d)}\Big] + \norm{\rho}^2_{L^2_{\mathcal{F}}([0,T]; H^{-\lambda-1}(\T^d)) } \le C \E \big[ \norm{\rho_0}_{H^{-\lambda-2}(\T^d)}^2 \big].
    \end{equation}
\end{theorem}

The proof is given in \S \ref{sec: well-posedness_SPDE}. 
However, the main property we require is not the mere existence and uniqueness of a strong solution, which could likely be obtained directly from the results in~\cite{Rozovsky2018} by verifying the properties of the operator \(A\) and extending the results to cover the white noise case. Instead, we construct an approximation of the SPDE whose solution converges in probability; see the main result Theorem \ref{thm: conv_prob}. In the classical framework, only weak convergence of finite-dimensional projections in the sense of functional analysis is typically available. The underlying reason is that, for linear equations, weak convergence suffices and stronger forms of convergence are not required. However, in order to establish the weak error in Theorem \ref{theorem: main_result} below, we aim to replace the limiting process \(\rho\) with its approximation $\rho^n$,  
since the test function \(\Phi\) is not linear. Moreover, convergence in probability allows to show existence of (probabilistically) strong solutions.

On the other hand, as explained in the Introduction,  we approximate with solutions to a mollified equation which can be written as  a SDE in the Hilbert space $H^{-\lambda-2}(\T^d)$. This in particular allows to apply It\^o formula for functions $\Phi(\rho^n_t)$, which is is crucial to compute the generator of the process in order to compare with the generator of the fluctuation process and prove the main result below. We also remark that, differently from all other results on fluctuations, which are based on tightness arguments, the initial condition lies in the same space of the noise; this is crucial to analyze the semigroup related to the equation and prove stability estimates in Section \ref{sec:generators}.



\subsection{Main result}
We can now state the main estimate of the paper. Recall that $\rho^N$ is the fluctuation process in \eqref{eq: def_fluctuation_process} and $\rho$ is the solution to SPDE \eqref{eq: limiting_spde} given by Theorem \ref{theorem: existence_SPDE}.
\begin{theorem}[Main result] \label{theorem: main_result}
Let the Assumptions~\ref{ass: inital_cond}--\ref{ass: convergence_ass} be satisfied. Then for any \(\Phi \in C_\ell^2(H^{-\lambda-2}(\T^d))\) we obtain 
\begin{equation}
\label{eq:main_estimate}
    \sup\limits_{0 \le t \le T} |\E[\Phi(\rho_t^N)] - \E[\Phi(\rho_t)]| \le C [\Phi]_{C^2(H^{-\lambda-2}(\T^d))}  \bigg( \frac{1}{\sqrt N } +   
      W_{1, H^{-\lambda-2}(\T^d)}\Big( \P_{\rho^N_0}, \P_{\rho_0}\Big) 
    \bigg), 
\end{equation}
where $C$ only depends on $d, \lambda, T$.    
\end{theorem}

The proof is given in Section \ref{sec: comparison}. As explained above, it is based on estimating the semigroups by using the generators of the processes which are computed in Section \ref{sec:generators}. It requires to use the approximation of the limiting SPDE and then to pass to the limit in probability, as $\Phi$ is nonlinear. Note that $\rho^n$ is Markovian, while $\rho^N$ and $\rho$ might not be. The reminder is estimated by using relative entropy methods and stability estimates on the limiting semigroup. We also need to approximate $\Phi$ with cylindrical functions. 

The main Theorem is written for $\Phi$ of linear growth because it includes the case of linear functionals, but notice that the convergence rate depends just on the supnorm of the derivatives of $\Phi$. We choose to consider the processes to be defined on the same probability space because we show (probabilistically) strong well-posedness to SPDE \eqref{eq: limiting_spde}, but we remark that the main result involves just the marginal distribution of the processes. 

In case of linear interaction, the main result, thanks to Lemma \ref{rem: convolution_estimate}  implies the following: 

\begin{corollary}
If $b(t,x,m) = K * m(x)$ and (iii)-(iv)-(v) of Lemma \ref{rem: convolution_estimate} hold, then  \eqref{eq:main_estimate} is satisfied. 
\end{corollary}

We finally show how Theorem \ref{theorem: main_result} may imply convergence in distribution of $\rho^N$ to $\rho$ on the path space $C([0,T], H^{-\lambda-2}(\T^d))$, as proved in ~\cite{Fernandez1997,Delarue2019,Xianliang2023}, among others. Since convergence in distribution on the space of continuous functions is equivalent to convergence of finite dimensional distributions plus tightness \cite[Lemma 16.2]{Kallenberg2002}, 
we assume tightness in the following result.   

\begin{corollary}
\label{cor:conv_processes}
Let the Assumptions~\ref{ass: inital_cond}--\ref{ass: convergence_ass} be satisfied, the sequence $\rho^N$ be tight on $C([0,T], H^{-\lambda-2}(\T^d))$, $(\mu^N_t, 0\leq t\leq T)$ be a Markov process on $\mathcal{P}(\T^d)$ and 
 $\lim\limits_N \rho^N_0 =\rho_0$ in law on $H^{-\lambda-2}(\T^d)$. 
Then $\lim\limits_N \rho^N =\rho$ in law on $C([0,T],H^{-\lambda-2}(\T^d))$.
\end{corollary}

The proof is also given in Section \ref{sec: comparison}.
Note that we can not apply \cite[Thm 19.25]{Kallenberg2002}, which also shows that strong convergence of the semigroups implies tightness, since the semigroups are not strongly continuous and further the space  $H^{-\lambda-2}(\T^d)$ is not locally compact and we do not characterize an invariant  core of the limiting semigroup. Nevertheless, we manage to prove convergence of finite-dimensional distributions, without making use of the martingale problem formulation nor of weak convergence arguments. Let us recall that a sufficient condition for $\mu^N$ to be Markovian is that the initial conditions $(X^1_0, \dots,X^N_0)$ are exchangeable  and weak uniqueness holds for the SDE \eqref{eq: interacting_particle_system}, so that the process $(X^1,\dots, X^N)$ is Markovian and thus Markovianity of $\mu^N$ follows by \cite[Prop. 2.3.3]{Dawson1993}.

\section{Well-posedness and approximation of SPDE} \label{sec: SPDE_existence}

\noindent \emph{Throughout this section Assumptions~\ref{ass: inital_cond},~\ref{ass: existence} are in force.}

\smallskip

The aim of this section is to construct an approximation of the SPDE \eqref{eq: limiting_spde} which writes as an SDE in the Hilbert space ${H}^{-\lambda-2}(\T^d)$, whose solution converges in probability to the solution to \eqref{eq: limiting_spde}. The main convergence result is Theorem \ref{thm: conv_prob}, which is proved in \S \ref{sec: well-posedness_SPDE} together with Theorem~\ref{theorem: existence_SPDE} on existence and uniqueness of strong solutions to \eqref{eq: limiting_spde}.

On a technical side, we remark that the approximating equation must be
formulated in a negative Sobolev space; in aprticular, the Laplacian reduces regularity by two derivatives, and any analysis must account for this loss. Nevertheless, it is intuitively clear that the dissipative nature of the Laplacian should help to gain regularity, rather than lose it.
Indeed, the Laplacian's dissipative structure yields an approximation sequence that converges not only in the weak topology of the Hilbert space, but also in the stronger topology of $L^2( [0,T], H^{-\lambda-2}(\T^d))$.

\subsection{Approximating SPDE} \label{subsec: mollification}
Let us introduce the following mollification. 
Let \(\tilde j\) be a smooth symmetric function on \(\T^d\) with compact support in the unit ball and \(\tilde j(0)=1\). 
Let us define the mollifier 
\begin{equation*}
    j_n(x) := \sum\limits_{k \in \Z^d} e^{2 \pi i x \cdot k} \tilde j\big( n^{-1}k \big). 
\end{equation*}
and the mollification operator. 
We have the following result on approximation of elements in \(H^{s}(\T^d)\)  via mollification. 
The proof follows by explicit computation.

\begin{lemma} 
\label{lemma: mollification}
    Let \(s \in \R \) and \(f,g \in H^{s}(\T^d)\).  The mollification operator \(j_n * f \in H^{\tilde s}(\T^d)\) for all \(\tilde s \in \R \). Moreover, 
    \begin{equation} 
    \label{eq: mollification_convergence}
    \begin{split}
    j_n * e_k &= e_k \tilde j(n^{-1}k), \\
    \langle j_n*f,e_k \rangle &= \langle f,e_k \rangle \tilde j(n^{-1}k),  \\
    \langle  j_n*f, g \rangle_{H^{s}(\T^d)} &=  \langle f, j_n*g \rangle_{H^{s}(\T^d)} , \\
        \lim\limits_{n \to \infty}\norm{j_n*f -f }_{H^{s}(\T^d)} &= 0, \\
        \norm{j_n*f}_{H^{s}(\T^d)}
        &\le C \norm{f}_{H^{s}(\T^d)}, \\
        \norm{j_n*f}_{H^{\tilde s}(\T^d)}
        &\le  C(n) \norm{\varphi}_{H^{s}(\T^d)} . 
    \end{split}
    \end{equation}
\end{lemma}

We define the approximation \((A_n, n \in \N)\) of the operator \(A\) by 
\begin{equation} \label{eq: regularized_A}
   \begin{cases}
    A_n \colon [0,T] \times {H}^{-\lambda-1}(\T^d)   \mapsto {H}^{-\lambda-1}(\T^d)\\
    (t,f) \to \Big( \varphi \to  \langle f, j_n *(A'(t)(j_n* \varphi)) \rangle  \Big) 
   \end{cases}
   \end{equation}
We have the following properties on the operator $A_n$. For the latter coercivity condition, as explained above, we use the dissipation provided by the Laplacian. Here, $C(n)$ denotes a constant depending also on $n$, while $C$ does not depend on $n$.

\begin{lemma} 
\label{lemma2: dissapation}
    The operator \(A_n\) is linear on ${H}^{-\lambda-2}(\T^d)$ and, for any \(f \in {H}^{-\lambda-2}(\T^d)\) and  \(n  \ge 1 \), we have 
    \begin{equation} 
    \label{item: approx_bound_smaller_space}
    \norm{A_n(t,f)}_{H^{-\lambda-2}(\T^d)} \le C(n) \norm{f}_{{H}^{-\lambda-2}(\T^d)} . 
    \end{equation}
     Moreover, there exists $\delta>0$ small such that, for any \(f \in  H^{-\lambda-1}(\T^d) \),  
    \begin{equation}
    \label{lemma: dissapation}
    \sup\limits_{n \in \N}   \langle A_n(t,f), f \rangle_{H^{-\lambda-2}(\T^d)}  \le C \norm{f}_{{H}^{-\lambda-2}(\T^d)}^2 - \delta \norm{f}_{{H}^{-\lambda-1}(\T^d)}^2 .
    \end{equation}
\end{lemma}
\begin{proof}
The linearity of the operator \(A_n\) is an immediate consequence of the linearity of \(A\). By Lemma~\ref{lemma: mollification}, we find 
    \begin{align*}
        \norm{ A_n(t,f)}_{H^{-\lambda-2}(\T^d)}^2
        &= \sum\limits_{k \in \Z^d} \langle k \rangle^{-2(\lambda+2)} |\langle f, j_n * A'(t)(j_n* e_k) \rangle|^2 \\
        &=  \sum\limits_{k \in \Z^d} \langle k \rangle^{-2(\lambda+2)} |\langle A(t)( j_n* f),  (e_k) \rangle|^2 |\tilde j(n^{-1}k)|^2 \\
        &\le C \norm{A(j_n *f)}_{H^{-\lambda-2}(\T^d)}^2 \\
        &\le C \norm{j_n *f}_{H^{-\lambda}(\T^d)}^2 \\
        &\le C(n) \norm{f}_{H^{-\lambda-2}(\T^d)}^2, 
    \end{align*}
which proves~\eqref{item: approx_bound_smaller_space}. To prove \eqref{lemma: dissapation}, we compute    
    \begin{align*}
        \langle A_n(t,f), f \rangle_{H^{-\lambda-2}(\T^d)} 
        &= \sum\limits_{k \in \Z^d} \langle k \rangle^{-2(\lambda+2)} 
        \langle f, j_n * A'(t)(j_n* e_k) \rangle  \overline{ \langle f,  e_k\rangle} \\
        &= \sum\limits_{k \in \Z^d} \langle k \rangle^{-2(\lambda+2)} 
        \langle j_n * f,  A'(t)( e_k) \rangle  \overline{\langle j_n*f,  e_k\rangle} .
    \end{align*}
Now, \(A'(t)\) is linear. Let us recall the definition of \(A'(t)\) and compute each term individually
\begin{equation*}
     \langle j_n * f,  A'(t)( e_k) \rangle := \bigg\langle j_n * f,  b(t,\cdot,\mu_t) \cdot D e_k (\cdot)+ \bigg\langle     \mu_t(x)  \frac{\delta b}{\delta m} (t,x,\mu_t , \cdot )  , D e_k(x) \bigg \rangle  +   \frac{\sigma^2}{2} \Delta e_k (\cdot) \bigg \rangle ,
\end{equation*}
For the diffusion term, we find 
\begin{align*}
    \frac{\sigma^2}{2}\sum\limits_{k \in \Z^d} \langle k \rangle^{-2(\lambda+2)} 
        \langle j_n * f, \Delta e_k \rangle  \overline{ \langle j_n*f,  e_k\rangle}
    &= -  \frac{\sigma^2}{2} \sum\limits_{k \in \Z^d} \langle k \rangle^{-2(\lambda+2)} \sum\limits_{j=1}^d
        \langle j_n * f, \partial_{x_j} e_k \rangle \overline{ \langle j_n*f, \partial_{x_j}  e_k\rangle} \\
    &= - \frac{\sigma^2}{2}\sum\limits_{j=1}^d
        \norm{ \partial_{x_j} j_n * f}_{H^{-\lambda-2}(\T^d)}\\
    &\le -3\delta  \norm{j_n * f}_{H^{-\lambda-1}(\T^d)}+ C\norm{j_n * f}_{H^{-\lambda-2}(\T^d)}
\end{align*}
for some small \(\delta >0\). 
The last line follows by the equivalence of the norms 
\begin{equation*}
    \norm{j_n * f}_{H^{-\lambda-1}(\T^d)} \quad \mathrm{and} \quad \sum\limits_{j=1}^d  \norm{ \partial_{x_j} j_n * f}_{H^{-\lambda-2}(\T^d)}+  \norm{j_n * f}_{H^{-\lambda-2}(\T^d)}. 
\end{equation*}
For one of the first order terms we obtain
\begin{align*}
    &\bigg|\sum\limits_{k \in \Z^d} \langle k \rangle^{-2(\lambda+2)} 
        \langle j_n * f,b(t,x,\mu_t) \cdot \nabla e_k \rangle  \overline{\langle j_n*f,  e_k\rangle}\bigg| \\
    &\le \norm{f}_{H^{-\lambda-2}(\T^d)} \norm{\nabla \cdot ( b(t,x,\mu_t) j_n*f) }_{H^{-\lambda-2}(\T^d)} \\
    &\le  C \norm{f}_{H^{-\lambda-2}(\T^d)} \norm{f}_{H^{-\lambda-1}(\T^d)} \\
    &\le \delta \norm{f}_{H^{-\lambda-1}(\T^d)}^2 + \frac{C^2}{4\delta}  \norm{f}_{H^{-\lambda-2}(\T^d)}^2,
\end{align*}
where we used inequality~\eqref{eq: b_ineq}. 
For the second first order term we find 
\begin{align*}
&\bigg|\sum\limits_{k \in \Z^d} \langle k \rangle^{-2(\lambda+2)} 
        \bigg\langle f, j_n * \Big\langle    \mu_t(x)  \frac{\delta b}{\delta m} (t,x,\mu_t , \cdot ) \cdot   D e_k(x)  \Big\rangle \bigg\rangle  \overline{\langle j_n*f,  e_k \rangle}\bigg| \\
    &\le \norm{\bigg\langle \nabla_x \cdot  \bigg(  \mu_t(x)  \frac{\delta b}{\delta m} (t,x,\mu_t , v ) \bigg), f(v)  \bigg\rangle}_{ H^{-\lambda-2}(\T^d)}^2
    \norm{f}_{H^{-\lambda-2}(\T^d)}\\
    &\le \delta \norm{f}_{H^{-\lambda-1}(\T^d)}^2 + \frac{C^2}{4\delta}  \norm{f}_{H^{-\lambda-2}(\T^d)}^2,
\end{align*}
where we used inequality~\eqref{eq: deriv_prob_inequ}.
Combining all estimates proves \eqref{lemma: dissapation}. 
\end{proof}

We can now introduce the approximating SPDE 
\begin{equation} \label{eq: approx_solution}
    \rho_t^n = \rho^n_0 + \int\limits_0^t A_n(s, \rho_s^n) \Id s + \sum\limits_{j=1}^d \int\limits_0^t B_{j}(s)  \Id W_j(s) , \quad \P\text{-a.s.} 
\end{equation}
on \(H^{-\lambda-2}(\T^d)\) with initial data \(\rho_0 \in L^2_{\cF_0}( H^{-\lambda-2}(\T^d))\). We stress that it is written as a strong solution of an SDE in a Hilbert space.  
\begin{definition} \label{def: approximation_spde}
    A continuous \( H^{-\lambda-2}(\T^d)\)-valued \(\cF\)-adapted \emph{Markov} process \((\rho_t^n, 0 \leq t \leq T) \) is called a solution to~\eqref{eq: approx_solution} with initial value \(\rho_0^n \in L^2_{\cF_0}( H^{-\lambda-2}(\T^d))\) if 
    $\rho^n\in  S^2_{\cF}( [0,T];H^{-\lambda-2}(\T^d))$
    and equation~\eqref{eq: approx_solution} holds \(\P\text{-a.s.} \) in the space \({H}^{-\lambda-2}(\T^d)\). 
\end{definition}

We establish existence and uniqueness of the solution to the approximating SPDE, and also regularity.

\begin{proposition}[Existence of approximation SPDE~\eqref{eq: approx_solution}] \label{lemma: existence_approx_SPDE}
Let \(\rho_0^n \in L^2_{\cF_0}( H^{-\lambda-2}(\T^d))\), then there exists a unique solution of the SPDE~\eqref{eq: approx_solution} in the sense of Definition~\ref{def: approximation_spde}. Moreover \(\rho^n \in L^2_{\mathcal{F}}([0,T]; H^{-\lambda-1}(\T^d))\).  
\end{proposition}
\begin{proof}
    Thanks to the linearity of $A_n$ and \eqref{item: approx_bound_smaller_space} we may apply~\cite[Theorem~4.2.4]{LiuWei2015SPDE} (with $V=H=H^{-\lambda-2}(\T^d)$ therein) to find a solution in the sense of Definition~\ref{def: approximation_spde}. Notice that all assumptions on \(B_{j}\) in~\cite[Theorem~4.2.4]{LiuWei2015SPDE} are satisfied, since the noise is independent of the solution and \(B_j\) already satisfies all assumption necessary in the variational framework~\cite{LiuWei2015SPDE}. 
    Further, \(\rho^n\) is a Markov process by~\cite[Theorem 9.20]{Prato1992}. 

    Moreover, the operator \(A_n\) satisfies a strong dissipation/coercive condition by~\eqref{lemma: dissapation}. Applying~\cite[Theorem~3.1]{Rozovsky2018} we find a solution \((\tilde \rho_t^n, 0 \leq t \leq T)\), which lies in the space \(L^2([0,T]; H^{-\lambda-1}(\T^d))\) and satisfies the SPDE~\eqref{eq: approx_solution} in a weak sense, i.e. against a test function with respect to scalar product on \(H^{-\lambda-2}(\T^d)\). But obviously, \(\tilde \rho^n\) must coincide with \(\rho^n\), since \(\rho^n\) solves the equation in a strong sense, the operator \(A_n\) maps into \(H^{-\lambda-1}(\T^d)\) and the relation~\eqref{eq: cbf_relation} holds. 
\end{proof}

We have the following uniform in $n$ bounds on the solution:
\begin{proposition}
For any $\rho^n_0\in H^{-\lambda-2}(\T^d)$, the solution $\rho^n$ to~\eqref{eq: approx_solution} satisfies
\begin{equation} \label{eq: aux_uniform_bound}
          \sup\limits_{n \in \N} \E\big[ \sup\limits_{0 \le t \le T} \norm{\rho_t^n}_{H^{-\lambda-2}(\T^d)}^2 \big] + \norm{\rho_t^n}_{L^2_{\mathcal{F}}([0,T];H^{-\lambda-1}(\T^d))}^2 
          \le C  \E\big[\norm{\rho_0}_{H^{-\lambda-2}(\T^d)}^2 \big]. 
    \end{equation}
and further
\begin{equation}
\label{eq:unif_t}
\E\big[\norm{\rho_t^n-\rho_s^n}_{ H^{-\lambda-2}(\T^d)}^2\big] \leq C|t-s|.
\end{equation} 
\end{proposition}

\begin{proof}
    The approximated equation satisfies the assumption to apply It\^{o}'s formula~\cite[Theorem~4.2.5]{LiuWei2015SPDE} in the Hilbert space \(H^{-\lambda-2}(\T^d)\), since the equality~\eqref{eq: approx_solution} also holds in \(H^{-\lambda-2}(\T^d)\). We obtain  
    \begin{align*}
        &\norm{\rho_t^n}_{ H^{-\lambda-2}(\T^d)}^2-  \norm{ \rho_0^n }_{ H^{-\lambda-2}(\T^d)}^2  \\
        &\quad = 
2\int\limits_0^t \langle A_n(s,\rho_s^n),\rho_s^n \rangle_{ H^{-\lambda-2}(\T^d)} + \sum\limits_{j=1}^d \norm{B_j(s)}_{L^2(L^2(\T^d),H^{-\lambda-2}(\T^d))}^2 \Id s + \; \mathrm{stochastic \; integral} \\
        &\quad \le C \norm{ \rho_0 }_{ H^{-\lambda-2}(\T^d)}^2  +
         \int\limits_0^t C \norm{\rho_s^n}_{{H}^{-\lambda-2}(\T^d)}^2 - \delta \norm{\rho_s^n}_{{H}^{-\lambda-1}(\T^d)}^2   \Id s+ C(T)+ \; \mathrm{stochastic \; integral },
    \end{align*}
    where we used Lemma~\ref{lemma: dissapation} in the last step, since \(\rho_s^n \in H^{-\lambda-1}(\T^d)\). After taking the expectation the stochastic integral vanishes. 
    Hence, after an application of Gronwall's lemma we obtain 
    \begin{equation*}
        \sup\limits_{0 \le t \le T} \E\big[\norm{\rho_t^n}_{ H^{-\lambda-2}(\T^d)}^2\big] \le C\E\big[\norm{\rho_0}_{H^{-\lambda-2}(\T^d)}^2 \big],
    \end{equation*}
    which after a bootstrap improves to \eqref{eq: aux_uniform_bound}. 
    Additionally, by the same computations we obtain 
    \begin{align*}
        &\E\big[\norm{\rho_t^n-\rho_s^n}_{ H^{-\lambda-2}(\T^d)}^2\big]\\
        &\quad\le   2\int\limits_s^t \E\big[\langle A_n(u,\rho_u^n),\rho_u^n \rangle_{ H^{-\lambda-2}(\T^d)}\big]  + \sum\limits_{j=1}^d \norm{B_j(u)}_{L^2(L^2(\T^d),H^{-\lambda-2}(\T^d))}^2 \Id u \\
        &\quad \quad +\sum\limits_{j=1}^d  \E\bigg[\int\limits_s^t \langle \rho^n_s , B_j(s) \Id W_j(s) \rangle_{H^{-\lambda-2}(\T^d)}\bigg] \\
        &\quad \le C \int\limits_s^t \E\big[\norm{\rho_u^n}_{H^{-\lambda-2}(\T^d)}^2+1\big] \Id u    \\
        &\quad \quad   +\sum\limits_{j=1}^d \E\bigg[ \int\limits_s^t \norm{\rho_u^n}_{H^{-\lambda-2}(\T^d)}^2 \Id u \bigg]^{\frac{1}{2}}\E\bigg[ \int\limits_s^t \norm{B_j(u)}_{L^2(L^2(\T^d),H^{-\lambda-2}(\T^d))}^2 \Id u\bigg]^{\frac{1}{2}} \\
        &\quad \le C|t-s|,
    \end{align*}
    where we used the Burkholder-Davis-Gundy (BDG) inequality for Hilbert space-valued martingales~\cite[Theorem1.1]{Marinelli2016} and the bound~\eqref{eq: aux_uniform_bound} in the last step. 
\end{proof}

\subsection{Well-posedness and stability of SPDE}
\label{sec: well-posedness_SPDE}

As explained above, to prove Theorem \ref{theorem: existence_SPDE}, the main novelty is that we establish convergence in probability of the approximations.

\begin{theorem}
\label{thm: conv_prob}
    Let \((\rho^n, n \in \N)\) be the solution to~\eqref{eq: approx_solution} and \(\rho\) the solution to~\eqref{eq: limiting_spde}, with initial conditions $\rho_0$. Then, \((\rho^n, n \in \N)\) converges in probability towards \(\rho\) on the space \(L^2([0,T],H^{-\lambda-2}(\T^d))\). 
\end{theorem}

We prove these two main results together.

\begin{proof}[Proof of Theorems \ref{theorem: existence_SPDE} and \ref{thm: conv_prob}]
   Thanks to \eqref{eq: aux_uniform_bound} and \eqref{eq:unif_t},      
 we can apply~\cite[Lemma~5.2]{Shang2024} to obtain tightness of the laws provided by the sequence \((\rho_n, n \in \N)\) on the space \(L^2([0,T];H^{-\lambda-2}(\T^d))\), since \(H^{-\lambda-1} \hookrightarrow H^{-\lambda-2}(\T^d)\) is compact by~\cite[Proposition~4.6]{Triebel2006}. Our goal is to apply the Gyöngy-Krylov diagonal criterion, which we recall in Lemma~\ref{lemma: gyongy}. Hence, if we have the subsequence \((n_1(k),n_2(k))\), the laws of \(((\rho^{n_1(k)},\rho^{n_2(k)}), k \in \N)\) is still tight. And obviously the sequence 
 \[
 ((\rho^{n_1(k)},\rho^{n_2(k)}, \rho_0,  (W_j,j=1,\ldots,d)), k \in \N)\] is tight on 
 \((L^2([0,T];H^{-\lambda-2}(\T^d)))^2 \times L^2(H^{-\lambda-2}(\T^d)) \times C([0,T];U_1^d)\).

    
    By Prohorov's theorem there is  a further subsequence, which will not be renamed, such that the law of the tuple converges. Applying Skorhod's representation theorem, we can find a new probability space \((\tilde \Omega, \tilde \cF, \tilde \P)\) with random variables \((\tilde \rho^{n_1(k)},\tilde \rho^{n_2(k)},(\tilde W_{j,k},j=1,\ldots d))\) such that they have the same law as the tuple \((\rho^{n_1(k)},\rho^{n_2(k)},(W_j,j=1,\ldots,d))\) and converge almost everywhere to \((\tilde \rho^1,\tilde \rho^2,(\tilde W_j , j=1 , \ldots, d))\). By Levy's charterization for Brownian motion and the generalized Yamada--Watanabe theorem~\cite[Theorem~1.5]{Kurtz2014} we can demonstrate that \(W_j\) is a cylindrical Brownian motion and \((\tilde \rho^{n_1(k)},\tilde \rho^{n_2(k)})\) solve the SPDE~\eqref{eq: approx_solution}, respectively with their index of the subsequence. Additionally, by employing~\eqref{lemma: dissapation}, we can deduce that 
    \begin{equation*}
        \norm{\rho^{n_1(k)}}_{L^2_{\tilde \cF}([0,T];H^{-\lambda-1}(\T^d))}^2+   \norm{\rho^{n_1(k)}}_{L^2_{\tilde \cF}([0,T];H^{-\lambda-1}(\T^d))}^2 \le C. 
    \end{equation*}
    Since \(H^{-\lambda-1}(\T^d)\) is dense in \(H^{-\lambda-2}(\T^d)\) and we have the almost everywhere convergence in \(L^2([0,T];H^{-\lambda-2}(\T^d))^2\),  
    the lower semi-continuity of the norm provides
  \begin{equation*}
        \norm{\tilde \rho^1}_{L^2_{\tilde \cF}([0,T];H^{-\lambda-1}(\T^d))}^2+   \norm{\tilde \rho^2}_{L^2_{\tilde \cF}([0,T];H^{-\lambda-1}(\T^d))}^2 \le C.
    \end{equation*}
    Since our equation is linear, it is quite standard to demonstrate that the following equality holds 
       \begin{align} \label{eq: skorohod_spde}
        &\langle \tilde \rho_t^{i},\varphi \rangle_{H^{-\lambda-2}(\T^d)}-\langle \rho_0,\varphi \rangle_{H^{-\lambda-2}(\T^d)} \nonumber \\
        &\quad =  \int\limits_0^t \langle \tilde \rho_s^{i}, A'(t) \varphi \rangle_{H^{-\lambda-2}(\T^d)} \Id s + \int\limits_0^t \sum\limits_{j=1}^d \langle B_{j}(s), \varphi \rangle_{H^{-\lambda-2}(\T^d)} \Id \tilde W_j(s) , \quad \tilde \P\text{-a.s.} 
    \end{align}
for \(\varphi \in C^\infty(\T^d)\) and \(i=1,2\). 
For \(f \in H^{-\lambda-1}(\T^d)\) we find 
\begin{align*}
    &\langle f, A'(t) \varphi \rangle_{H^{-\lambda-2}(\T^d)} \\
    &\quad =  \bigg\langle   \nabla \cdot (f b(t,\cdot,\mu_t) ) +\bigg\langle \nabla_x \cdot  \bigg(  \mu_t(x)  \frac{\delta b}{\delta m} (t,x,\mu_t , v ) \bigg), f(v)  \bigg\rangle,\varphi \bigg \rangle_{H^{-\lambda-2}(\T^d)} \\
    &\quad \quad -\sigma^2  \sum\limits_{j=1}^d \langle \partial_{x_j} f, \partial_{x_j} \varphi \rangle_{H^{-\lambda-2}(\T^d)} .   
\end{align*}
The right hand side defines a bilinear map on \(H^{-\lambda-1}(\T^d)\) since we have the bound 
\begin{equation*}
    | \langle f, A'(t) \varphi \rangle_{H^{-\lambda-2}(\T^d)} | \le C \norm{f}_{H^{-\lambda-1}(\T^d)}\norm{\varphi}_{H^{-\lambda-1}(\T^d)}.
\end{equation*}
Hence, we can extend the map to \(\varphi \in H^{-\lambda-1}(\T^d)\). For fixed \(f\in H^{-\lambda-1}(\T^d)\) the map is linear on \(H^{-\lambda-1}(\T^d)\) and we can apply~\cite[Proposition~3.4]{Rozovsky2018} to the canonical bilinear form (CBF) \([\cdot, \cdot]_{-\lambda-2}: H^{-\lambda-1}(\T^d) \times H^{-\lambda-3}(\T^d) \to \R\) for the triple \((H^{-\lambda-3}(\T^d),H^{-\lambda-2}(\T^d),H^{-\lambda-1}(\T^d))\) to find a map \(\hat A(t) \colon H^{-\lambda-1}(\T^d) \to  H^{-\lambda-3}(\T^d)\) such that 
\begin{align*}
    [\varphi ,\hat A(t) f ]_{-\lambda-2}  &=  \bigg\langle   \nabla \cdot (f b(t,\cdot,\mu_t) ) +\bigg\langle \nabla_x \cdot  \bigg(  \mu_t(x)  \frac{\delta b}{\delta m} (t,x,\mu_t , v ) \bigg), f(v)  \bigg\rangle,\varphi \bigg \rangle_{H^{-\lambda-2}(\T^d)} \\
    & \quad -\sigma^2  \sum\limits_{j=1}^d \langle \partial_{x_j} f, \partial_{x_j} \varphi \rangle_{H^{-\lambda-2}(\T^d)}    . 
\end{align*}
Consequently, we can replace the drift in the SPDE~\eqref{eq: skorohod_spde} with \([\varphi ,\hat A(t) \tilde \rho_s^{i} ]_{-\lambda-2}\). 

Applying~\cite[Theorem~2.13]{Rozovsky2018} we find a modification of \(\tilde \rho^{i}\), which is continuous in \(H^{-\lambda-2}(\T^d)\) and we can apply It\^o's formula for the norm of the difference to obtain
\begin{equation*}
    \E\big[\sup\limits_{0 \le t \le T} \norm{\tilde \rho_t^{1}-\tilde \rho_t^2}_{H^{-\lambda-2}(\T^d)}^2\big] = 0 
\end{equation*}
by utilizing the dissipation and Gronwall's lemma as previously. This gives uniqueness of the solution to SPDE \eqref{eq: limiting_spde}. 
As a result we can apply Lemma~\ref{lemma: gyongy} and obtain a process \(\rho\) such that \((\rho_t^n, n \in \N)\) converges in probability on \(L^2([0,T],H^{-\lambda-2}(\T^d))\). 
Additionally, another application of It\^o's formula provides the estimate in Definition~\eqref{eq:esti_spde}. 
\end{proof}

\begin{remark}
    Similar arguments as in the proof Theorem~\ref{theorem: existence_SPDE} can be used for parabolic SPDE's of the form 
    \begin{equation*}
        \Id u(t) = \frac{1}{2} \sum\limits_{i,j=1}^d \partial_{x_i}(a_{i,j}\partial_{x_j} u(t) ) + b \cdot \nabla u(t) + c u(t) + B(t,u(x)) \Id W(t) 
    \end{equation*}
    with  sufficiently smooth coefficient functions \(a_{i,j},b,c,d\) and ellipticity of \(a_{i,j}\) to find a regular approximation which converges in probability.

    For a detailed discussion in the more difficult setting of the Navier–Stokes equations, we refer to \cite[Chapter 2.4]{Flandoli2023}.
\end{remark}

A further consequence of Theorems~\ref{theorem: existence_SPDE} and ~\ref{thm: conv_prob} and It\^o's formula is the following stability estimate.

\begin{proposition} \label{cor: stability_spde}
Let \(\rho^1 \), \(\rho^2\) be two solutions of SPDE~\eqref{eq: limiting_spde} with initial data \(\rho^1_0,\rho^2_0\) and for any $n$ let \(\rho^{n,1} \), \(\rho^{n,2}\) be two solutions of SPDE~\eqref{eq: approx_solution} with initial data 
\(\rho^{1,n}_0,\rho^{2,n}_0\). Then 
    \begin{align} 
    \label{eq: stability_spde}
         \E\big[\sup\limits_{0 \le t \le T } \norm{\rho_t^1-\rho_t^2}_{H^{-\lambda-2}(\T^d)}\big] + \norm{\rho^1-\rho^2}_{L^2_{\cF}([0,T]; H^{-\lambda-1}(\T^d)) } &\le C \E \norm{\rho_0^1-\rho^2_0}_{H^{-\lambda-2}(\T^d)}^2  ,\\
         \E\big[\sup\limits_{0 \le t \le T } \norm{\rho_t^{n,1}-\rho_t^{n,2}}_{H^{-\lambda-2}(\T^d)}\big] + \norm{\rho^{n,1}-\rho^{n,2}}_{L^2_{\cF}([0,T]; H^{-\lambda-1}(\T^d)) } &\le C \E\norm{\rho_0^{n,1}-\rho^{n,2}_0}_{H^{-\lambda-2}(\T^d)}^2 . 
    \end{align}
    Moreover, there exists a set $\tilde{\Omega}$ with $\P(\tilde{\Omega})$ such that for any $\omega\in\tilde{\Omega}$, for any $n$, 
    \begin{align}
    \label{eq:LIP_flow}
         \sup\limits_{0 \le t \le T } \norm{\rho_t^1-\rho_t^2}_{H^{-\lambda-2}(\T^d)} + \norm{\rho^1-\rho^2}_{L^2([0,T]; H^{-\lambda-1}(\T^d)) } &\le C  \norm{\rho_0^1-\rho^2_0}_{H^{-\lambda-2}(\T^d)}^2  ,\\
         \sup\limits_{0 \le t \le T } \norm{\rho_t^{n,1}-\rho_t^{n,2}}_{H^{-\lambda-2}(\T^d)} + \norm{\rho^{n,1}-\rho^{n,2}}_{L^2([0,T]; H^{-\lambda-1}(\T^d)) } &\le C \norm{\rho_0^{n,1}-\rho^{n,2}_0}_{H^{-\lambda-2}(\T^d)}^2 .
    \end{align} 
    where al random variables are evaluated on $\omega$. 
\end{proposition}
\noindent The last claim holds since the noise $\zeta(t)$ does not depend on the solution to the SPDE.

\subsection{Derivative of the Flow of SPDE}

Next, we want to analyze the regularity with respect to deterministic initial data. Let us introduce our guess for the Fréchet derivative. Let \(y\) solve the following infinite dimensional evolution equation 
\begin{equation}\label{eq: frechet_SPDE}
    \Id y_t(h) = A(t,y_t(h)) \Id t 
\end{equation}
with initial data \(y_0(h)= h\) for \(h \in  H^{-\lambda-2}(\T^d)\) deterministic. By the same arguments as in proof of Theorem~\ref{theorem: existence_SPDE} we deduce that~\eqref{eq: limiting_spde} has a unique solution with values in \( H^{-\lambda-2}(\T^d) \). 
It is clear that the operator \(h \mapsto (y_t(h), 0 \leq t \leq T)\) is a linear map from \( H^{-\lambda-2} (\T^d)\) to \( C([0,T];  H^{-\lambda-2}(\T^d))\), bounded by \eqref{eq:esti_spde}, and is independent of the solution to the SPDE ~\eqref{eq: limiting_spde}, in particular independent of its initial condition, since the SPDE is linear. We denote this operator by 
\[
\mathsf{Y} \in  L(H^{-\lambda-2} ,C([0,T];  H^{-\lambda-2}(\T^d))), \qquad \mathsf{Y}(h) = y_t(h). 
\] 
Moreover,  as in Proposition~\ref{cor: stability_spde} the map is continuous and is a candidate for the Fréchet derivative. 
We denote the \emph{flow} of the SPDE \eqref{eq: limiting_spde}  by $\rho_t(f)$, where $\rho$ is the solution, seen as a function of the initial condition $f\in H^{-\lambda-2}(\T^d))$, for any $t$ and almost every $\omega$; $\rho_t$ is an affine and random function of the initial condition, from $H^{-\lambda-2}(\T^d))$ to itself.  Note that  \eqref{eq:LIP_flow} yields that the flow is Lipschitz, for any $t$ and almost every $\omega$.

\begin{lemma}[Differentiability with respect to initial data]\label{lemma: diff_spde}
     The flow \((\rho_t, 0 \leq t \leq T)\) is infinitely many times Fréchet differentiable, for any $t$ and almost every $\omega$, with first derivative given by   
     \[
      \nabla \rho_t(f)(\omega) (h) = y_t(h),
     \]
     i.e. $\nabla \rho_t(f)(\omega) = \mathsf{Y}_t$ as linear operators on $H^{-\lambda-2}(\T^d)$.  
\end{lemma}
\begin{proof}
Let \(h \in  H^{-\lambda-2}(\T^d)\), \((\rho^h_t, 0 \leq t \leq T)\) be a solution to the fluctuation SPDE~\eqref{eq: limiting_spde} with initial datum \(\rho_0+  h \) and \((\rho_t,0 \leq t \leq T)\) the solution with initial datum \(\rho_0\). Additionally, let \((y_t, 0 \leq t \leq T)\) be the solution to~\eqref{eq: frechet_SPDE} with initial datum \(h\). 
Notice that \(\rho^h_s-\rho_s\) solves~\eqref{eq: frechet_SPDE} with initial datum \( h\). Hence,  \(\rho^h_s-\rho_s- y_s\) solves~\eqref{eq: frechet_SPDE} with initial datum zero. By uniqueness of the solution~\eqref{eq: frechet_SPDE} it follows that 
\begin{equation*}
    \rho^h_s-\rho_s-y_s = 0,  
\end{equation*}
which implies that  \((\rho_t, 0 \leq t \leq T)\) is Fréchet differentiable with respect to the initial datum.  
Notice that, since \(\rho\) is a affine-linear SDPE, the Fréchet derivative is independent of the point it is taken. Hence, it follows that \(\rho\) is infinitely many times Fréchet differentiable.
\end{proof}

All above results remain valid if we let the process start at time \(s > 0\) with an initial deterministic condition \(f \in H^{-\lambda-2}(\T^d)\). Accordingly, we denote by \(\rho_{s,\cdot}(f)\) and \(\rho^n_{s,\cdot}(f)\) the solutions to~\eqref{eq: limiting_spde} and~\eqref{eq: approx_solution}, respectively, starting at time \(s\) with initial condition \(f\). To keep the notation concise, we omit the index \(s\) whenever \(s=0\) and the initial condition \(f\), whenever the process starts from the initial condition \(\rho_0 \in L^2_{\cF_0}(H^{\lambda-2}(\T^d))\). The same convention applies to the processes \(y\), \(y^n\), \(\mathsf{Y}\), and \(\mathsf{Y}^n\).
Moreover, when the initial condition is not essential for the context, we omit it from the notation. Similarly, if the process starts at time zero, we continue to use the original notation without additional indices.

Notice that with the same arguments, we can prove the differentiability of the approximation process \(\rho^n\). In the following we will use \(\mathsf{Y}^n\) for the derivative map and \(y^n\) for the process, if we replace \(\rho\) by \(\rho^n\) and, consequently, \(A\) by \(A_n\) in~\eqref{eq: frechet_SPDE}. Observe that \(\mathsf Y^n \colon H^{-\lambda-2}(\T^d) \to C([0,T];  H^{-\lambda-2}(\T^d)) \).


\begin{lemma}[Properties of the derivative]\label{cor: properties_frechet}
Let \(\mathsf Y^n_{s,t} \), \((y_{s,t}^n, n\in \N)\) be given as above.
Then, for \(t \in [0,T]\), \(u,s \in [0,t]\), we have 
\begin{align}
\label{der_1}
    \sup\limits_{n \in \N} \norm{\mathsf{Y}_{s,t}^n (h)}_{ H^{-\lambda-2}(\T^d)} = \sup\limits_{n \in \N} \norm{y_{s,t}^n(h)}_{ H^{-\lambda-2}(\T^d)} 
    &\le C \norm{h}_{ H^{-\lambda-2}(\T^d)}, \\
    \label{der_2}
     \norm{Y^n_{u,t}(h) - Y^n_{s,t}(h)}_{ H^{-\lambda-2}(\T^d)}
    &\le C(n) |u-s|\norm{h}_{ H^{-\lambda-2}(\T^d)}  . 
\end{align}
\end{lemma}
\begin{proof}
The first inequality follows readily from ~\eqref{lemma: dissapation}. We have  
\begin{align*}
    \norm{y_{s,t}^n(h)}_{H^{-\lambda-2}(\T^d)}^2
    &\le \norm{h}_{H^{-\lambda-2}(\T^d)}^2 +\int\limits_s^t \langle A_n(r,y_{s,r}^n(h)),y_{s,r}^n(h) \rangle_{H^{-\lambda-2}(\T^d)} \Id r \\
   & \le \norm{h}_{H^{-\lambda-2}(\T^d)}^2 +  C \int\limits_s^t \norm{y_{s,r}^n(h)}_{H^{-\lambda-2}(\T^d)}^2 \Id r,
\end{align*}
which after an application of Gronwall's lemma provides the first inequality in our satement. 
For the second inequality, let w.l.o.g. \(s < u\). Then, 
\begin{align*}
    \norm{Y^n_{s,t}(h) - Y^n_{u,t}(h)}_{ H^{-\lambda-2}(\T^d)}
    &= \norm{y_{s,t}^n(h)-y_{u,t}^n(h)}_{ H^{-\lambda-2}(\T^d)}\\
    &=\norm{y_{u,t}^n(y^n_{s,u}(h))-y_{u,t}^n(h)}_{ H^{-\lambda-2}(\T^d)}\\
    &\le C \norm{y^n_{s,u}(h)-h}_{ H^{-\lambda-2}(\T^d)}\\
    &\le C \int\limits_s^u \norm{A_n(r,y^n_{u,r}(h))}_{ H^{-\lambda-2}(\T^d)} \Id r \\
    &\le C(n) \int\limits_s^u \norm{y^n_{u,r}(h)}_{ H^{-\lambda-2}(\T^d)}  \Id r \\
    &\le C(n) |u-s|\norm{h}_{ H^{-\lambda-2}(\T^d)} , 
\end{align*}
where we used the flow property of the infinite dimensional evolution equation \(y^n\).
\end{proof}

\section{Semigroups and Generators}
\label{sec:generators}


\noindent \emph{Throughout this section Assumptions~\ref{ass: inital_cond}-\ref{ass: coef_fokker} are in force.} 
\smallskip

The first goal of this section is to study regularity and stability of the semigroup related to ~\eqref{eq: approx_solution} and thus to compute the generator of the  SPDE~\eqref{eq: approx_solution}. Unfortunately, we can not directly apply the framework of Da Prato, Zabczyk~\cite[Chapter~4]{Prato1992} for the SPDE~\eqref{eq: limiting_spde}, since the semigroup will not be strongly continuous and is only a weak solution, i.e. the equation holds only in the sense of distributions. Hence, we proceed in a different way, and want first to compute the time and space derivatives of the semigroup corresponding to the approximation SPDE~\eqref{eq: approx_solution}. Then we apply It\^o formula for Hilbert-valued SDEs and we study the Sobolev regularity of the Fréchet derivatives, viewed as functions of one or two variables. Hence, we employ such regularity to apply It\^o formula for flows of measures and the restriction to empirical measures, thus characterizing the generator of the fluctuation process $\rho^N$. 

Throughout this section we  assume that  \(\Phi \in FC^\infty(H^{-\lambda-2}(\T^d))\), where the set is given in Definition \ref{def: cylindrical_class} in the appendix, and we fix the representation
\begin{equation} \label{eq: representation_cylindrical} 
\Phi(f) = g( \langle f, \varphi_1 \rangle_{H^{-\lambda-2}} , \ldots,\langle f, \varphi_m   \rangle_{H^{-\lambda-2}} ) 
\end{equation} 
for \( g \in C_c^\infty(\R^m), \varphi_1, \ldots, \varphi_m \in C^\infty(\T^d)  \). 
Note that $\Phi\in C^\infty_b(H^{-\lambda-2}(\T^d))$ and this set of cylindrical functions approximate continuous functions, uniformly on compacts; see Lemma \ref{lemma: approximation_hilbert_smooth}. Since the space $H^{-\lambda-2}(\T^d)$ remains fixed, we may omit it in the notation $[\Phi]_{C^2}$.

\subsection{Regularity of approximating semigroup}

Fix a \(t \in [0,T]\). Let us introduce the ``semigroup'' to SPDE~\eqref{eq: limiting_spde} by
\begin{equation} \label{eq: def_approx_semigroup}
    (s,f) \mapsto T_{s,t}^n\Phi(f):= \E[\Phi(\rho_{s,t}^n( f)] ,
\end{equation}
where \(0 \le s \le t \le T\), \(f \in H^{-\lambda-2}(\T^d)\), \(\Phi \in C_\ell(H^{-\lambda-2}(\T^d))\)
and \(\rho_{s,t}^n( f)\) is the solution of the SPDE~\eqref{eq: approx_solution} with initial data \( f \) at initial time \(s\). 
Recall that \(\rho^n\) is a Markov process and we have the following Chapman--Kolmogorov equation
\begin{equation} \label{eq: chapman}
    T_{s,t}^n\Phi( f) = T_{s,u}^n ( T_{u,t}^n \Phi ) ( f) , \quad 0 \le s \le u \le t \le T
\end{equation}
for \(f \in H^{-\lambda-2}(\T^d)\). 

We have the following explicit computation of the space derivative of \(T_{s,t}^n \Phi\) for \(\Phi \in FC^\infty(H^{-\lambda-2}(\T^d))\), which is similar to Lemma \ref{lemma: derivative}.

\begin{lemma} \label{lemma: derivative_test_semigroup}
Let \(\Phi \in  FC^\infty(H^{-\lambda-2}(\T^d))\). 
Then, for every \(s \in [0,t]\) we have the following formulas for the Fréchet derivatives: 
\begin{align*}
\begin{cases}
    \nabla T_{s,t}^n\Phi (  \cdot):  H^{-\lambda-2}(\T^d) \to  H^{-\lambda-2}(\T^d)  \\
    f \mapsto  \sum\limits_{i=1}^m \E\big[\partial_{x_i} g( \langle \rho_{s,t}^n( f), \varphi_1\rangle_{H^{-\lambda-2}}  , \ldots,\langle \rho_{s,t}^n( f), \varphi_m  \rangle_{H^{-\lambda-2}} ) \big] (\mathsf{Y}_{s,t}^n)^* \varphi_i
\end{cases}
\end{align*}
and 
\begin{align*}
\begin{cases}
    \nabla^2 T_{s,t}^n\Phi (\cdot):  H^{-\lambda-2}(\T^d) \to L( H^{-\lambda-2}(\T^d), H^{-\lambda-2}(\T^d))\\
    f \mapsto \!\bigg(h \!\mapsto\!\! \sum\limits_{i,j=1}^m  \!\E\big[\partial_{x_i}\partial_{x_j} g( \langle \rho_{s,t}^n( f), \varphi_1\rangle_{H^{-\lambda-2}} , \ldots,\langle \rho_{s,t}^n(f), \varphi_m  \rangle_{H^{-\lambda-2}} )\big] 
    \langle   (\mathsf{Y}_{s,t}^n)^* \varphi_j, h \rangle_{H^{-\lambda-2}}  (\mathsf{Y}_{s,t}^n)^* \varphi_i 
 \bigg).
    \end{cases}
\end{align*}
\end{lemma}
\begin{remark} \label{remark: infinite_diff}
    Even without computing the derivatives explicitly, we obtain the fact that \(T_{s,t}^n(\cdot)\) is infinitely many times Fréchet differentiable by the chain rule, since \(\Phi \in  FC^\infty(H^{-\lambda-2}(\T^d))\) and \(f \mapsto  \rho_{s,t}^n(f)\) are infinitely many times Fréchet differentiable by Lemma \ref{lemma: diff_spde}. Additionally, all derivatives are bounded because \(\Phi \in  FC^\infty(H^{-\lambda-2}(\T^d)) \) and \(f \mapsto  \rho_{s,t}^n(f)\) is affine linear. The claim also holds for \(T_{s,t}^n\Phi \) replaced by \(T_{s,t} \Phi \). 
\end{remark}
The next result strengthens the above remark by showing that the regularity can be made uniform in \(n \in \N\).

\begin{lemma}[Regularity of derivatives]
\label{cor: reg_derivative}
Let \(\Phi \in  C_{\ell}^2( H^{-\lambda-2}(\T^d) ) \). 
    For the functions \(\nabla T_{s,t}^n(f)\), \(\nabla^2 T_{s,t}^n(f)\) we have the following estimates, where $C$ is independent of $n$ and of $s,t$:
    \begin{align}
        \sup\limits_{f \in  H^{-\lambda-2}(\T^d) } \norm{\nabla T_{s,t}^n\Phi ( f)}_{ H^{-\lambda-2}(\T^d)} \le C [\Phi]_{C^1}    , 
        \label{eq:est_1}\\
        \sup\limits_{f \in  H^{-\lambda-2}(\T^d)} \norm{\nabla^2 T_{s,t}^n\Phi (f)}_{L( H^{-\lambda-2}(\T^d), H^{-\lambda-2}) } \le C [\Phi]_{C^2} .
        \label{eq:est_2}
    \end{align}
\end{lemma}

\begin{proof}
For the moment, let us view \(\nabla T_{s,t}^n\Phi(f)\) as a linear map from \(H^{-\lambda-2}(\T^d)\) to \(\R\), and similarly for \(\nabla \Phi\). The advantage of this viewpoint is that we can apply the classical chain rule to obtain 
\begin{equation*}
    \nabla T_{s,t}^n\Phi ( f)(h)
    = \E \big[ \nabla \Phi(\rho_{s,t}^n(f)) \circ \mathsf{Y}^n_{s,t}(h) \big] 
\end{equation*}
for each \( h \in H^{-\lambda-2}(\T^d)\). 
Consequently, 
    \begin{align*}
       |\nabla T_{s,t}^n\Phi ( f)(h)|
       &\le \E\big[|\nabla \Phi(\rho_{s,t}^n(f)) \circ \mathsf{Y}^n_{s,t}(h) |\big] \\
       &\le [\Phi]_{C^1}    \norm{\mathsf{Y}^n_{s,t}(h)}_{H^{-\lambda-2}(\T^d)}\\
       &\le C [\Phi]_{C^1}   \norm{h}_{H^{-\lambda-2}(\T^d)}, 
    \end{align*}
    where we used \eqref{der_1} in the second step; thus \eqref{eq:est_1} follows. 
For the second derivative, we have 
\begin{align*}
   | \langle \nabla^2 T_{s,t}^n\Phi (f)& (\tilde{h}), h\rangle_{H^{-\lambda-2}(\T^d)} | 
   = |\E \langle \nabla^2 \Phi(\rho_{s,t}^n(f)) 
    \big( \mathsf{Y}^n_{s,t}(\tilde{h}) \big) ,  \mathsf{Y}^n_{s,t}(h)  \rangle_{H^{-\lambda-2} (\T^d)}| \\
    & \leq  \E \norm{\nabla^2 \Phi(\rho_{s,t}^n(f)) }_{L(H^{-\lambda-2}(\T^d), H^{-\lambda-2}(\T^d)}  
    \norm{\mathsf{Y}^n_{s,t}(\tilde{h})}_{H^{-\lambda-2}(\T^d)} 
    \norm{\mathsf{Y}^n_{s,t}(h)}_{H^{-\lambda-2}(\T^d)} \\
    & \leq C^2 \sup_{f\in H^{-\lambda-2}(\T^d) } 
    \norm{\nabla^2 \Phi(f) }_{L(H^{-\lambda-2}(\T^d), H^{-\lambda-2}(\T^d)}
    \norm{\tilde{h}}_{H^{-\lambda-2}(\T^d)} 
    \norm{h}_{H^{-\lambda-2}(\T^d)},
\end{align*}
which yields \eqref{eq:est_2}.    
\end{proof}
We have the following regularity with respect to initial time. 
\begin{lemma}[Continuity of derivatives with respect to time] \label{cor: con_semigroup_time}
Let \(t \in [0,T], u,s \in [0,t]\),\(\Phi \in  FC^\infty(H^{-\lambda-2}(\T^d))\) and \(f \in  H^{-\lambda-2}(\T^d)\). 
    For the functions \(\nabla T_{s,t}^n\Phi(f)\), \(\nabla^2 T_{s,t}^n\Phi( f ) \) we have the following estimates with respect to time 
    \begin{align}
    \label{der_semi_1}
        \norm{\nabla T_{u,t}^n\Phi(f)- \nabla T_{s,t}^n\Phi(f) }_{ H^{-\lambda-2}(\T^d)}
        &\le C(n) |u-s|(1+\norm{f}_{ H^{-\lambda-2}(\T^d) }), \\
        \label{der_semi_2}
     \norm{\nabla^2 T_{u,t}^n\Phi(f )- \nabla^2 T_{s,t}^n\Phi( f )}_{L( H^{-\lambda-2}(\T^d), H^{-\lambda-2}(\T^d))  } &\le C(n) |u-s| (1+\norm{f}_{ H^{-\lambda-2}(\T^d) }).
    \end{align}
\end{lemma}
\begin{proof}
W.l.o.g we assume \(s <u<t\). 
For the first inequality we have 
\begin{align*} 
&\norm{\nabla T_{u,t}^n\Phi( f)- \nabla T_{s,t}^n\Phi( f) }_{H^{-\lambda-2}(\T^d)} \\
&\quad \le \sum\limits_{i=1}^m \E\bigg[ \big\|(\mathsf{Y}^n_{u,t})^* \varphi_i- (\mathsf{Y}^n_{s,t})^* \varphi_i  \big\|_{ H^{-\lambda-2}}   |\partial_{x_i}  g( \langle \rho_{u,t}^n(f), \varphi_1\rangle_{H^{-\lambda-2}} , \ldots,\langle \rho_{u,t}^n( f), \varphi_m  \rangle_{H^{-\lambda-2}} )| \bigg]   \\
&\quad +  \sum\limits_{i=1}^m  \E\bigg[ \norm{ (\mathsf{Y}^n_{s,t})^* \varphi_i}_{_{H^{-\lambda-2}}}    \big|\partial_{x_i}  g( \langle \rho_{u,t}^n(f), \varphi_1\rangle_{H^{-\lambda-2}} , \ldots,\langle \rho_{u,t}^n( f), \varphi_m  \rangle_{H^{-\lambda-2}} ) \\
&\quad \quad \qquad - \partial_{x_i}  g( \langle \rho_{s,t}^n(f), \varphi_1\rangle_{H^{-\lambda-2}} , \ldots,\langle \rho_{s,t}^n(f), \varphi_m  \rangle_{H^{-\lambda-2}} ) \big| \bigg]   \\
&\quad =:(I)+(II)
\end{align*}
For the first term we find 
\begin{align*}
(I) & \le C \sum\limits_{i=1}^m \E\bigg[ \big\|(\mathsf{Y}^n_{u,t})^* \varphi_i-(\mathsf{Y}^n_{s,t})^* \varphi_i  \big\|_{ H^{-\lambda-2}}\bigg]\\
   &\quad =  C \sum\limits_{i=1}^m  \sup\limits_{\norm{h}_{  H^{-\lambda-2} } \le 1 }
   | \langle \mathsf Y^n_{u,t}(h)- \mathsf Y^n_{s,t}(h) , \varphi_i \rangle_{ H^{-\lambda-2}} | \\
   &\quad \le C(n) |u-s|\sum\limits_{i=1}^m  \norm{\varphi_i}_{ H^{-\lambda-2} },
\end{align*}
where we used \eqref{der_2}.
For the second term \((II)\) we use the fact that \(\rho^n\) is a unique solution, which implies
\begin{equation*}
    \rho_{s,t}^n(f) = \rho_{u,t}^n(\rho_{s,u}^n(f)), \quad \P\text{-a.s.};
\end{equation*}
see for instance~\cite[Remark~4.2.11]{LiuWei2015SPDE}. We obtain 
\begin{align*}
    (II) &\le C\norm{ g}_{C_b^2(\R^m)} \sum\limits_{j=1}^m \E\big[ |\langle\rho_{s,t}^n(f)- \rho_{u,t}^n(f) , \varphi_j \rangle_{H^{-\lambda-2}(\T^d)} |\big]  \\
    &\le C \E\big[ \norm{\rho_{u,t}^n(\rho_{s,u}^n(f))- \rho_{u,t}^n(f)}_{ H^{-\lambda-2}(\T^d) }  \big] \\
    &\le C \E\big[\norm{\rho_{s,u}^n( f)- f}_{ H^{-\lambda-2}(\T^d) }  \big] \\
    &\le C  \E\bigg[ \int\limits_{s}^u \norm{A_n(r,\rho_{u,r}^n( f)  ) }_{ H^{-\lambda-2}(\T^d) }  \Id r  \bigg] \\
    &\le C(n)  \E\bigg[ \int\limits_{s}^u \norm{\rho_{u,r}^n(f)   }_{ H^{-\lambda-2}(\T^d) }  \Id r  \bigg] \\
     &\le C(n) |u-s| \norm{f}_{ H^{-\lambda-2}(\T^d) },
\end{align*}
where we used the uniform bound~\eqref{eq: aux_uniform_bound} in third and last step and \eqref{item: approx_bound_smaller_space} in the fifth step. Putting the inequalities for \((I)\) and \((II)\) together, the claim follows. 

For the second inequality we notice that by the affine-linear nature of the SPDE~\eqref{eq: limiting_spde}, we do not produce any new terms in the derivative (see Lemma~\ref{lemma: derivative_test_semigroup}). Instead, we have now three terms, which can be handled in the exact same way as before. 
\end{proof}

\subsection{Generator of approximating SPDE}

In this section, we derive the generator of the approximation SPDE \((\rho_t^n,0 \leq t \leq T)\) corresponding to the drift operator \(A_n\). 
Let us recall that the SPDE \(\rho^n\) is a Markov process~\cite[Theorem 9.20]{Prato1992}, and we can apply It\^{o}'s formula~\cite[Theorem~4.32]{Prato1992} for general smooth $\Phi$.  
Following~\cite[Chapter~9.3]{Prato1992} with the necessary modification to include the time non-homogeneus case, we compute the generator for the backward in time propagation.

\begin{proposition}
\label{lemma: ito_formula_approx}
Let \(\Phi \in  FC^\infty(H^{-\lambda-2}(\T^d))\) and \(T_{s,t}^n\) be given by~\eqref{eq: def_approx_semigroup}. Then, we have 
\begin{equation}
    \partial_s T_{s,t}^n\Phi(f) = - \frac{1}{2} \sum\limits_{j=1}^d  \mathrm{Tr}\big( \nabla^2 T_{s,t}^n\Phi(f) B_j(s) B_j^*(s) \big)- \langle A^n(s,f) , \nabla T_{s,t}^n\Phi (f)) \rangle_{{H}^{-\lambda-2}(\T^d)} 
\end{equation}    
for any $0\leq s<t \leq T$,  and we set
\begin{equation*}
   \mathcal{G}_s^n T_{s,t}^n\Phi(f) : =  - \partial_s T_{s,t}^n\Phi(f). 
\end{equation*}
In particular, \(f \mapsto \partial_s T_{s,t}^n \Phi(f) \) as a map from \(H^{-\lambda-2}(\T^d)\) to \(\R\) is continuous. 
\end{proposition}

\begin{proof} 
    Let us recall that $\rho^n$ solves SDE~\eqref{eq: approx_solution} 
in the space \( H^{-\lambda-2}(\T^d)\). 
 We can apply the Chapman--Kolmogorov equality~\eqref{eq: chapman} for \(\rho^n\). For small enough \(\epsilon >0\) we find 
    \begin{align*}
    T_{s,t}^n\Phi(f) &= \E \big[T_{s+\epsilon,t}^n \Phi(\rho_{s,s+\epsilon}(f)) \big]\\
        & = T_{s+\epsilon,t}^n\Phi( f)+ \int\limits_{s}^{s+\epsilon} \E\big[ \langle A_n(r, \rho_{s,r}^n(f)), \nabla T_{s+\epsilon,t}^n\Phi (\rho_{s,r}^n(f)) \rangle_{{H}^{-\lambda-2}(\T^d)} \big] \Id r \\
        &\quad  + \frac{1}{2} \sum\limits_{j=1}^d 
        \int\limits_{s}^{s+\epsilon}  \E\big[  \mathrm{Tr}\big( \nabla^2 T_{s+\epsilon,t}^n\Phi( \rho_{s,r}^n(f))  B_{j}(r) B_{j}^*(r) \big] \big) \Id r,
    \end{align*}
    where we applied It\^{o}'s formula~\cite[Theorem~4.32]{Prato1992} in the last step. 
This implies 
\begin{align}\label{eq: ito_aux1}
\begin{split}
    \frac{1}{\epsilon}\big(T_{s+\epsilon,t}^n\Phi-T_{s,t}^n\Phi\big)  & = - \frac{1}{\epsilon}\int\limits_{s}^{s+\epsilon} \E\big[ \langle A_n(r, \rho_{s,r}^n(f)), \nabla T_{s+\epsilon,t}^n\Phi (\rho_{s,r}^n(f)) \rangle_{{H}^{-\lambda-2}(\T^d)} \big] \Id r \\
    &\quad -\frac{1}{2\epsilon} \sum\limits_{j=1}^d 
        \int\limits_{s}^{s+\epsilon}  \E\big[  \mathrm{Tr}\big( \nabla^2 T_{s+\epsilon,t}^n \Phi ( \rho_{s,r}^n(f))  B_{j}(r) B_{j}^*(r) \big) \big] \Id r.     
\end{split}
\end{align}
It remains to demonstrate the the integrands are continuous with respect to \(r \in [s, s+\epsilon]\), so that the claim will then follow by sending $\epsilon$ to zero.
For the first order term we find  
\begin{align*}
     &\big|\E\big[ \langle A_n(r, \rho_{s,r}^n(f)), \nabla T_{s+\epsilon,t}^n \Phi (\rho_{s,r}^n(f)) \rangle_{H^{-\lambda-2}(\T^d)} - \langle A_n(s,f), \nabla T_{s,t}^n\Phi (f) \rangle_{H^{-\lambda-2}(\T^d)} \big] \big|\\
    &\quad \le \big| \E\big[ \langle A_n(r, \rho_{s,r}^n(f)), \nabla T_{s+\epsilon,t}^n \Phi (\rho_{s,r}^n(f))- \nabla T_{s,t}^n \Phi (\rho_{s,r}^n( f)) \rangle_{H^{-\lambda-2}(\T^d)} \big] \big| \\
    &\quad \quad + \big| \E\big[ \langle A_n(r, \rho_{s,r}^n(f)), \nabla T_{s,t}^n\Phi (\rho_{s,r}^n(f))- \nabla T_{s,t}^n\Phi(f) \rangle_{H^{-\lambda-2}(\T^d)} \big] \big|\\
    &\quad \quad+ \big| \E\big[ \langle A_n(r, \rho_{s,r}^n(f))-A_n(s,\rho_{s,r}^n(f)) , \nabla T_{s,t}^n \Phi(f) \rangle_{H^{-\lambda-2}(\T^d)} \big] \big| \\
    &\quad \quad+ \big| \E\big[ \langle A_n(s, \rho_{s,r}^n(f))-A_n(s,f) , \nabla T_{s,t}^n \Phi (f) \rangle_{H^{-\lambda-2}(\T^d)} \big] \big| \\
    &\quad =:(I)+(II)+(III)+(IV). 
\end{align*}
For the first term \((I)\) we obtain 
\begin{align*}
    (I) &= \big| \E\big[ \langle A_n(r, \rho_{s,r}^n(f)), \nabla T_{s+\epsilon,t}^n \Phi (\rho_{s,r}^n(f))- \nabla T_{s,t}^n \Phi (\rho_{s,r}^n( f)) \rangle_{H^{-\lambda-2}(\T^d)} \big] \big|  \\
    &\le C(n) \E\big[\sup\limits_{s \le t \le T} \norm{\rho_{s,t}(f)}_{{H}^{-\lambda-2}(\T^d))}^2 \big]^{\frac{1}{2}}  \E\big[\norm{\nabla T_{s+\epsilon,t}^n \Phi (\rho_{s,r}^n(f))- \nabla T_{s,t}^n \Phi (\rho_{s,r}^n( f)) }^2_{{H}^{-\lambda-2}(\T^d)}\big ]^{\frac{1}{2}}\\
    &\le C(n) \epsilon \norm{f}_{{H}^{-\lambda-2}(\T^d)} 
\end{align*}
where we used property\eqref{item: approx_bound_smaller_space}, the uniform bound~\eqref{eq: aux_uniform_bound} and Lemma~\ref{cor: con_semigroup_time}.
For the second term we find by similar computations 
\begin{align*}
    (II) &\le C\big( n,\norm{f}_{ H^{-\lambda-2}(\T^d)} \big) \sum\limits_{i=1}^n\norm{\varphi_i}_{ H^{-\lambda-2}(\T^d)}\\
    &\quad \cdot \E\big[ \big| g(\langle \rho_{s,t}^n(\rho_{s,r}^n(f)) , \varphi_1 \rangle_{H^{-\lambda-2}(\T^d)}  , \ldots, \langle \rho_{s,t}^n(\rho_{s,r}^n(f)) , \varphi_m \rangle_{H^{-\lambda-2}(\T^d)}  ) \\
    &\quad - g(\langle \rho_{s,t}^n(f) , \varphi_1 \rangle_{H^{-\lambda-2}(\T^d)}  , \ldots, \langle \rho_{s,t}^n(f) , \varphi_m \rangle_{H^{-\lambda-2}(\T^d)}  ) \big| \big] \\
    &\le  C\big( n,m,\norm{f}_{ H^{-\lambda-2}(\T^d)} \big) 
    \norm{g}_{C_b^2(\R^m)} \sum\limits_{j=1}^m \E\big[|\langle \rho^n_{s,t}(\rho_{s,r}^n(f))-\rho^n_{s,t}(f) , \varphi_j \rangle_{H^{-\lambda-2}(\T^d)} | \big]  \\
    &\le  C\big( n,m,\norm{f}_{ H^{-\lambda-2}(\T^d)} \big)  \E\big[\norm {\rho_{s,t}^n(\rho_{s,r}^n(f))-\rho_{s,t}^n(f)}_{H^{-\lambda-2}(\T^d)}^2\big]^{\frac{1}{2}}  \\
    &\le  C\big( n,m,\norm{f}_{ H^{-\lambda-2}(\T^d)} \big)  \E\big[\norm {\rho_{s,r}^n(f)-f}_{H^{-\lambda-2}(\T^d)}^2\big]^{\frac{1}{2}} \\
    &\le  C\big( n,m,\norm{f}_{ H^{-\lambda-2}(\T^d)} \big) \int\limits_s^{s+\epsilon} \E\big[\norm{A_n(r,\rho_{s,r}(f))}_{ H^{-\lambda-2}(\T^d) }^2 \big]^{\frac{1}{2}}\Id r   \\
    &\quad \quad + \sum\limits_{j=1}^d  \E \bigg[ \norm{\int\limits_{s}^{s+\epsilon} B_j(r) \Id W_j(r) }_{ H^{-\lambda-2}(\T^d)}^2 \bigg ]^{\frac{1}{2}} \\
    &\le  C\big( n,m,\norm{f}_{ H^{-\lambda-2}(\T^d)} \big) \bigg( \int\limits_s^{s+\epsilon} \E\big[\norm{\rho_{s,r}(f) }_{ H^{-\lambda-2}(\T^d)}^2\big] +1 \Id r \big)^{\frac{1}{2}} \\
    &\le C\big( n,m,\norm{f}_{ H^{-\lambda-2}(\T^d)} \big) \norm{f}_{ H^{-\lambda-2}(\T^d)} \epsilon^{\frac{1}{2}},
\end{align*}
where we used~\eqref{eq: stability_spde} in the fourth step, property \eqref{item: approx_bound_smaller_space}, the BDG-inequality in the sixth step and the bound ~\eqref{eq: aux_uniform_bound} in the last step.
For the third term we find 
\begin{align*}
    (III) &\le \E\bigg[ \norm{\nabla T_{s,t}^n\Phi(f) }_{H^{-\lambda-2}}\bigg( 
    \norm{j_n * \big( \nabla\cdot ( (b(r,\cdot,\mu_r)-b(s,\cdot,\mu_s)) j_n* \rho_{s,r}^n(f)  \big) }_{H^{-\lambda-2}(T^d)}\\
    &\quad + \norm{j_n*  \bigg\langle \nabla_x \cdot  \bigg(  \mu_r(\cdot )  \frac{\delta b}{\delta m} (r,\cdot ,\mu_r , v )- \mu_s(\cdot)  \frac{\delta b}{\delta m} (s,\cdot,\mu_s , v ) \bigg), j_n* \rho_{s,r}^n(v)  \bigg\rangle  \bigg\rangle }_{H^{-\lambda-2}(\T^d)} \bigg) \bigg] 
\end{align*}
Applying Lemma~\ref{lemma: product_distr} and Assumption~\ref{ass: coef_fokker} we obtain 
\begin{align*}
      &\norm{j_n * \big( \nabla\cdot ( (b(r,\cdot,\mu_r)-b(s,\cdot,\mu_s)) j_n* \rho_{s,r}^n(f)\big)  }_{H^{-\lambda-2}(T^d)}\big) \\
      &\quad \le C \norm{  (b(r,\cdot,\mu_r)-b(s,\cdot,\mu_s)) j_n* \rho_{s,r}^n(f) }_{H^{-\lambda-1}(T^d)} \\
      &\quad \le C(n)   \norm{\rho_{s,r}^n(f) }_{H{-\lambda-2}(T^d)} \norm{ b(r,\cdot,\mu_r)-b(s,\cdot,\mu_s)}_{H^{\lambda'}(\T^d)}.
\end{align*}
This implies 
\begin{align*}
    &\E\big[ \norm{\nabla T_{s,t}^n\Phi(f) }_{H^{-\lambda-2}(T^d)}\big( 
    \norm{\nabla\cdot ( (b(r,\cdot,\mu_r)-b(s,\cdot,\mu_s))  \rho_{s,r}^n(f) }_{H{-\lambda-2}(T^d)}\big)\big]  \\
    &\quad \le  C \norm{\nabla T_{s,t}^n\Phi(f) }_{H^{-\lambda-2}(T^d)}
    \norm{ b(r,\cdot,\mu_r)-b(s,\cdot,\mu_s)}_{H^{\lambda'}(\T^d)} \sup\limits_{s \le r \le T} \E\big[\norm{\rho_{s,r}^n(f) }_{H^{-\lambda-2}(T^d)}^2\big]^{\frac{1}{2}}. 
\end{align*}
The right hand side vanishes as \(\epsilon \to 0\) and, therefore, \(r \to s\) by continuity of \(b\) and Assumption~\ref{ass: coef_fokker}. For the flat derivative term we have 
\begin{align*}
    &\norm{j_n*  \bigg\langle \nabla_x \cdot  \bigg(  \mu_r(\cdot )  \frac{\delta b}{\delta m} (r,\cdot ,\mu_r , v )- \mu_s(\cdot)  \frac{\delta b}{\delta m} (s,\cdot,\mu_s , v ) \bigg), j_n* \rho_{s,r}^n(v)  \bigg\rangle  \bigg\rangle}_{H^{-\lambda-2}(\T^d)} \\
    &\quad 
    \le \norm{j_n*  \bigg\langle \nabla_x \cdot  \bigg(  (\mu_r(\cdot ) -\mu_s(\cdot)) \frac{\delta b}{\delta m} (s,\cdot ,\mu_s, v )\bigg), j_n* \rho_{s,r}^n(v)  \bigg\rangle  \bigg\rangle}_{H^{-\lambda-2}(\T^d)} \\
    &\quad \quad + \norm{j_n*  \bigg\langle \nabla_x \cdot  \bigg(  \mu_r(\cdot ) \Big(  \frac{\delta b}{\delta m} (r,\cdot ,\mu_r , v )-  \frac{\delta b}{\delta m} (s,\cdot,\mu_s , v )  \Big) \bigg), j_n* \rho_{s,r}^n(v)  \bigg\rangle  \bigg\rangle}_{H^{-\lambda-2}(\T^d)}
\end{align*}
For the first term we find 
\begin{align*}
    &\norm{j_n*  \bigg\langle \nabla_x \cdot  \bigg(  (\mu_r(\cdot ) -\mu_s(\cdot)) \frac{\delta b}{\delta m} (s,\cdot ,\mu_s, v )\bigg), j_n* \rho_{s,r}^n(v)  \bigg\rangle  \bigg\rangle}_{H^{-\lambda-2}(\T^d)} \\
    &\quad \le \norm{ \bigg\langle  \bigg(  (\mu_r(\cdot ) -\mu_s(\cdot)) \frac{\delta b}{\delta m} (s,\cdot ,\mu_s, v )\bigg), j_n* \rho_{s,r}^n(v)  \bigg\rangle  \bigg\rangle}_{H^{-\lambda-1}(\T^d)} \\
    &\quad \le C(n) \norm{  (\mu_r(x ) -\mu_s(x)) \norm{  \frac{\delta b}{\delta m} (s,x ,\mu_s,v)}_{H^{\lambda'}(\T^d)_v } }_{H^{-\lambda-2}(\T^d)_x}
    \norm{\rho_{s,r}^n }_{H^{-\lambda-2}(\T^d)}
\end{align*}
%
For the fourth term we obtain 
\begin{equation*}
    (IV) \le  C(n) \E\big[ \norm{\rho_{s,r}^n(f) -f }_{{H}^{-\lambda-2}(\T^d)}^2 \big]^{\frac{1}{2}} 
    \le C(n)\norm{f}_{ H^{-\lambda-2}(\T^d)}  \epsilon ,
\end{equation*}
where the second inequality follows from similar computations as in the second term. 
Combining all estimates for \((I),(II),(III),(IV)\) on the interval \(r \in [s, s+\epsilon]\), we obtain 
\begin{equation*}
     \lim_{\epsilon \to 0} \big|\E\big[ \langle A_n(r, \rho_{s,r}^n(f)), \nabla T_{s+\epsilon,t}\Phi (\rho_{s,r}^n(f)) \rangle_{{H}^{-\lambda-2}(\T^d)} - \langle A_n(s,f), \nabla T_{s,t}\Phi (f) \rangle_{{H}^{-\lambda-2}(\T^d)} \big] \big| = 0,
\end{equation*}
and, consequently, 
\begin{equation} \label{eq: ito_aux2}
\lim\limits_{\epsilon \to 0 } \frac{1}{\epsilon}\int\limits_{s}^{s+\epsilon} \E\big[ \langle A_n(r, \rho_{s,r}^n(f)), \nabla T_{s+\epsilon,t}^n\Phi (\rho_{s,r}^n(f)) \rangle_{{H}^{-\lambda-2}(\T^d)} \big] \Id r = \langle A_n(s,f), \nabla T_{s,t}^n\Phi (f) \rangle_{{H}^{-\lambda-2}(\T^d)} \big) .    
\end{equation}

For the second order term with the trace we have 
\begin{align*}
    &\big| \E\big[  \mathrm{Tr}\big( \nabla^2 T_{s+\epsilon,t}^n\Phi( \rho_{s,r}^n(f))  B_{j}(r) B_{j}^*(r) \big) \big] -  \E\big[  \mathrm{Tr}\big( \nabla^2 T_{s,t}^n( f)  B_{j}(s) B_{j}^*(s) \big) \big]\big| \\
    &\quad \le \big| \E\big[  \mathrm{Tr}\big( \big( \nabla^2 T_{s+\epsilon,t}^n\Phi( \rho_{s,r}^n(f))-\nabla^2 T_{s,t}^n\Phi( \rho_{s,t}^n(f))\big)   B_{j}(r) B_{j}^*(r) \big) \big] \big| \\
    &\quad \quad + \big| \E\big[  \mathrm{Tr}\big( \big( \nabla^2 T_{s,t}^n\Phi( \rho_{s,r}^n(f)) - T_{s,t}^n\Phi(f)\big)  B_{j}(r) B_{j}^*(r) \big) \big] \big| \\
    &\quad \quad + \big| \E\big[  \mathrm{Tr}\big( T_{s,t}^n\Phi(f)  \big( B_{j}(r)- B_{j}(s)\big)  B_{j}^*(r) \big) \big] \big| \\
    &\quad \quad + \big| \E\big[  \mathrm{Tr}\big( T_{s,t}^n\Phi(f)   B_{j}(s)  \big( B_{j}^*(r)- B_{j}^*(s) \big) \big] \big| \\
    &\quad =: (\tilde I) +(\tilde {II}) +(\tilde{III}) +(\tilde {IV} ) 
\end{align*}
for \(r \in [s,s+\epsilon]\). Recall the following inequality~\cite[Appendix~B]{LiuWei2015SPDE} for the trace 
\begin{equation} \label{eq: trace_ineq}
     |\mathrm{Tr}\big( \nabla^2 T_{s,t}^n\Phi(f)  B_{j}(r) B_{j}^*(r) \big)|  \le \norm{B_{j}(r)}^2_{L_2(L^2(\T^d),  H^{-\lambda-2}(\T^d))} \norm{\nabla^2 T_{s,t}^n\Phi(f)}_{L({H}^{-\lambda-2}(\T^d),{H}^{-\lambda-2}(\T^d)) }.
\end{equation}
This inequality, in combination with Lemma~\ref{cor: con_semigroup_time} proves 
\begin{equation*}
    (\tilde I) \le C(n) \sup\limits_{s \le r \le T} \norm{B_{j}(r)}^2_{L_2(L^2(\T^d),  H^{-\lambda-2}(\T^d))}  \epsilon
    \le C(n) \epsilon . 
\end{equation*}
For the second term, we find 
\begin{align*}
    (\tilde {II}) \le C \E\big[ \norm{\nabla^2 T_{s,t}^n\Phi(\rho_{s,r}(f))-\nabla^2 T_{s,t}^n\Phi(f)}_{L({H}^{-\lambda-2}(\T^d),{H}^{-\lambda-2}(\T^d)) }\big] .
\end{align*}
By Remark~\ref{remark: infinite_diff} the map \(f \mapsto T_{s,t}^n\Phi(f)\) is infinitely many times differentiable with bounded derivatives. Hence, applying the mean-value theorem we find 
\begin{align*}
    (\tilde {II})  \le C \norm{T_{s,t}^n\Phi}_{C_b^3(H^{-\lambda-2}(\T^d)} \E\big[\norm{\rho_{s,r}(f)- f }_{H^{-\lambda-2}(\T^d)}\big] 
    \le C(n)\norm{T_{s,t}^n\Phi}_{C_b^3(H^{-\lambda-2}(\T^d)} \epsilon^{\frac{1}{2}},
\end{align*}
where the last step follows by similar computations as in for the term \((II)\). 
For the third term \((\tilde{III})\) we utilize again~\eqref{eq: trace_ineq} to obtain
\begin{align*}
    (\tilde{III})
    &\le C \norm{B_{j}(r)-B_{j}(s)}^2_{L_2(L^2(\T^d),  H^{-\lambda-2}(\T^d))} \norm{B^*_{j}(r)}^2_{L_2(L^2(\T^d),  H^{-\lambda-2}(\T^d))} \\
    &\le C \norm{B^*_{j}(r)-B^*_{j}(s)}^2_{L_2(L^2(\T^d),  H^{-\lambda-2}(\T^d))} \sup\limits_{0 \le r \le T} \norm{B^*_{j}(r)}^2_{L_2(L^2(\T^d),  H^{-\lambda-2}(\T^d))} \\
    &\le C \norm{\sqrt{\mu_r}-\sqrt{\mu_s}}_{L^2(\T^d)} \sup\limits_{0 \le r \le T}  \norm{\sqrt{\mu_r}}_{L^2(\T^d)} \\
    &\le C \norm{\sqrt{\mu_r}-\sqrt{\mu_s}}_{L^2(\T^d)} ,
\end{align*}
where we used~\eqref{eq: hilbert_schmidt_b} in the third step. 
By Assumption~\ref{ass: coef_fokker} the last term vanishes as \(r \to s \). 
The fourth term \((\tilde{IV})\) follows by analogous computations, which proves the continuity. Hence, we find 
\begin{equation*}
    \lim\limits_{\epsilon \to 0 } \frac{1}{2\epsilon} \sum\limits_{j=1}^d 
        \int\limits_{s}^{s+\epsilon}  \E\big[ \mathrm{Tr}\big( \nabla^2 T_{s+\epsilon,t}^n \Phi ( \rho_{s,r}^n(f))  B_{j}(r) B_{j}^*(r) \big) \big] \Id r  =  \mathrm{Tr}\big( \nabla^2 T_{s,t}^n \Phi (f)  B_{j}(s) B_{j}^*(s) \big)   . 
\end{equation*}
Combining it with equalities~\eqref{eq: ito_aux1},~\eqref{eq: ito_aux2} proves the equation stated in the Lemma. 

 For the continuity of \(f \mapsto \partial_s T_{s,t}^n \Phi(f) \) we notice that \(\Phi\) is smooth, \(A_n\) is bounded and continuous in the space variable and the trace operator has a continuity property by inequality~\eqref{eq: trace_ineq}. 
\end{proof}

In the following, it is important to view the derivative as a function of $x$, in $H^{\lambda +2}(\T^d)$. We define
\begin{equation}\label{eq: hat_first}
    \nabla \widehat T^{n}_{s,t}\Phi(f) (x)
:= \sum\limits_{k \in \Z^d} \langle k \rangle^{-2(\lambda+2)} \overline{ \langle \nabla T_{s,t}^n \Phi(f) ,e_k \rangle }  e_k(x).   
\end{equation}
A simple computation and duality~\cite[Proposition~1.58]{Bahouri2011} provides the following: 
\begin{lemma} \label{cor: dual_identification_first_derivative}
Let \(\Phi \in  FC^\infty(H^{-\lambda-2}(\T^d))\) and \(f,h \in H^{-\lambda-2}(\T^d)\), then
    \begin{equation}
    \label{eq:identity_first_der}
        \langle \nabla T_{s,t}^n \Phi(f)  , h \rangle_{H^{-\lambda-2}(\T^d)}
        =  \langle \nabla \hat T_{s,t}^n \Phi(f)  , h \rangle. 
    \end{equation}
    As a consequence, for any \(f \in H^{-\lambda-2}(\T^d)\)
     \begin{equation}
     \label{eq:reg_1_Phi}
    \norm{\nabla \hat T_{s,t}^n\Phi ( f) }_{H^{\lambda+2}(\T^d)}
    =   \norm{\nabla T_{s,t}^n \Phi(f) }_{H^{-\lambda-2}(\T^d)} \le C [\Phi]_{C^1}.  
    \end{equation}
    Moreover, the function \(x \mapsto \nabla \hat T_{s,t}^n\Phi (f)(x) \in B_{\infty,\infty}^{\lambda+2-d/2}( \T^d)\) and 
    \begin{align}
    \nonumber
        &\sup\limits_{n \in \N} \sup\limits_{f \in H^{-\lambda-2}(\T^d)} \norm{\nabla \hat T_{s,t}^n\Phi ( f) }_{B_{\infty,\infty}^{\lambda+2-d/2}(\T^d)} \\
        \label{cor: besov_reg_first_order_function}
        &\quad \le C
        \sup\limits_{n \in \N} \sup\limits_{f \in H^{-\lambda-2}(\T^d)} \norm{\nabla \hat T_{s,t}^n\Phi ( f) }_{H^{\lambda+2}(\T^d)}  \le   C [\Phi]_{C^1}  .  
    \end{align}
\end{lemma}
\noindent Note that the latter claim follows by Lemma~\ref{cor: reg_derivative} and the Sobolev embedding~\cite[Corollary~3.5.3]{Schmeisser1987}.

It is important to write also the second derivative  as a function of two variables: for \(\Phi \in  FC^\infty(H^{-\lambda-2}(\T^d))\) we define 
\begin{align} \label{eq: def_second_derivative_function}
    \nabla^2 \hat T_{s,t}^n \Phi(f)(x,y)  &:=  \sum\limits_{i_1,i_2=1}^m \E\big[\partial_{x_{i_1}}\partial_{x_{i_2}} g( \langle \rho_{s,t}^n( f), \varphi_1\rangle_{H^{-\lambda-2}} , \ldots,\langle \rho_{s,t}^n(f), \varphi_m  \rangle_{H^{-\lambda-2}} )\big] \nonumber  \\
    &\quad \cdot \sum\limits_{k,l \in \Z^d} \langle k \rangle^{-2(\lambda+2)}\langle l \rangle^{-2(\lambda+2)} e_k(x)e_l(y) \overline{\langle (\mathsf{Y}_{s,t}^n)^* \varphi_{i_1}, e_k \rangle }  \overline{
    \langle  (\mathsf{Y}_{s,t}^n)^* \varphi_j , e_l \rangle }.
   \end{align}
   
\noindent Our aim is to derive an alternative representation for the term involving the trace.

\begin{lemma} \label{lemma: generatpr_trace_computation}
For \(\Phi \in  FC^\infty(H^{-\lambda-2}(\T^d))\) we have  
\begin{equation*}
    \frac{1}{2} \sum\limits_{j=1}^d   \mathrm{Tr}\big( \nabla^2 T_{s,t}^n\Phi(f) B_j(s) B_j^*(s) \big)  = \frac{\sigma^2}{2} \sum\limits_{j=1}^d   \int_{\T^d}  \partial_{x_j}  \partial_{y_j} \nabla^2 \hat T_{s,t}^n \Phi(f)( x, y) \big|_{x=y} \Id \mu_s(x)  
\end{equation*} 
\end{lemma}

\begin{proof} 
We start with the cyclic property~\cite[Proposition~B.0.10]{LiuWei2015SPDE}, 
\begin{equation*}
    \mathrm{Tr}\big( \nabla^2 T_{s,t}^n \Phi(f) B_j(s) B_j^*(s) \big)
    = \mathrm{Tr}\big(  B_j^*(s) \nabla^2 T_{s,t}^n  \Phi(f) B_j(s)\big).
\end{equation*}
Recalling the representation\eqref{eq: representation_cylindrical}, denote by \begin{equation} \label{eq: tau_def}
    \tau_{i_1,i_2} = \E\big[\partial_{x_{i_1}}\partial_{x_{i_2}} g( \langle \rho_{s,t}^n( f), \varphi_1\rangle_{H^{-\lambda-2}} , \ldots,\langle \rho_{s,t}^n(f), \varphi_m  \rangle_{H^{-\lambda-2}} )\big]
    \end{equation}
Let \(u \in L^2(\T^d)\), then 
\begin{align*}
    &\langle \nabla^2 T_{s,t}^n \Phi(f)B_j(s)(u), e_k \rangle \\
    &\quad =\sum\limits_{i_1,i_2=1}^m \tau_{i_1,i_2} \langle (\mathsf{Y}_{s,t}^n)^* \varphi_{i_1}, e_k \rangle \langle  (\mathsf{Y}_{s,t}^n)^* \varphi_{i_2}, B_j(s) u  \rangle_{H^{-\lambda-2}} \\
    &\quad = -\sigma \sum\limits_{i_1,i_2=1}^m \tau_{i_1,i_2} \langle (\mathsf{Y}_{s,t}^n)^* \varphi_{i_1}, e_k \rangle  \sum\limits_{l \in \Z^d} \langle l \rangle^{-2(\lambda+2)} \overline{ \langle \sqrt{\mu_s} u , \partial_{x_{i_2}} e_l \rangle_{L^2(\T^d)} }
    \langle  (\mathsf{Y}_{s,t}^n)^* \varphi_{i_2} , e_l \rangle 
\end{align*}
Now, recall that 
\begin{equation*}
    B_j^*(s)(f) = \sigma \sum\limits_{k \in \Z^d} \langle k \rangle^{-2(\lambda+2)} \overline{\langle f, \partial_{x_j} e_k \rangle} \sqrt{\mu_s} e_k 
     = -\sigma \sum\limits_{k \in \Z^d} \langle k \rangle^{-2(\lambda+2)} 2\pi i k_j \overline{\langle f,  e_k \rangle} \sqrt{\mu_s}  e_k 
\end{equation*}
Consequently, we find 
\begin{align*}
     &B_j^*(s) \nabla^2 T_{s,t}^n  \Phi(f) B_j(s)(u)(x) \\
     &\quad = -\sigma \sum\limits_{k \in \Z^d} \langle k \rangle^{-2(\lambda+2)} 2\pi i k_j \overline{ \langle  \nabla^2 T_{s,t}^n  \Phi(f) B_j(s),  e_k \rangle} \sqrt{\mu_s}(x)e_k(x) \\
    &\quad = \sigma^2 \sum\limits_{k \in \Z^d} \langle k \rangle^{-2(\lambda+2)} \sqrt{\mu_s}(x)  \partial_{x_{j}} e_k(x) \\ 
    &\quad \quad \sum\limits_{i_1,i_2=1}^m \tau_{i_1,i_2} \overline{\langle (\mathsf{Y}_{s,t}^n)^* \varphi_{i_1}, e_k \rangle}  \sum\limits_{l \in \Z^d} \langle l \rangle^{-2(\lambda+2)} \langle \sqrt{\mu_s} u , \partial_{x_{j}} e_l \rangle_{L^2(\T^d)} \overline{
    \langle  (\mathsf{Y}_{s,t}^n)^* \varphi_{i_2} , e_l \rangle } \\
    &\quad =  \int_{\T^d} \sqrt{\mu_s(x)} \sqrt{\mu_s(y)} u(y)  \partial_{x_{j}}  \partial_{y_{j}} \nabla^2 \hat T_{s,t}^n \Phi(f)(x,y) \Id y . 
\end{align*}
The change of integral, derivative and series is permitted, since 
\begin{align} \label{eq: continuity_represent}
\begin{split}
    &\norm{\sum\limits_{l\in \Z^d} \langle l \rangle^{-2(\lambda+2)}e_l \overline{ (\mathsf{Y}_{s,t}^n)^* \varphi_i, e_l \rangle }  }_{H^{\lambda+1}(\T^d)}\\
    &\quad = \sum\limits_{k\in \Z^d } \langle k \rangle^{2(\lambda+1)} \bigg| \bigg\langle \sum\limits_{l\in \Z^d} \langle l \rangle^{-2(\lambda+2)}e_l \overline{\langle(\mathsf{Y}_{s,t}^n)^* \varphi_i, e_l \rangle } , e_k \bigg \rangle\bigg|^2 \\
    &\quad = \sum\limits_{k\in \Z^d } \langle k \rangle^{2(\lambda+1)} \big| \big\langle \langle k \rangle^{-2(\lambda+2)}e_k \overline{\langle  (\mathsf{Y}_{s,t}^n)^* \varphi_i, e_k \rangle } , e_k \big \rangle\big|^2 \\
    &\quad = \sum\limits_{k\in \Z^d }\langle k \rangle^{-2(\lambda+3)} \big|  \langle  (\mathsf{Y}_{s,t}^n)^*\varphi_i , e_k \rangle \big|^2
    \le \norm{(\mathsf{Y}_{s,t}^n)^* \varphi_i }_{H^{-\lambda-3}(\T^d)}
     \le C \norm{\varphi_i }_{H^{-\lambda-2}(\T^d)}.
     \end{split}
\end{align}
The last term is bounded and by the Sobolev embedding the function 
\begin{equation*}
    \partial_{y_j}  \sum\limits_{l\in \Z^d} \langle l \rangle^{-2(\lambda+2)}e_l(y) \overline{\langle(\mathsf{Y}_{s,t}^n(f))^* \varphi_i, e_l \rangle }
\end{equation*}
is continuous and, consequently, the product is continuous. 
Hence, $B_j(t)^*\nabla^2 T_{s,t}^n \Phi(f)B_j(t) $ is an integral operator (from $L^2(\T^d)$ to itself), which is of trace class by construction of \(B_j\). 
Applying~\cite[page~102, Proposition~3.1]{Duflo1972}, we obtain the claim. 
\end{proof}

We state a useful result on the identification of the second derivative, as well as estimates on the Sobolev norms.

\begin{lemma}\label{cor: besov_reg_second_order_function}
    Let \(\Phi \in FC^\infty(H^{-\lambda-2}(\T^d))\). For any $h, \tilde{h}\in H^{-\lambda-2}(\T^d)$ we have 
    \footnote{The notation $h\otimes \tilde{h}$ denotes the distribution 
    on the product space, such that $\langle\varphi(x,y),h\otimes\tilde h\rangle = \langle \langle \varphi(x,y), h(x) \rangle h(y) \rangle$.
    }
    \begin{equation}
    \label{eq:second_der}
    \langle \nabla^2 T_{s,t}^n\Phi ( f) (\tilde{h}), h \rangle_{H^{-\lambda-2}(\T^d)} 
    =\langle \nabla^2 \hat T_{s,t}^n\Phi ( f), h \otimes \tilde h \rangle.
    \end{equation}
     The function \((x,y) \mapsto \nabla^2 \widehat T_{s,t}^n\Phi ( f)(x,y)\) belongs to 
    \(H^{\tilde \lambda+2-d/2}(\T^d \times \T^d)\) for all \(\tilde \lambda \in (3d/2, \lambda)\). More precisely, we have the estimate 
    \begin{equation}
    \label{eq:reg_2_Phi}
        \sup\limits_{n \in \N} \sup\limits_{f \in H^{-\lambda-2}(\T^d)} \norm{\nabla^2 \hat T_{s,t}^n\Phi ( f) }_{H^{\tilde \lambda+2-d/2}(\T^d\times \T^d)} \le C [\Phi]_{C^2}  .  
    \end{equation}
    for a constant $C$ depending on  $\tilde\lambda$. Moreover,
    \begin{equation}
    \label{ineq:diag}
          \sup\limits_{n \in \N} \sup\limits_{f \in H^{-\lambda-2}(\T^d)}  \norm{\partial_{x_j}  \partial_{y_j} \nabla^2 \hat T_{s,t}^n \Phi(f)( x, y) \big|_{x=y}}_{H^{\tilde \lambda-d}(\T^d)}
         \le C [\Phi]_{C^2}  . 
    \end{equation}
\end{lemma}

\begin{proof}

On the one hand, we have 
    \begin{equation}
    \begin{split}
    \langle \nabla^2 T_{s,t}^n\Phi ( f)(\tilde h), h \rangle_{H^{-\lambda-2}(\T^d)}
    &=   \sum\limits_{i,j=1}^m  \E\big[\partial_{x_i}\partial_{x_j} g( \langle \rho_{s,t}^n( f), \varphi_1\rangle_{H^{-\lambda-2}} , \ldots,\langle \rho_{s,t}^n(f), \varphi_m  \rangle_{H^{-\lambda-2}} )\big]\\
    &\quad \quad \cdot \langle   (\mathsf{Y}_{s,t}^n)^* \varphi_j, \tilde  h \rangle_{H^{-\lambda-2}} \langle  (\mathsf{Y}_{s,t}^n)^* \varphi_i , h \rangle_{H^{-\lambda-2} }  .
    \end{split}
    \end{equation}
On the other hand, assuming the additional regularity \(h, \tilde h \in C^\infty(\T^d)\), we find 
\begin{align*}
   & \langle \nabla^2 \hat T_{s,t}^n\Phi ( f), h \otimes \tilde h \rangle  \\
    &\quad =  \sum\limits_{i,j=1}^m \E\big[\partial_{x_i}\partial_{x_j} g( \langle \rho_{s,t}^n( f), \varphi_1\rangle_{H^{-\lambda-2}} , \ldots,\langle \rho_{s,t}^n(f), \varphi_m  \rangle_{H^{-\lambda-2}} )\big] \\ 
    &\quad \quad \cdot \sum\limits_{k,l\in\Z^d}
\langle e_k (\cdot_1) e_l(\cdot_2) , h \otimes \tilde h(\cdot_1,\cdot_2) \rangle  \overline{\langle  (\mathsf{Y}_{s,t}^n(f))^* \varphi_j, e_l \rangle \langle    (\mathsf{Y}_{s,t}^n(f))^* \varphi_i , e_k \rangle} \langle k\rangle^{-2(\lambda+2)} \langle l\rangle^{-2(\lambda+2)}  \\
   &\quad =  \sum\limits_{i,j=1}^m  \E\big[\partial_{x_i}\partial_{x_j} g( \langle \rho_{s,t}^n( f), \varphi_1\rangle_{H^{-\lambda-2}} , \ldots,\langle \rho_{s,t}^n(f), \varphi_m  \rangle_{H^{-\lambda-2}} )\big]
 \\   &\quad \quad \cdot \langle (\mathsf{Y}_{s,t}^n(f))^* \varphi_j, \tilde h \rangle_{H^{-\lambda -2}} \langle  (\mathsf{Y}_{s,t}^n(f))^* \varphi_i , h \rangle_{H^{-\lambda -2}}  .
\end{align*}
Consequently, we obtain \eqref{eq:second_der} and thus
\begin{align*}
    &\sup\limits_{\norm{ \tilde h}_{H^{-\lambda -2}(\T^d) } \le 1   } \sup\limits_{\norm{ h}_{H^{-\lambda -2}(\T^d) } \le 1} 
    |\langle \nabla^2 \hat T_{s,t}^n\Phi ( f), h \otimes \tilde h \rangle |\\ 
    &\le \sup\limits_{\norm{ \tilde h}_{H^{-\lambda -2}(\T^d) } \le 1   } \norm{\nabla^2 T_{s,t}^n\Phi ( f)(\tilde h)}_{H^{-\lambda-2}(\T^d)}
    \le  C [\Phi]_{C^2}, 
\end{align*}
where we used Lemma~\ref{cor: reg_derivative}.
In particular, for \(k = (k_1,k_2) \in \Z^d \times \Z^d\) we have 
\begin{align*}
    |\langle \nabla^2 \hat T_{s,t}^n\Phi ( f), e_k  \rangle |^2
    &\le (\langle k_1 \rangle\langle k_2 \rangle)^{-2(\lambda +2)}  \sup\limits_{k_1,k_2 \in \Z^d } |\langle \nabla^2 \hat T_{s,t}^n\Phi ( f), \langle k_1 \rangle^{\lambda + 2} e_{k_1} \otimes \langle k_2 \rangle^{\lambda + 2} e_{k_2}  \rangle |^2 \\
    &\le C \langle k_1 \rangle^{-2(\lambda+2)} \langle k_2 \rangle^{-2(\lambda +2)} [\Phi]_{C^2}^2,
\end{align*}
where we used the fact that \(\langle k_j \rangle^{\lambda + 2} \norm{e_{k_j}}_{H^{-\lambda-2}(\T^d)} =  1   \) for \(k \in \Z^d , j = 1,2\). 
Consequently, let \(\epsilon > 0\) we obtain 
\begin{align*}
&\norm{\nabla^2 \hat  T_{s,t}^n\Phi ( f)}_{H^{\lambda+2-d/2-\epsilon/2}(\T^d \times \T^d)}^2 \\
&\quad = \sum\limits_{k \in \Z^{d} \times \Z^d} \langle k \rangle^{2(\lambda+2-d/2-\epsilon/2)}
|\langle \nabla^2 \hat T_{s,t}^n\Phi ( f), e_k  \rangle |^2 \\
&\quad \le C [\Phi]_{C^2}^2
\sum\limits_{k_1 \in \Z^d }\sum\limits_{k_2 \in \Z^d }
\langle k_1 \rangle^{2(\lambda+2-d/2-\epsilon/2)- 2(\lambda+2)}\langle k_2 \rangle^{2(\lambda+2-d/2-\epsilon/2)- 2(\lambda+2)} \\
&\quad \le   C [\Phi]_{C^2}^2 \sum\limits_{k_1 \in \Z^d }\sum\limits_{k_2 \in \Z^d } 
\langle k_1 \rangle^{-d-\epsilon} \langle k_2 \rangle^{-d-\epsilon}
\le    C [\Phi]_{C^2}^2 . 
\end{align*}
Taking supremum over \(f \in H^{-\lambda-2}(\T^d) \) and \(n \in \N\) proves the first claim.
The second claim follows by Lemma~\ref{lemma: diagonal}. 
\end{proof}

Putting together Proposition \ref{lemma: ito_formula_approx} with Lemmas and \ref{cor: dual_identification_first_derivative} and \ref{lemma: generatpr_trace_computation} we obtain the following expression for the generator of $\rho^n$, with respect to the initial time. 

\begin{proposition}
\label{prop:4.9}
For any \(\Phi \in  FC^\infty(H^{-\lambda-2}(\T^d))\) and $0\leq s<t\leq T$, we have
\begin{equation}
\mathcal{G}_s^n T_{s,t}^n\Phi(f) = 
 \frac{\sigma^2}{2} \sum\limits_{j=1}^d   \int_{\T^d}  \partial_{x_j}  \partial_{y_j} \nabla^2 \hat T_{s,t}^n \Phi(f)( x, y) \big|_{x=y} \Id \mu_s(x) 
 + \langle A^n(s,f) , \nabla \widehat{T}_{s,t}^n\Phi (f)) \rangle
\end{equation} 
\end{proposition}

\subsection{On the notion of derivative} \label{subsec: notion_derivative}
The aim here is to provide a link between the flat derivative on $\mathcal{P}(\T^d)$ and the Fréchet derivative in \(H^{-\lambda-2}(\T^d)\). 
We provide it for smooth function \(\Phi \in FC^\infty(H^{-\lambda-2}(\T^d)) \), which allows us to change infinite sums and derivatives without problems, and actually we write for $T^n_{s,t}\Phi$ because it is the function for which  we characterize above its derivatives as regular functions of $x$, and also it is what is needed need below in the proof of the main result. 

\begin{lemma} \label{lemma: notion_of_derivatives}
Let $\Phi \in FC^\infty(H^{-\lambda-2}(\T^d))$. Then, when restricted to probability measures, we have, for any $m\in \mathcal{P}(\T^d)$,
\begin{align}
    \nabla \hat T_{s,t}^n \Phi(m)( x) &=\tfrac{\delta}{\delta m} T_{s,t}^n \Phi(m; x)   \qquad \forall m\in \mathcal{P}(\T^d), x\in \T^d, \label{eq: first_flat_derivatives} \\ 
   \nabla^2\hat T_{s,t}^n \Phi(m) (x, y) &= \tfrac{\delta^2}{\delta m^2}T_{s,t}^n  \Phi(m; x,y)  \qquad \forall m\in \mathcal{P}(\T^d), x,y\in \T^d. \label{eq: second_flat_derivatives}
\end{align}
Moreover, the functions 
\begin{align*}
\mathcal{P}(\T^d) \ni m &\mapsto \tfrac{\delta}{\delta m} T_{s,t}^n \Phi(m; \cdot) \in H^{\lambda+2}(\T^d) \\ 
\mathcal{P}(\T^d) \ni m &\mapsto \tfrac{\delta^2}{\delta m^2} T_{s,t}^n \Phi(m; \cdot,\cdot) \in H^{\tilde\lambda+2-d/2}(\T^d\times \T^d),
\end{align*} 
are continuous, for all $\tilde{\lambda}\in (\tfrac32 d, \lambda)$.  
\end{lemma}

\begin{proof}
We recall that 
\begin{align*}
    \lim_{\epsilon \to 0} \frac{T_{s,t}^n \Phi(m+\epsilon (\tilde{m}-m))-\Phi(m)}{\epsilon}
    &= \int_{\T^d} \frac{\delta}{\delta m} T_{s,t}^n \Phi(m;x) (\tilde{m}-m)(dx), \qquad m, \tilde{m}\in \mathcal{P}(\T^d) \\
    \lim_{\epsilon \to 0} \frac{T_{s,t}^n \Phi(f+\epsilon h)-T_{s,t}^n\Phi(f)}{\epsilon}
    &= \Big\langle \nabla T_{s,t}^n \Phi(f), h \Big\rangle_{H^{-\lambda-2}(\T^d)} , \qquad f,h\in H^{-\lambda-2}(\T^d).
\end{align*}
Now identity \eqref{eq:identity_first_der}, choosing \(f=m, \; h = \tilde m -m \) provides~\eqref{eq: first_flat_derivatives}. 
For the second identity~\eqref{eq: second_flat_derivatives}, we recall that 
\begin{equation*}
\lim_{\epsilon \to 0} \frac{\tfrac{\delta}{\delta m}T_{s,t}^n\Phi(m+\epsilon (\tilde{m}-m);x)-\tfrac{\delta}{\delta m}T_{s,t}^n\Phi(m;x)}{\epsilon}
    = \int_{\T^d} \frac{\delta^2}{\delta m^2} T_{s,t}^n\Phi(m;x,y) (\tilde{m}-m)(dy),
\end{equation*}
for \( m, \tilde{m}\in \mathcal{P}(\T^d), x\in\T^d\). 
On the other hand, utilizing the explicit formula~\eqref{eq: hat_first} and Lemma~\ref{lemma: derivative_test_semigroup}, we have 
\begin{align*}
&    \lim_{\epsilon \to 0} \frac{\nabla \hat T_{s,t}^n \Phi(f+\epsilon h)(x)-\nabla  \hat T_{s,t}^n  \Phi(f)(x)}{\epsilon} \\
& \quad = \lim_{\epsilon \to 0} \sum\limits_{k \in \Z^d } \langle k \rangle^{-2(\lambda+2)} \epsilon^{-1}
\overline{ \langle \nabla T_{s,t}^n \Phi(f+\epsilon h)-\nabla  T_{s,t}^n  \Phi(f), e_k \rangle } e_k(x) \\
&\quad  = \lim_{\epsilon \to 0} \sum\limits_{k \in \Z^d } \sum\limits_{i=1}^m  \langle k \rangle^{-2(\lambda+2)}\epsilon^{-1}
\Big( \E\big[\partial_{x_i} g\big( \langle \rho_{s,t}^n(f+\epsilon h), \varphi_1 \rangle_{H^{-\lambda-2}} ,\ldots , \langle \rho_{s,t}^n(f+\epsilon h), \varphi_m \rangle_{H^{-\lambda-2}} \big)  \\
 &\quad \quad  - \partial_{x_i} g\big( \langle \rho_{s,t}^n(f), \varphi_1 \rangle_{H^{-\lambda-2}} ,\ldots , \langle \rho_{s,t}^n(f), \varphi_m \rangle_{H^{-\lambda-2}} \big) \big] \langle \overline{(\mathsf{Y}_{s,t}^n)^* \varphi_i, e_k \rangle} e_k(x) .
\end{align*}
Now, an application of the chain rule provides
\begin{align*}
    &\lim_{\epsilon \to 0} \frac{\nabla \hat T_{s,t}^n \Phi(f+\epsilon h)(x)-\nabla  \hat T_{s,t}^n  \Phi(f)(x)}{\epsilon} \\
    &\quad =\sum\limits_{k \in \Z^d } \sum\limits_{i,j=1}^m  \langle k \rangle^{-2(\lambda+2)} \E\big[\partial_{x_i}\partial_{x_j} g( \langle \rho_{s,t}^n( f), \varphi_1\rangle_{H^{-\lambda-2}} , \ldots,\langle \rho_{s,t}^n(f), \varphi_m  \rangle_{H^{-\lambda-2}} )\big] \\
    &\quad\quad   \cdot \langle   (\mathsf{Y}_{s,t}^n)^* \varphi_j, h \rangle_{H^{-\lambda-2}} \overline{ \langle (\mathsf{Y}_{s,t}^n)^* \varphi_i, e_k \rangle }  e_k(x) \\
&\quad = \langle \nabla^2 \hat T_{s,t}^n \Phi(f)(x,\cdot) , h \rangle ,
\end{align*}
where we used again~\cite[Proposition~1.58]{Bahouri2011} to make sense of the last line and it should be understood as the action of \(h\) on the second variable of \( \nabla^2 \hat T_{s,t}^n \Phi(f)\) for fixed \(x \in \T^d\). 
Hence, ~\eqref{eq: second_flat_derivatives} follows if we let $f=m$ and $h=\tilde{m}-m$.
To prove the last claim, recall first that the maps 
\begin{align*}
H^{-\lambda-2}(\T^d)\ni f &\mapsto \nabla  T_{s,t}^n \Phi(f) \in H^{-\lambda-2}(\T^d), \\
H^{-\lambda-2}(\T^d)\ni f &\mapsto \nabla^2  T_{s,t}^n \Phi(f) \in L(H^{-\lambda-2}(\T^d), H^{-\lambda-2}(\T^d))
\end{align*}
are continuous. Thus the claim follows since the embedding $\mathcal{P}(\T^d) \hookrightarrow H^{-\lambda-2}(\T^d)$ is continuous ($\mathcal{P}(\T^d)$ being equipped with $W_1$, and the maps 
\begin{align*}
H^{-\lambda-2}(\T^d)\ni T_{s,t}^n \Phi(f) &\mapsto \hat{T}_{s,t}^n \Phi(f) \in H^{\lambda2}(\T^d), \\
L(H^{-\lambda-2}(\T^d), H^{-\lambda-2}(\T^d))\ni \nabla^2  T_{s,t}^n \Phi(f) &\mapsto \nabla^2  \hat{T}_{s,t}^n \Phi(f) \in H^{\tilde\lambda+2-d/2}(\T^d\times \T^d)
\end{align*}
are linear and continuous by \eqref{eq:reg_1_Phi} and (the proof of) 
\eqref{eq:reg_2_Phi}.
\end{proof}

\subsection{Generator of the fluctuation process}
\label{sec:generator_fluctuation}

We compute the generator of the fluctuation process in \eqref{eq: def_fluctuation_process}
\begin{equation*}
    \rho_t^N := \sqrt{N}(\mu_t^N-\mu_t) , 
\end{equation*}
on the particular test function which we require.
Recall that we have a weak solution to the particle system \eqref{eq: interacting_particle_system} by Assumption \ref{ass: existence}.
We crucially utilize the results from Section~\ref{subsec: notion_derivative} to calculate the generator of  \(\rho_t^N\) and then to connect it to the generator of the classical SPDE~\eqref{eq: limiting_spde}.

We recall how the derivative of functions of probability measures behave when evaluated on empirical measures. For $\bm{x}=(x_1,\dots,x_N)\in (\T^d)^N$, denote $\mu^N_{\bm{x}}= \frac1N \sum_{i=1}^N \delta_{x_i}$.
For a function $U:\mathcal{P}_2(\T^d)\rightarrow \R$ let $u^N: (\T^d)^N\rightarrow \R$ given by $u^N(\bm{x})=U(\mu^N_{\bm{x}})$. 
We say that $U$ is \emph{fully} $C^2$ on $\mathcal{P}(\T^d)$ if it has two flat derivatives and the functions
\begin{align*}
\mathcal{P}(\T^d) \ni m &\mapsto \tfrac{\delta}{\delta m} T_{s,t}^n \Phi(m; \cdot) \in C^2(\T^d) \\ 
\mathcal{P}(\T^d) \ni m &\mapsto \tfrac{\delta^2}{\delta m^2} T_{s,t}^n \Phi(m; \cdot,\cdot) \in C^2(\T^d\times \T^d),
\end{align*} 
are continuous. In this case, \cite[Proposition 3.1]{chassagneux2022probabilistic} states that  
 $u^N\in C^2(\T^{Nd})$ and we have
\begin{align}
    \partial_{x_i} u^N(\bm{x})&= \frac1N \partial_\mu U(\mu^N_{\bm{x}}; x_i), \\
    \partial^2_{x_i x_i} u^N(\bm{x}) &=
    \frac1N D_x \partial_\mu U(\mu^N_{\bm{x}}; x_i)
    +\frac{1}{N^2} \partial^2_{\mu\mu} U(\mu^N_{\bm{x}}; x_i, x_i), \\
    \partial^2_{x_i x_j} u^N(\bm{x}) &=
    \frac{1}{N^2} \partial^2_{\mu\mu} U(\mu^N_{\bm{x}}; x_i, x_j), \quad i\neq j, 
\end{align}
where we denote the Lions' derivatives
\[
\partial_\mu U (m;x) = D_x \frac{\delta}{\delta m} U(m;x), \qquad
\partial^2_{\mu\mu} U (m;x,y) = D_x D_y \frac{\delta^2}{\delta m^2} U(m;x,y).
\]

\noindent We also recall that, under the same condition, the process $(\mu_t)_t$ is deterministic and  satisfies
\begin{equation*}
    \tfrac{d}{dt} U(\mu_t) = 
    \int_{\T^d} \Big( b(t,x, \mu_t) \cdot \partial_\mu U(\mu_t; x) 
    + \frac{\sigma^2}{2} 
    \Tr[  D_x \partial_\mu U(\mu_t; x) ] 
    \Big) \mu_t(dx).
\end{equation*}
thanks to It\^o formula for flows of probability measures; 
see \cite[Theorem 3.3]{chassagneux2022probabilistic}.

Note that $T^n_{s,t} \Phi$ has the required full $C^2$ regularity on $\mathcal{P}(\T^d)$ thanks to the regularity in Lemma \ref{lemma: notion_of_derivatives} and the Sobolev embeddings. We can now compute the (forward) generator of the fluctuation process.

\begin{proposition}
\label{prop:4_11}
    For any \(\Phi \in FC^\infty(H^{-\lambda-2}(\T^d))\), any (random) initial condition $\rho^N_0$ and for any $0\leq t_1 <t_2\leq t_3\leq T$, we have 
\begin{equation}
\label{eq: fluc_gen_N}
\begin{split}
\E[T^n_{t_1,t_3} \Phi(\rho^N_{t_2})] &- \E[T^n_{t_1,t_3}\Phi(\rho^N_{t_1})] 
 =\int_{t_1}^{t_2} \E\big[\mathcal{G}^N_s T_{t_1,t_3}^n \Phi (\rho_s^N) \big] ds\\
 &:= \int_{t_1}^{t_2} \E\bigg[ \sqrt{N} \int_{\T^d} \Big( b(s, x, \mu_s^N ) - b(s, x,\mu_s) \Big) \cdot  D ( \nabla \hat T_{t_1,t_3}^n  \Phi(\rho_{s}^N)(x) \Id \mu_s(x)   \nonumber \\
    & \quad + \Big \langle   b(s, \cdot, \mu_s^N ) \cdot  D  \nabla \hat T_{t_1,t_3}^n  \Phi(\rho_{s}^N)(\cdot) 
    + \frac{\sigma^2}{2} 
    \Delta \, \nabla \hat T_{t_1,t_3}^n  \Phi(\rho_{s}^N)(\cdot) 
    , \rho_{s}^N \Big\rangle   \nonumber \\
    & \quad +
    \frac{\sigma^2}{2} \sum\limits_{j=1}^d \int_{\T^d} \partial_{x_j} \partial_{y_j}
     ( \nabla^2 \hat T_{t_1,t_3}^n  \Phi(\rho_{s}^N)(x,y)_{\big|x=y}  
    \Id \mu^N_s ( x) \bigg]  ds .
\end{split}
\end{equation}
\end{proposition}

\begin{proof}
We apply It\^o formula to the function $\psi:[0,T]\times (\T^d)^N \rightarrow \R$:
\[
 \psi(t,x_1,\dots,x_N)  := T_{t_1,t_3}^n  \Phi \big( \sqrt{N} (\mu^N_{\bm{x}} - \mu_t) \big).  
\]
For the time derivative, we have  
\begin{align*}
  \partial_t \psi(t,\bm{x}) &=  \frac{d}{dt} \Psi_{\bm{x}}(\mu_t) \\
  &= \int_{\T^d} \Big( b(x, \mu_t) \cdot D \frac{\delta}{\delta m} \Psi_{\bm{x}}(\mu_t; x) 
+ \frac{\sigma^2}{2} 
    \Tr[  D^2_x \frac{\delta}{\delta m} \Psi_{\bm{x}}(\mu_t; x) ] 
    \Big) \mu_t(dx)
\end{align*}
where $\Psi_{\bm{x}}(m):= T_{t_1,t_3}^n  \Phi \big( \sqrt{N} (\mu^N_{\bm{x}} - m) \big)$.  Thanks to the above results we have 
\[
\frac{\delta}{\delta m} \Psi_{\bm{x}}(m; x) = -\sqrt{N} \nabla \hat T_{t_1,t_3}^n \Phi(\sqrt{N}(\mu^N_{\bm{x}}-m))(x)
\]
and thus 
\begin{align*}
    \partial_t \psi(t,\bm{X}_t) &= -\sqrt{N} \int_{\T^d} \Big( b(x, \mu_t) \cdot D \hat T_{t_1,t_3}^n \nabla \Phi(\rho^N_t; x) 
    + \frac{\sigma^2}{2} 
    \Delta \, \nabla \hat T_{t_1,t_3}^n \Phi(\rho^N_t; x)  
    \Big) \mu_t(dx)
\end{align*}
The space derivatives are given, letting $\Psi_t(m) =\hat T_{t_1,t_3}^n\Phi(\sqrt{N}(m-\mu_t))$, by
\begin{align*}
    \partial_{x_i} \psi(t,\bm{x})&= 
    \frac1N D\frac{\delta}{\delta m}  \Psi_t(\mu^N_{\bm{x}})(x_i)\\
    &= \frac{1}{\sqrt{N}} D \nabla \hat T_{t_1,t_3}^n \Phi(\sqrt{N}(\mu^N_{\bm{x}}-\mu_t))(x_i) \\
    \partial^2_{x_i x_i} \psi(t,\bm{x})&= 
    \frac1N D^2_x\frac{\delta}{\delta m}  \Psi_t(\mu^N_{\bm{x}})(x_i) 
    +\frac{1}{N^2} D_x D_y \frac{\delta}{\delta m}  \Psi_t(\mu^N_{\bm{x}})(x_i, x_i)\\
    &= \frac{1}{\sqrt{N}} D^2_x \nabla \hat T_{t_1,t_3}^n \Phi(\sqrt{N}(\mu^N_{\bm{x}}-\mu_t))(x_i) 
    +\frac1N D_x D_y \nabla^2 \hat T_{t_1,t_3}^n \Phi(\sqrt{N}(\mu^N_{\bm{x}}-\mu_t))(x_i, x_i) .
\end{align*}
Notice that we use above a  slightly different version of \eqref{eq: first_flat_derivatives}-\eqref{eq: second_flat_derivatives}, whereas $T_{t_1,t_3}^n \Phi$ is evaluated on an affine function of a measure, but the result follows in the same way. 
We apply It\^o formula and take expectation to get
\begin{align*}
\E&[T^n_{t_1,t_3} \Phi(\rho^N_{t_2})] - \E[T^n_{t_1,t_3}\Phi(\rho^N_{t_1})] \\
 &= \int_{t_1}^{t_2} \E\bigg[ \partial_t 
 \psi(t,\bm{X}_t) 
 + \sum_{i=1}^N b(t,X^i_t,\mu^N_t) \cdot \partial_{x_i} \psi(t,\bm{X}_t) 
 + \frac{\sigma^2}{2} \sum_{i=1}^N \Delta_{x_i} \psi(t,\bm{X}_t)  
 \bigg] dt \\
&= \int_{t_1}^{t_2} \E\bigg[ - \sqrt{N}
\int_{\T^d} \Big( b(t,x, \mu_t) \cdot D \nabla \Phi(\rho^N_t; x) 
    + \frac{\sigma^2}{2} 
    \Delta \, \nabla \Phi(\rho^N_t; x)  
    \Big) \mu_t(dx) \\ 
&\qquad + \sqrt{N} \frac1N \sum_{i=1}^N 
\Big( b(t,X^i_t, \mu^N_t) \cdot D \nabla \Phi(\rho^N_t; X^i_t) 
    + \frac{\sigma^2}{2} 
    \Delta \, \nabla \Phi(\rho^N_t; X^i_t)  
    \Big)     
     \\
    &\qquad + \frac{\sigma^2}{2} \frac1N \sum_{i=1}^N \sum\limits_{j=1}^d  \partial_{x_j} \partial_{y_j}
      \nabla^2 \hat T_{t_1,t_3}^n  \Phi(\rho_{s}^N)(X^i_t,X^i_t)  \bigg]dt \\
&= \int_{t_1}^{t_2} \E\bigg[ - \sqrt{N}
\int_{\T^d} \Big( b(x, \mu_t) \cdot D \nabla \Phi(\rho^N_t; x) 
    + \frac{\sigma^2}{2} 
    \Delta \, \nabla \Phi(\rho^N_t; x)  
    \Big) \mu_t(dx) \\ 
&\qquad + \sqrt{N} \int_{\T^d} 
\Big( b(t,x, \mu^N_t) \cdot D \nabla \Phi(\rho^N_t; x) 
    + \frac{\sigma^2}{2} 
    \Delta \, \nabla \Phi(\rho^N_t; x)  
    \Big) \mu^N_t(dx)     
     \\
    &\qquad + \frac{\sigma^2}{2} \int_{\T^d}\sum\limits_{j=1}^d  \partial_{x_j} \partial_{y_j}
      \nabla^2 \hat T_{t_1,t_3}^n  \Phi(\rho_{t}^N)(x,x) \mu^N_t(dx) \bigg]dt \\
 &= \int_{t_1}^{t_2} \E\bigg[ \sqrt{N} \int_{\T^d} \Big( b(t, x, \mu_s^N ) - b(t, x,\mu_t) \Big) \cdot  D  \nabla \hat T_{t_1,t_3}^n  \Phi(\rho_{t}^N)(x) \Id \mu_t(x)   \nonumber \\
    & \qquad + \sqrt{N} \int_{\T^d}  \Big( b(t, x, \mu_t^N ) \cdot  D  \nabla \hat T_{t_1,t_3}^n  \Phi(\rho_{t}^N)(x)  
    + \frac{\sigma^2}{2} 
    \Delta \, \nabla \hat T_{t_1,t_3}^n  \Phi(\rho_{t}^N)(x) 
    \Big) \Id (\mu^N_t-\mu_t)(x)  \\
    & \qquad +
    \frac{\sigma^2}{2} \sum\limits_{j=1}^d \int_{\T^d} \partial_{x_j} \partial_{y_j}
      \nabla^2 \hat T_{t_1,t_3}^n  \Phi(\rho_{t}^N)(x,y)_{\big|x=y}  
    \Id \mu^N_t ( x) \bigg]  dt, 
\end{align*}
which gives the claim. 
\end{proof}

Utilizing flat derivative we can rewrite it as follows:     
\begin{proposition}\label{prop: generator_n_fluctuation}
Let \(\Phi \in FC^\infty(H^{-\lambda-2}(\T^d))\). Then 
 \begin{align*} 
    &\E\big[\mathcal{G}^N_s T_{t_1,t_3}^n \Phi (\rho_{s}^N )\big]\\
    &\quad =\E\bigg[ \int_{\T^d}  \int_{\T^d} \frac{\delta b}{\delta m} \bigg(s, x, \mu_s , v \bigg)  \Id \rho_{s}^N(v) \cdot D ( \nabla \hat T_{t_1,t_3}^n  \Phi(\rho_{s}^N))(x) \Id \mu_s(x) \\
    & \quad \quad  + \Big\langle  b(s, \cdot, \mu_s) \cdot  D ( \nabla \hat T_{t_1,t_3}^n  \Phi(\rho_{s}^N))(x) 
    + \frac{\sigma^2}{2} 
    \Delta ( \nabla \hat T_{t_1,t_3}^n  \Phi(\rho_{s}^N))(x) 
    , \rho_{s}^N \Big\rangle \\
    &\quad \quad +
    \frac{\sigma^2}{2} \sum\limits_{j=1}^d \int_{\T^d} \partial_{x_j} \partial_{y_j}
     ( \nabla^2 \hat T_{t_1,t_3}^n  \Phi(\rho_{s}^N))(x,y)_{\big|x=y}  
    \Id \Big(\mu_s + \frac{\rho_s^N}{\sqrt{N}}\Big) ( x) \\
    & \quad \quad  + \frac{1}{\sqrt{N}} \Big \langle \int\limits_0^1 \int_{\T^d} \frac{\delta b}{\delta m} (s,\cdot,  r\mu_s^N +(1-r)  \mu_s  , v ) \Id  \rho_{s}^N (v) \Id r \cdot D ( \nabla \hat T_{t_1,t_3}^n  \Phi(\rho_s^N))(x)  , \rho_s^N \Big \rangle \\
    &\quad \quad  +  \int_{\T^d}  \int\limits_0^1 \frac{r}{\sqrt{N}}  \int_{\T^d} \int\limits_0^1 \int_{\T^d} \frac{\delta^2 b}{\delta m^2} (s,x, r' r\mu_s^N + (1-rr')  \mu_s   ,v,v') \Id v' \Id \rho_s^N(v') \Id r' \Id  \rho_s^N (v) \dd r \\
    &\quad \quad \cdot D( \nabla \hat T_{t_1,t_3}^n  \Phi(\rho_s^N))(x)  \Id \mu_s(x) \bigg]. 
     \end{align*}
     Further, the first three term in the above representation can be written as
    \begin{align}
    \nonumber
     &\int_{\T^d}  \int_{\T^d} \frac{\delta b}{\delta m} \bigg(s, x, \mu_s , v \bigg)  \Id \rho_{s}^N(v) \cdot D ( \nabla \hat T_{t_1,t_3}^n  \Phi(\rho_{s}^N))(x) 
     \mu_s( \dd x) 
     + \Big\langle  b(s, \cdot, \mu_s) \cdot  D ( \nabla \hat T_{t_1,t_3}^n  \Phi(\rho_{s}^N))(x)  \\
     \label{remark: vannishing}
    &\qquad \qquad + \frac{\sigma^2}{2} 
    \Delta ( \nabla \hat T_{t_1,t_3}^n  \Phi(\rho_{s}^N))(x) 
    , \rho_{s}^N \Big\rangle 
     =    \langle \rho_s^N , A'(s) \nabla \hat T^{n}_{t_1,t_3}\Phi(\rho_s^N)\rangle .
    \end{align}
Moreover, the map \(r \mapsto \mathcal G_s^N T_{r,t}^n\Phi(\rho_s^N)\) is continuous on \([0,t]\). 
\end{proposition}
\begin{proof}
The above formula is basically~\eqref{eq: fluc_gen_N}. We only need to notice that by the definition of the flat derivative we have 
\begin{equation*}
    b(s,x, \mu_s^N ) - b(s,x,\mu_s) 
    = \frac{1}{\sqrt{N}}\int\limits_0^1 \int_{\T^d} \frac{\delta b}{\delta m} (s,x, r\mu_s^N +(1-r)  \mu_s  , v )  \Id \rho_s^N(v) \Id r
\end{equation*}
and 
\begin{align*}
     &\frac{\delta b}{\delta m} (s,x, r\mu_s^N +(1-r)  \mu_s  , v )  - \frac{\delta b}{\delta m} \bigg(s,x, \mu_s , v \bigg) \\
     &\quad = \frac{r}{\sqrt{N}}\int\limits_0^1 \int_{\T^d} \frac{\delta^2 b}{\delta m^2} (s,x, r' r\mu_s^N + (1-rr')  \mu_s   ,v,v') \Id v' \Id \rho_s^N( v') \Id r' .
\end{align*}
Plugging both formulas in~\eqref{eq: fluc_gen_N} we prove the first claim. 
For the continuity, we notice that \(\rho_s^N\) is a finite signed measure, all terms are bounded and, therefore, it is enough to demonstrate that 
\(r \mapsto \norm{\nabla \hat{T}_{r,t}^n \Phi(\rho_s^N) }_{C^2(\T^d)}\) and \(r \mapsto \norm{\partial_{x_j} \partial_{y_j}
     ( \nabla^2 \hat T_{r,t}^n  \Phi(\rho_{s}^N))(x,y)_{\big|x=y} }_{L^\infty(\T^d)} \) are continuous. 
Combining the estimates from~\eqref{cor: besov_reg_first_order_function} with Lemma~\ref{cor: con_semigroup_time} and \(\lambda+2-d/2 > 2\) proves the continuity of \(r \mapsto \norm{\nabla \hat{T}_{r,t}^n \Phi(\rho_s^N) }_{C^2(\T^d)}\). In order to prove the continuity of the second term we recall~\eqref{eq: def_second_derivative_function}, which is the definition of \(\nabla^2 \hat T_{r,t}^n  \Phi\). Then, the continuity follows by similar computations as in Lemma~\ref{cor: con_semigroup_time}. 
\end{proof}

Notice that  Lemma~\ref{lemma: generatpr_trace_computation} implies that the trace term (fourth term) is actually the trace term computed in the generator of \(\rho_s^n\). 

\begin{remark}
So far, we have computed the generator for the specific function \(T_{t_1,t_3}^n \Phi\). However, the above arguments can be extended to general sufficiently smooth functions \(\Phi\) in place of \(T_{t_1,t_3}^n \Phi\). 
The only subtle point arises in Lemma~\ref{lemma: notion_of_derivatives}, where we must identify the second flat derivative with the classical second derivative of the function. While this identification can, in principle, be justified using the Riesz representation theorem twice, doing so would require careful considerations regarding measurability.
To keep the computation explicit and avoid technical complications, we chose to compute the generator directly for \(T_{t_1,t_3}^n \Phi\), which will be used later in Section~\ref{sec: comparison}.
\end{remark}

\section{Proof of the main result} 
\label{sec: comparison}

The aim of this section is to prove our Main Theorem~\ref{theorem: main_result} by comparing the generators of \(\rho^N\) in \eqref{eq: def_fluctuation_process} and \(\rho^n\) in \eqref{eq: approx_solution}. The strategy lies in utilizing the convergence in probability provided by Theorem~\ref{thm: conv_prob} to reduce the problem to the difference between \(\rho^N\) and \(\rho^n\), and then to approximate $\Phi\in C^2_\ell(H^{-\lambda-2}(\T^d))$ with $FC^\infty(H^{-\lambda-2}(\T^d))$ by Lemma~\ref{lemma: approximation_hilbert_smooth}.
 As mentioned in the introduction, the semigroups are not strongly continuous and therefore the results such as~\cite[Lemma~1.2.5]{Ethier1986} or~\cite[Theorem~2.11]{Kolokoltsov2010}, which would consist in writing
 \[
 \E[\Phi(\rho_t^N(f))-\Phi(\rho^n_t(f))] 
 =\int\limits_0^t \frac{d}{ds} T^N_{0,s} T^n_{s,t}\Phi (f)
    = \int\limits_0^t T^N_{0,s} (\mathcal{G}_s^N -\mathcal{G}_s)T^n_{s,t} \Phi (f) \Id s ,
 \] are not applicable. 
%
%
%
 Instead, we use the technique provided by~\cite[Lemma~2.2]{gess2024}, exploiting the continuity in time of $\mathcal{G}_s$ proved above. 
Let us recall that $\rho^N$ is fixed, given by \eqref{eq: def_fluctuation_process} with an initial condition $\rho^N_0$, and might not be a Markov process. 

\begin{lemma} \label{lemma: product}
    Under Assumptions~\ref{ass: inital_cond}-\ref{ass: coef_fokker}, let \(\Phi \in FC^\infty(H^{-\lambda-2}(\T^d))\) 
    and let $(\rho^n_t(\rho^N_0), 0\leq t\leq T)$ solve \eqref{eq: approx_solution} with  $\rho^n_0=\rho^N_0$
    Then, for all \(0 \leq t \leq T\) and \(n,N\in \N\), we have 
    \begin{equation*}
        \E\big[\Phi(\rho_t^N)\big]  - 
         \E[\Phi(\rho^n_t(\rho^N_0))]
        =  \int\limits_0^t \E \big[\mathcal{G}_s^{N} T^n_{s,t}\Phi( \rho^N_s )+ \partial_s T_{s,t}^n\Phi(\rho_{s}^N)\big] \Id s . 
    \end{equation*}
\end{lemma}
\begin{proof}
    Consider the partition \(0 = t_0 < t_1 \cdots <t_n = t\). 
    Since $\rho^n$ is a Markov process, we have
    \[
\E[\Phi(\rho^n_t(\rho^N_0))] = \E [T_{0,t}^n\Phi(\rho^N_0)],
    \]
    and we obtain  
    \begin{align*}
        &\E\big[\Phi(\rho_t^N)\big]  - \E[\Phi(\rho^n_t(\rho^N_0))] \\
        &\quad = \sum\limits_{i=1}^n \E\big[T_{t_{i},t}^n\Phi(\rho_{t_{i}}^N) - T_{t_{i-1},t}^n\Phi(\rho_{t_{i-1}}^N)\big] \\
        &\quad = \sum\limits_{i=1}^n \E \big[ T_{t_{i-1},t}^n\Phi(\rho_{t_{i}}^N) -T_{t_{i-1},t}^n\Phi(\rho_{t_{i-1}}^N)\big]+ \E\big[ T_{t_{i},t}^n\Phi(\rho_{t_{i}}^N) - T_{t_{i-1},t}^n\Phi(\rho_{t_{i}}^N)\big] \\
        &\quad = \sum\limits_{i=1}^n \int\limits_{t_{i-1}}^{t_i}  \E\big[ \mathcal G_s^{N}T_{t_{i-1},t}^n\Phi( \rho^N_s ) \big] \Id s  + \int\limits_{t_{i-1}}^{t_i} \E \big[ \partial_s T_{s,t}^n\Phi(\rho_{t_i}^N) \big] \Id s ,
    \end{align*}
   where we apply Proposition \ref{prop:4_11} in the last line. The last integrals converge towards
    \begin{equation*}
    \int\limits_0^t \E \big[ G_s^{N}T_{s,t}^n\Phi ( \rho^N_s )+ \partial_s T_{s,t}^n\Phi (\rho_{s}^N)\big] \Id s , 
    \end{equation*}
    since  \(r \mapsto \partial_s T_{s,t}^n\Phi (\rho_{r}^N) \) is continuous by Proposition~\ref{lemma: ito_formula_approx} and the fact that \(r \mapsto \rho_{r}^N\) is continuous in \(H^{-\lambda-2}(\T^d) \), while continuity of \(r \mapsto G_s^{N}T_{r,t}^n\Phi ( \rho^N_s ) \) is shown in Proposition~\ref{prop: generator_n_fluctuation}. 
\end{proof}

We can now turn to our initial goal, which was to estimate the weak error by the associated generators. The following proposition gives the explicit expression of the reminder, which will be estimated in the next results.  

\begin{proposition} \label{prop: reminder}
    Under Assumptions~\ref{ass: inital_cond}-\ref{ass: coef_fokker}, let \(\Phi \in FC^\infty(H^{-\lambda-2 }(\T^d))\)
    and let $(\rho^n_t, 0\leq t\leq T)$ solve \eqref{eq: approx_solution} with  $\rho^n_0=\rho_0$. Then 
\begin{equation}
 \E\big[ \Phi(\rho_t^N)- \Phi(\rho^n_t) \big] 
= R^0_t(N,n,\Phi) + \frac{1}{\sqrt{N}} \int\limits_0^t \E\big[(R^1_s + R^2_s +R^3_s +R^4_s)(N,n,\Phi)\big] \Id s,
\end{equation}
with 
\begin{align*}
R^0_t(N,n,\Phi) &= \int_{H^{-\lambda-2 }(\T^d)} T_{0,t}^n \Phi(f) \Id \big(\P_{\rho_0^N}-\P_{\rho_0}\big)(f) \\ 
R^1_s(N,n,\Phi)&= \frac{\sigma^2}{2} \sum\limits_{j=1}^d \int_{\T^d} \partial_{x_j} \partial_{y_j}
( \nabla^2 \hat T_{s,t}^n  \Phi(\rho_{s}^N))(x,y)_{\big|x=y}  \Id \rho_s^N ( x)  \\
  R^2_s(N,n,\Phi) &  =  \Big \langle \int\limits_0^1 \int_{\T^d} \frac{\delta b}{\delta m} (s,\cdot,  r\mu_s^N +(1-r)  \mu_s  , v )\Id  \rho_{s}^N (v) \Id r \cdot D  ( \nabla \hat T_{s,t}^n  \Phi(\rho_s^N))(\cdot)  , \rho_s^N \Big \rangle \\
   R^3_s(N,n,\Phi) &  =\!\int_{\T^d} \! \int\limits_0^1\!\! r \!\!\int_{\T^d} \!\int\limits_0^1 \!\!\int_{\T^d} \!\frac{\delta^2 b}{\delta m^2} \bigg(\!s,x,r r'\mu_s^N\!\!+\!(1\!-\!rr')\mu_s ,v,v'\!\bigg)\! \Id v' \Id \rho_s^N(v') \Id r' \!\Id  \rho_s^N\! (v) \dd r  \cdot \!D ( \nabla \hat T_{s,t}^n  \Phi(\rho_s^N))(x)  \Id\mu_s(x)  \\
   R^4_s(N,n,\Phi) &= \langle \rho_s^N , A'(s) \nabla \hat T^{n}_{s,t}\Phi(\rho_s^N)\rangle - \langle A_n(s, \rho_s^N) , \nabla T_{s,t}^n \Phi (\rho_s^N) \rangle_{H^{-\lambda-2}(\T^d)}.  
\end{align*}
\end{proposition}

\begin{proof}
We want to apply Lemma~\ref{lemma: product}. We utilize the Markov property of $\rho^n$ 
to rewrite the left hand side as  
\begin{align*} 
     \E\big[ \Phi(\rho_t^N)\big] &- \E\big[ \Phi(\rho_t^n)\big]   =  \E\big[ \Phi(\rho_t^N)\big] - 
        \int_{H^{-\lambda-2}(\T^d)} T_{0,t}^n \Phi(f) \Id \P_{\rho_0}(f)  \\
      &= \E\big[ \Phi(\rho_t^N)\big] - \E\big[ \Phi(\rho_t^n(\rho^N_0))\big] + \int_{H^{-\lambda-2}(\T^d)} T_{0,t}^n \Phi(f) \Id \P_{\rho_0^N}(f) - \int_{H^{-\lambda-2}(\T^d)} T_{0,t}^n \Phi(f) \Id \P_{\rho_0}(f) 
\end{align*}
Hence, applying Lemma~\ref{lemma: product} we get 
\begin{equation*}
    \E\big[ \Phi(\rho_t^N)- \Phi(\rho^n_t) \big] =
    \int\limits_0^t \E\big[ \big(\mathcal G_s^N  - \mathcal{G}_s^n  \big) T_{s,t}^n \Phi (\rho_s^N) \big] \Id s
    + R^0_t (N,n,\Phi).
\end{equation*}
Therefore, applying Proposition \ref{prop:4.9} and   Proposition~\ref{prop: generator_n_fluctuation} with~\eqref{remark: vannishing}, we obtain 
the claim.
\end{proof}

\begin{remark} \label{remark: second_flat_vanish}
    Notice that for the classical interaction drift \(b(t,x,m) = K*m(x)\) with some interaction force kernel \(K\), we have \(\tfrac{\delta^2 b}{\delta m^2}(t,x,m,v) = 0\) and the term $R^3(N,n,\Phi)$ vanishes. 

     Notice also that $R^4(N,n,\Phi)$ depends on $A_n$, while the other terms depend on $n$ only via derivatives of $T^n_{s,t}\Phi$. 
\end{remark}
We now estimate the reminder terms and thus prove Theorem~\ref{theorem: main_result}.
We start with the initial condition. 
\begin{lemma}
\label{lem:r0}
Under Assumptions~\ref{ass: inital_cond}-\ref{ass: coef_fokker}, for $\Phi\in C_{\ell}^2( H^{-\lambda-2}(\T^d)$, we have
\begin{equation}
|R^0_t(N,n,\Phi)| \leq C [\Phi]_{C^1(H^{-\lambda-2}(\T^d))} W_{1, H^{-\lambda-2}(\T^d)}\Big( \P_{\rho^N_0}, \P_{\rho_0}\Big).
\end{equation} 
\end{lemma}
\begin{proof}
We recall that by Kantorovich-Rubenstein duality 
\[
W_{1, H^{-\lambda-2}(\T^d)}\Big( \P_{\rho^N_0}, \P_{\rho_0}\Big) = \sum_{\Psi \mathrm{ Lipschitz, \,[\Phi]_{C^1(H^{-\lambda-2}(\T^d))}} \leq 1} \int_{H^{-\lambda-2}(\T^d)} \Psi(f) \Id (\P_{\rho^N_0}- \P_{\rho_0})(f). 
\]
Then 
\[
R^0_t(N,n,\Phi) \leq [\nabla  T_{s,t}^n \Phi]_{C^1(H^{-\lambda-2}(\T^d))} W_{1, H^{-\lambda-2}(\T^d)}\Big( \P_{\rho^N_0}, \P_{\rho_0}\Big)
\leq C [\Phi]_{C^1(H^{-\lambda-2}(\T^d))} W_{1, H^{-\lambda-2}(\T^d)}\Big( \P_{\rho^N_0}, \P_{\rho_0}\Big), 
\] 
where we used the bound \eqref{eq:est_1} in Lemma~\ref{cor: reg_derivative}. 
\end{proof} 

We then show that the last term vanishes with $n$.

\begin{lemma} \label{lemma: vanish_first_order}
Under Assumptions~\ref{ass: inital_cond}-\ref{ass: coef_fokker}, for  \(\Phi \in  FC^\infty(H^{-\lambda-2}(\T^d))\), we have
    \begin{equation*}
         \lim_{ n\to \infty}  
         \int_0^t\E|R^4_s(N,n,\Phi)| \Id s
        = 0 
    \end{equation*}
\end{lemma}
\begin{proof}
Notice that 
\begin{align*}
\langle A_n(s, \rho_s^N) , \nabla T_{s,t}^n \Phi (\rho_s^N) \rangle_{H^{-\lambda-2}(\T^d)}
& =  \sum\limits_{k \in \Z^d} \langle k \rangle^{-2(\lambda+2)} \langle A(s) (j_n*\rho_s^N-\rho_s^N),j_n*e_k \rangle 
 \overline{\langle \nabla T_{s,t}^n \Phi (\rho_s^N) , e_k \rangle} \\
&\quad + \langle k \rangle^{-2(\lambda+2)} \langle A(s) (\rho_s^N),j_n* e_k-e_k  \rangle 
 \overline{\langle \nabla T_{s,t}^n \Phi (\rho_s^N) ,e_k\rangle } \\
&\quad + \langle k \rangle^{-2(\lambda+2)} \langle \rho_s^N ,A'(s)(e_k) \rangle 
 \overline{\langle \nabla T_{s,t}^n \Phi (\rho_s^N) , e_k \rangle}.
\end{align*}
The first term can be estimated by 
\begin{align*}
    &\E\big[\norm{\tilde j}_{L^\infty} \norm{A(s)(j_n*\rho_s^N-\rho_s^N)}_{H^{-\lambda-2}(\T^d)}    \norm{\nabla T_{s,t}^n \Phi (\rho_s^N)}_{H^{-\lambda-2}(\T^d)}\big] \\
    &\quad \le C \E\big[ \norm{j_n*\rho_s^N-\rho_s^N}_{H^{-\lambda}(\T^d)}\big], 
\end{align*}
where we applied Lemma~\ref{cor: reg_derivative} in the last step. 
The right hand side vanishes for each \(N \in \N\) by the properties of mollification, the fact that \(\rho_s^N \in H^{-\lambda}(\T^d)\) and the dominated convergence theorem. Similarly, the second term vanishes. Hence, we obtain 
    \begin{align*}
         &\limsup_{ n\to \infty} |\langle A_n(s, \rho_s^N) , \nabla T_{s,t}^n \Phi ( \rho_s^N) \rangle_{H^{-\lambda-2}(\T^d)} -  \langle  \rho_s^N, A'(t) \nabla \hat T^{n}_{s,t}\Phi(f)\rangle| \\
        &\quad \le \limsup_{ n\to \infty} \bigg|\sum\limits_{k \in \Z^d} \langle k \rangle^{-2(\lambda+2)} \langle \rho_s^N ,A'(s)(e_k) \rangle 
 \langle \nabla T_{s,t}^n \Phi ( \rho_s^N) , e_k \rangle - \langle  \rho_s^N, A'(t) \nabla \hat T^{n}_{s,t}\Phi( \rho_s^N)\rangle\bigg|.
    \end{align*}
    Now, we only need to notice that 
    \begin{equation*}
        \sum\limits_{k \in \Z^d} \langle k \rangle^{-2(\lambda+2)} \langle  \rho_s^N ,A'(s)(e_k) \rangle 
\overline{ \langle \nabla T_{s,t}^n \Phi ( \rho_s^N) , e_k \rangle}
 =\langle  \rho_s^N, A'(t) \nabla \hat T^{n}_{s,t}\Phi( \rho_s^N)\rangle
    \end{equation*}
    for each \(n \in \N\) and the claim is proven. Indeed, by Sobolev's embedding the series of derivatives converge uniformly and we can exchange the series and the derivatives in the operator \(A'\). 
\end{proof}

We next estimate the remaining terms using Assumption~\ref{ass: convergence_ass}. We start with the second order term. Recall that \(\bar \mu_t^N\) is the joint distribution of the whole \(N\)-particle system~\eqref{eq: interacting_particle_system}.

\begin{lemma} \label{lemma: second_ord_rest}
Under Assumptions~\ref{ass: inital_cond}-\ref{ass: convergence_ass}, for \(\Phi \in FC^\infty(H^{-\lambda-2}(\T^d))\), we have
\begin{equation*}
    \int_0^t\E |R^1_s(N,n,\Phi)| \Id s
    \le C [\Phi]_{C^2}. 
\end{equation*}

\end{lemma}
\begin{proof}
    Recall that for each \(n \in \N \), \(j\), \(\omega \in \Omega\) and \(s \in [0,t]\) the function \(x \mapsto \partial_{x_j} \partial_{y_j} \hat T^n_{s,t}\Phi (\rho_s^N(\omega))(x,y)\big|_{x=y} \) is bounded in \(H^{\tilde \lambda - d}(\T^d) \) by Lemma~\ref{cor: besov_reg_second_order_function}. Therefore, for \(\tilde \lambda \in (3/2 d , \lambda ) \) we have 
    \begin{align*}
        &\bigg|\int\limits_0^t \E\bigg[ \int\limits_{\T^{d}} \partial_{x_j}\partial_{y_j} \nabla^2 \hat T^n_{s,t}\Phi(\rho_s^N)(x,y)_{\big| x= y} \Id \rho_s^N(x) \bigg]  \Id s \bigg| \\
        &\quad \le \sqrt{N}  \int\limits_0^t  \E\big[\norm{ \partial_{x_j} \partial_{y_j} \hat T^n_{s,t}\Phi (\rho_s^N(\omega))(x,y)\big|_{x=y} }_{H^{\tilde \lambda -d}(\T^d) }
        \norm{\mu_s^N-\mu_s}_{H^{-(\tilde \lambda -d)}(\T^d) } \big] 
       \Id s  \\
        &\quad \le C \sqrt{N} \int\limits_0^t  \sup\limits_{f \in H^{-\lambda-2}(\T^d)} \norm{\nabla^2 \hat T^n_{s,t}\Phi(f)(\cdot, \cdot)\big|_{x=y}}_{H^{\tilde \lambda -d}(\T^d) } \E\big[  \norm{\mu_s^N-\mu_s}_{H^{-(\tilde \lambda -d)}(\T^d) }^2 \big]^{\frac{1}{2}} \Id s  \\
        &\quad \le C T \sqrt{N} [\Phi]_{C^2} \sup\limits_{0 \le t \le T }  \E\big[\norm{\mu_s^N-\mu_s}_{H^{-(\tilde \lambda -d)}(\T^d) }^2 \big]^{\frac{1}{2}} \Id s, 
    \end{align*}
where the last inequality follows by~\eqref{ineq:diag}. Notice that \(\tilde \lambda -d > d/2\) and we can apply~\cite[Lemma~2.6]{Xianliang2023} to find 
\begin{equation} \label{eq: second_order_rest_term_aux}
    \int_0^t \E |R^1_s(N,n,\Phi)| \Id s 
    \le C T [\Phi]_{C^2} \Big( \sup\limits_{0 \le t \le T} \mathcal{H}(\bar \mu_t^N | \mu_t^{\otimes N} )^{\frac{1}{2}} +1 \Big). 
\end{equation}
The right-hand side is bounded by Assumption~\ref{ass: convergence_ass}, which proves the Lemma.  
\end{proof}

Next, we analyze the rest terms concerning the flat derivative. We employ relative entropy bounds; thus  we need to recall some results from the theory of mean-fields limits via the relative entropy method. 
Let us recall the change of measure via the variational formula of the relative entropy~\cite[Lemma~1]{JabinWang2018}. 
\begin{lemma}\label{lemma: variational_formual}
    Let \(\nu_1,\nu_2\) be two probability densities on \(\T^{dN}\) and \(\varphi \in L^\infty(\T^{dN})\). Then for any \(\kappa > 0\) we have 
    \begin{equation} \label{eq: variational_formula}
        \int_{\T^{dN}} \varphi \Id \nu_1 \le 
        \frac{1}{\kappa N }\bigg( \mathcal{H}(\nu_1 \vert \nu_2) + \log\bigg( \int\limits_{\T^{dN}} \exp(\kappa N \varphi) \Id \nu_2  \bigg) \bigg).  
    \end{equation}
\end{lemma}

We also recall the following large deviation estimate by Jabin, Wang~\cite[Theorem~4]{JabinWang2018} in combination with \cite[Remark~2.2]{Xianliang2023}. 
\begin{theorem}\label{theorem: large_deviation}
    Consider \(\mu \in L^1(\T^d)\) with \(\mu \geq 0\) and \(\int_{\T^d} \mu \,dx = 1\). Consider  
further any \(\varphi(x, z) \in L^\infty(\T^d)\) with  
\begin{equation} \label{eq: deviation_p_bound}
\gamma := \hat C \norm{\varphi}_{L^\infty(\T^{2d})} < 1,    
\end{equation}
where \(\hat C\) is a universal constant. Assume that \(\varphi\) satisfies the following cancellations:
\[
\int_{\T^d} \varphi(x, z) \mu(x) \Id x = 0 \quad \forall z, \quad  
\int_{\T^d} \varphi(x, z) \mu(z) \Id z = 0 \quad \forall x.
\]
Then,
\[
\int_{\T^{dN}} \mu^{\otimes N}(x_1,\ldots,x_N) \exp \left(\left| \frac{1}{N} \sum_{i,j=1}^{N} \varphi(x_i, x_j) \right |\right) \Id x_1 \cdots \Id x_N  
\leq \frac{2}{1 - \gamma} < \infty,
\]
where we recall that \(\mu^{\otimes N} (t, x_1,\ldots,x_N) = \prod_{i=1}^{N} \mu(t, x_i)\).
\end{theorem}

\begin{lemma} \label{lemma: flat_reminder}
Under Assumptions~\ref{ass: inital_cond}-\ref{ass: convergence_ass}, for \(\Phi \in FC^\infty(H^{-\lambda-2}(\T^d))\), we have  
    \begin{align*}
        \int_0^t \E |R^2_s(N,n,\Phi)| \Id s \le C  [\Phi]_{C^1}. 
    \end{align*}
\end{lemma}
\begin{proof}
We start by rewriting the term as a function integrated by the product measure \(\mu_s^N \otimes \mu_s^N\): 
\begin{align*}
    & \bigg|  \E  \bigg[ \Big \langle \int\limits_0^1 \int_{\T^d} \frac{\delta b}{\delta m} (s,\cdot,r \mu_s^N + (1-r)\mu_s , v ) \Id  \rho_{s}^N (v) \Id r \cdot D ( \nabla \hat T_{s,t}^n  \Phi(\rho_s^N))(\cdot)  , \rho_s^N \Big \rangle \bigg] \bigg| \\
     &\quad  = \bigg|  \E  \bigg[ \Big \langle \int\limits_0^1 \sum\limits_{k \in \Z^d } \langle k \rangle^{-2(\lambda +2) } \langle  \nabla T_{s,t}^n  \Phi(\rho_s^N), e_k \rangle  \int_{\T^d} \frac{\delta b}{\delta m} (s,\cdot,r \mu_s^N + (1-r)\mu_s , v ) \Id  \rho_{s}^N (v) \Id r \cdot D e_k(\cdot)   , \rho_s^N \Big \rangle \bigg] \bigg| \\
    &\quad  \le C      \int\limits_0^1 \sum\limits_{k \in \Z^d } |k| \sup\limits_{f \in H^{-2-\lambda}(\T^d)} |\langle  \nabla T_{s,t}^n  \Phi(f), e_k \rangle_{H^{-\lambda-2}}| \E \bigg[ \bigg| \Big \langle \int_{\T^d} \frac{\delta b}{\delta m} (s,\cdot,r \mu_s^N + (1-r)\mu_s , v ) \Id  \rho_{s}^N (v) \Id r  e_k(\cdot)   , \rho_s^N \Big \rangle \bigg] \bigg| \\
    &\quad  \le C \int\limits_0^1 \sum\limits_{k \in \Z^d }  \langle k \rangle^{-\lambda-1} \sup\limits_{f \in H^{-\lambda-2}(\T^d)} \norm{ \nabla T_{s,t}^n  \Phi(f)}_{H^{-\lambda-2}(\T^d)} \E \bigg[ \bigg| \Big \langle \int_{\T^d} \frac{\delta b}{\delta m} (s,\cdot,r \mu_s^N + (1-r)\mu_s , v ) \Id  \rho_{s}^N (v) \Id r  e_k(\cdot)   , \rho_s^N \Big \rangle \bigg] \bigg| \\
    &\quad \le  C [\Phi]_{C^1} N \sup\limits_{k \in \Z^d } \int\limits_0^1 \E \big[ \big|  \langle  \varphi_k(s,\cdot,r,\cdot )  , \mu_s^N \otimes \mu_s^N \rangle  \big| \big]  \Id r 
\end{align*}
with 
\begin{align*}
&\varphi_k(s,x,r,v) \\
&\quad := 
    \frac{\delta b}{\delta m} (s,x,r \mu_s^N + (1-r)\mu_s , v ) e_k (x) \rangle  - \langle  \frac{\delta b}{\delta m} (s,\cdot ,r \mu_s^N + (1-r)\mu_s , v ) e_k(\cdot)  , \mu_s \rangle \\
    &\quad \quad  - \langle \frac{\delta b}{\delta m} (s,x,r \mu_s^N + (1-r)\mu_s , \cdot  )  e_k(x) ,\mu_s \rangle + \langle \frac{\delta b}{\delta m} (s,\cdot,r \mu_s^N + (1-r)\mu_s , \cdot  )  e_k(\cdot) ,\mu_s \otimes \mu_s \rangle. 
\end{align*}
In the above computation we utilized the fact that \(\lambda +1 > d\) in order for the series to converge. 
    Applying the variational formula~\eqref{eq: variational_formula} for the relative entropy, we find 
    \begin{align*}
    & N \int\limits_0^1 \E\big[ |\langle  \varphi_k(s,\cdot,r,\cdot)  , \mu_s^N \otimes \mu_s^N \rangle| \big] \Id r
     \le  \frac{1}{\kappa }   \mathcal H(\bar\mu_s^N \vert  \mu_s^{\otimes N} ) \\
          &\quad +   \frac{1}{\kappa } \int\limits_0^1 \log \bigg( \int_{\T^{2N}} \bar\mu_s^{\otimes N} \exp\bigg(\kappa N \bigg|  \frac{1}{N^2}\sum\limits_{i,j=1}^N \varphi_k(s,x_i,r,x_j)  \bigg| \bigg)\Id \bm x \bigg)  \Id r 
    \end{align*}
    for every \(\kappa > 0\). Our goal is to apply Theorem~\ref{theorem: large_deviation}. Notice that the cancellation property 
    \[
\int_{\T^d} \varphi(s,x,r,v) \mu(x) \Id x = 0 \quad \forall v, \quad  
\int_{\T^d} \varphi(s,x,r,v) \mu(z) \Id v = 0 \quad \forall x
\]
holds. For the smallness condition~\eqref{eq: deviation_p_bound} we have 
\begin{equation*}
    |\varphi(s,x,r,v)|
    \le 4 \norm{\frac{\delta b}{\delta m} (s,\cdot,r \mu_s^N + (1-r)\mu_s , \cdot  ) }_{L^\infty(\T^d \times \T^d)} 
    \le C   , 
\end{equation*}
where we applied the boundedness of the flat derivative provideed by Assumption~\ref{ass: convergence_ass}.
Hence, choosing \(\kappa\) small enough, we can guarantee the smallness condition~\eqref{eq: deviation_p_bound} for the function \(\kappa \varphi(s,x,r,v)\) and we obtain 
\[
 \E |R^2_s(N,n,\Phi)| \le C [\Phi]_{C^1} \big( \mathcal H(\bar\mu_s^N \vert  \mu_s^{\otimes N} ) +1 \big)
\]
 Integration over the variable \(s\) and utilizing that the right hand side is uniformly bounded in time and in \(N \in \N\) by Assumption~\ref{ass: convergence_ass}, we deduce the claim.
\end{proof} 
\begin{remark}
The proof could alternatively be carried out using Pinsker's inequality, which relates the total variation norm to the relative entropy. However, the current line of proof will play a central role in more complex models, such as the Vortex model discussed in Section~\ref{sec: application}. 
\end{remark}

We now come to estimate the last term.

\begin{lemma} \label{lemma: second_flat_reminder}
Let Assumptions~\ref{ass: inital_cond}-\ref{ass: convergence_ass} hold and \(\Phi \in FC^\infty(H^{-\lambda-2}(\T^d))\). Then, the following inequality holds 
    \begin{equation*}
        \int_0^t \E|R^3_s(N,n,\Phi)| \Id s  \le  C [\Phi]_{C^1}. 
    \end{equation*}
\end{lemma}
\begin{proof}
    By~\eqref{cor: besov_reg_first_order_function} we can estimate the term in our statement by  
    \begin{align*}
    &\int\limits_0^t    \int\limits_0^1 r \int\limits_0^1 \int_{\T^d}  \E\bigg[   \bigg| \int_{\T^d}  \int_{\T^d} \frac{\delta^2 b}{\delta m^2} \bigg(s,x,r r'\mu_s^N+(1-rr')\mu_s ,v,v'\bigg) \Id v' \Id \rho_s^N(v')  \Id  \rho_s^N (v) \bigg| \bigg]   \Id \mu_s(x) \Id r' \dd r \Id s  \cdot   C [\Phi]_{C^1}.
    \end{align*}
    We observe that we are in a similar situation as in Lemma~\ref{lemma: flat_reminder}. By replacing the components of the function
\[
(x,v) \mapsto \frac{\delta b}{\delta m} \left(s,\cdot, r \mu_s^N + (1 - r)\mu_s , v \right) e_k(x) 
\]
with each component of the vector valued function
\[
(x,v,v') \mapsto \frac{\delta^2 b}{\delta m^2} \left(s,x, r r'\mu_s^N + (1 - r r')\mu_s , v, v'\right),
\]
w  e can carry out the same steps as in Lemma~\ref{lemma: flat_reminder}, now using the integration variables \((v,v')\) while keeping \(x\) fixed. Indeed, under Assumption~\ref{ass: convergence_ass}, we know that the second flat derivative is uniformly bounded. Therefore, by following the same reasoning as in Lemma~\ref{lemma: flat_reminder}, that is, performing a change of measure through the variational formula~\eqref{eq: variational_formula}, applying Theorem~\ref{theorem: large_deviation}, and using the bound on the relative entropy \(\mathcal H(\bar\mu_s^N \vert  \mu_s^{\otimes N})\), we get the result. 
\end{proof}

Having estimated all reminder terms in Proposition~\ref{prop: reminder}, we are finally in the position to prove our Main Theorem~\ref{theorem: main_result}, thanks to the  approximation (in probability) of $\rho$ by $\rho^n$ by Theorem ~\ref{thm: conv_prob}, and of \(\Phi \in C_\ell^2(H^{-\lambda-2}(\T^d)) \) by cylindrical functions by Lemma~\ref{lemma: approximation_hilbert_smooth}.

\begin{proof}[Proof of Theorem~\ref{theorem: main_result}] 

We first pass to the limit $n\to\infty$ in Proposition~\ref{prop: reminder}. Thanks to the convergence in probability on $L^2([0,T],H^{-\lambda-2}(\T^d))$ of $\rho^n$ to $\rho$, guaranteed by Theorem ~\ref{thm: conv_prob}, and dominated convergence, we obtain that 
\begin{equation*}
    \lim\limits_{n \to \infty}\E\big[ \Phi_m(\rho_t^n)] = \E\big[ \Phi_m(\rho_t)]. 
\end{equation*} 
for any \(\Phi_m \in FC^\infty(H^{-\lambda-2}(\T^d)\).
Therefore Proposition~\ref{prop: reminder} together with lemmas ~\ref{lem:r0}-~\ref{lemma: vanish_first_order}-
 ~\ref{lemma: second_ord_rest}-~\ref{lemma: flat_reminder}-~\ref{lemma: second_flat_reminder} provide
\begin{equation}
\label{eq:main_m}
 \big|\E \big[ \Phi_m(\rho_t^N)- \Phi_m(\rho_t) \big] \big|
 \le C [\Phi_m]_{C^2}\bigg( \frac{1}{\sqrt N } +   W_{1, H^{-\lambda-2}(\T^d)}\Big( \P_{\rho^N_0}, \P_{\rho_0}\Big) \bigg)
\end{equation} 
for any \(\Phi_m \in FC^\infty(H^{-\lambda-2}(\T^d))\). Indeed, convergence holds for almost every $t$, but then the estimate holds for any $t$ since the processes are continuous. 

    Then we show the estimate for a given $\Phi\in C^2_\ell(H^{-\lambda-2}(\T^d))$.  
     By Lemma~\ref{lemma: approximation_hilbert_smooth} we can find a sequence \((\Phi_m, m \in \N)\) in \( FC^\infty( H^{-\lambda-2}(\T^d))\) converging pointwise to \(\Phi\). 
Then by the linear growth of \((\Phi_m, m \in \N)\) and ~\eqref{eq:Phi_n_2}, we have 
\begin{equation*}
  |\Phi_m(\rho_t^N)- \Phi_m(\rho_t)| \le 4\|\Phi\|_{C_\ell(H^{-\lambda-2}(\T^d))} \big(1+ \norm{\rho_t^N}_{H^{-\lambda-2}(\T^d)}+   \norm{\rho_t}_{H^{-\lambda-2}(\T^d)} \big) . 
\end{equation*}
We demonstrate the uniform integrability of the right hand side by showing that it is square integrable. 
Utilizing~\cite[Lemma~2.6]{Xianliang2023} we find 
\begin{equation*}
    \E\bigg[  \norm{\rho_t^N}_{H^{-\lambda-2}(\T^d)}^2+   \norm{\rho_t}_{H^{-\lambda-2}(\T^d)}^2  \bigg]
    \le C(\big( \mathcal{H}(\bar \mu_t^N | \mu_t^{\otimes N} )^{\frac{1}{2}} +1  + \E\big[ \norm{\rho_t}_{H^{-\lambda-2}(\T^d)}^2 \big] \big),
\end{equation*}
where the right hand side is finite by Assumption~\ref{ass: convergence_ass} and the inequality~\eqref{eq:esti_spde}. 
Hence, we can utilize Vitali's convergence theorem~\cite[Theorem~4.5.4]{Bogachev2007} to obtain 
\begin{equation*}
    \big|\E\big[\Phi(\rho_t^N)- \Phi(\rho_t) \big] \big| 
    = \lim\limits_{m \to \infty} \big|\E \big[ \Phi_m(\rho_t^N)- \Phi_m(\rho_t) \big] \big|. 
\end{equation*}
Therefore we take the limit in the left hand side of~\eqref{eq:main_m} and use the bound~\eqref{eq:Phi_n_2} to get~\eqref{eq:main_estimate}. 
\end{proof}

We thus also prove Corollary~\ref{cor:conv_processes} on convergence in law of the processes.

\begin{proof}[Proof of Corollary ~\ref{cor:conv_processes} ]
Thanks to \cite[Lemma 16.2]{Kallenberg2002}, it is enough to prove convergence of the finite-dimensional distributions of $\rho^N$ to $\rho$. Let us fix $k\in\N$, $0=t_0< t_1<\dots<t_k\leq T$ and $\Phi_0, \Phi_1, \dots, \Phi_k \in C_b(H^{-\lambda-2}(\T^d))$. It is sufficient to show that 
\begin{equation}
\label{eq:conv_fdd}
\lim_{N\to \infty} \E^N\big[\Phi_0(\rho^N_{t_0})\Phi_1(\rho^N_{t_1}) \cdots \Phi_k(\rho^N_{t_k})\big] 
= \E\big[\Phi_0(\rho_{t_0})\Phi_1(\rho_{t_1}) \cdots \Phi_k(\rho_{t_k})\big]; 
\end{equation}
note that $\E^N$ depends on $N$ as we consider weak solutions of the particle system, while $\rho$ is a strong solution in a given probability space.  For this proof, we suppose that $\rho$ lives in a given probability  space, independent of all spaces where the $\rho^N$-s live.  
We proceed by induction, similarly to the proof of \cite[Thm 19.25]{Kallenberg2002}. The step zero holds true by assumption. Suppose then that \eqref{eq:conv_fdd} holds up to $t_{k-1}$. Let $(\Phi^m_k)_m$ be a sequence in $C_b(H^{-\lambda-2}(\T^d))$ given by lemma \ref{lemma: approximation_hilbert_smooth}, which approximates $\Phi_k$ uniformly on compact sets, and preserves the uniform norm. 
%
Theorem \ref{theorem: main_result} and the proof above also yield
\begin{equation}
\label{eq:conv_semiN}
|\E^N[\Phi^m_k(\rho_{t_{k-1},t_k}^N(f))] - \E[\Phi^m_k(\rho_{t_{k-1},t_k}(f))]| \le \frac{C}{\sqrt N } [\Phi^m_k]_{C^2(H^{-\lambda-2}(\T^d))} ,
\end{equation}
for any $f$ in $S^N_{k-1}:= \sqrt{N}(\mathcal{P}^N(\T^d)-\mu_{t_{k-1}})$ which is the state space where $\rho^N_{t_{k-1}}$ lives, where $C$ is independent of $f$ and $\mathcal{P}^N(\T^d)$ is the set of empirical measures. 
Since  $\mu^N$ is Markovian, the fluctuation process $\rho^N$ is also Markovian as it is a deterministic transformation of $\mu^N$ at any $t$. Hence, 
using the Markov property, \eqref{eq:conv_fdd} rewrites as 
\begin{equation*}
\lim_{N\to \infty} \E^N\big[\Phi_0(\rho^N_{t_0}) \cdots \Phi_{k-1}(\rho^N_{t_{k-1}}) 
T^N_{t_{k-1},t_k}\Phi_k(\rho^N_{t_k-1})\big] 
= \E\big[\Phi_0(\rho_{t_0}) \cdots \Phi_{k-1}(\rho_{t_{k-1}}) 
T_{t_{k-1},t_k}\Phi_k(\rho_{t_k-1}) \big]
\end{equation*} 
where $T^N$ and $T$ denote the semigroups of $\rho^N$ and $\rho$. Notice that the right hand side can still be derived, even though $\rho$ might not be Markovian, by approximating $\Phi_k(\rho_{t_k})$ with $T^n_{t_{k-1},t_k}\Phi_k(\rho_{t_k-1})$. 
We bound the difference above by 
\begin{align*}
&\big| \E^N\big[\Phi_0(\rho^N_{t_0}) \cdots \Phi_{k-1}(\rho^N_{t_{k-1}}) 
T^N_{t_{k-1},t_k}\Phi_k(\rho^N_{t_k-1})\big] 
- \E^N\big[\Phi_0(\rho^N_{t_0}) \cdots \Phi_{k-1}(\rho^N_{t_{k-1}}) 
T_{t_{k-1},t_k}\Phi_k(\rho^N_{t_k-1}) \big] \big|\\
&\quad + \big|\E^N\big[\Phi_0(\rho^N_{t_0}) \cdots \Phi_{k-1}(\rho^N_{t_{k-1}}) 
T_{t_{k-1},t_k}\Phi_k(\rho^N_{t_k-1})\big] 
- \E\big[\Phi_0(\rho_{t_0}) \cdots \Phi_{k-1}(\rho_{t_{k-1}}) 
T_{t_{k-1},t_k}\Phi_k(\rho_{t_k-1}) \big]\big|
\end{align*}
The second term goes to zero, as $N\to\infty$, by the inductive hypothesis, since $T_{t_{k-1},t_k}\Phi_k)(\cdot)$ is bounded and continuous thanks to \eqref{eq:LIP_flow} and dominated convergence. 

To bound the first term, we use the tightness assumption, which says that for any $\epsilon>0$ there exists a compact set $K^\epsilon \subset C([0,T], H^{-\lambda-2}(\T^d))$ such that $\P^N(\rho^N_\cdot \notin K^\epsilon) \leq \epsilon$ for any $N$.  As a consequence, for any $t$ there exists a compact $K^\epsilon_t \subset H^{-\lambda-2}(\T^d))$ such that $\P^N(\rho^N_t \notin K^\epsilon_t) \leq \epsilon$ for any $N$. We can now bound the first term as 
\begin{align*}
&\big| \E^N\big[\Phi_0(\rho^N_{t_0}) \cdots \Phi_{k-1}(\rho^N_{t_{k-1}}) 
T^N_{t_{k-1},t_k}\Phi_k(\rho^N_{t_k-1})\big] 
- \E^N\big[\Phi_0(\rho^N_{t_0}) \cdots \Phi_{k-1}(\rho^N_{t_{k-1}}) 
T_{t_{k-1},t_k}\Phi_k(\rho^N_{t_k-1}) \big] \big| \\
&\leq \big|\E^N\big[\Phi_0(\rho^N_{t_0}) \cdots \Phi_{k-1}(\rho^N_{t_{k-1}}) 
T^N_{t_{k-1},t_k}\Phi_k^m(\rho^N_{t_k-1})\big] 
- \E^N\big[\Phi_0(\rho^N_{t_0}) \cdots \Phi_{k-1}(\rho^N_{t_{k-1}}) 
T_{t_{k-1},t_k}\Phi_k^m(\rho^N_{t_k-1}) \big] \big| \\
&\quad +  \big|\E^N\big[\Phi_0(\rho^N_{t_0}) \cdots \Phi_{k-1}(\rho^N_{t_{k-1}}) 
T^N_{t_{k-1},t_k}\Phi_k^m(\rho^N_{t_k-1})\big] 
- \E^N\big[\Phi_0(\rho^N_{t_0}) \cdots \Phi_{k-1}(\rho^N_{t_{k-1}}) 
T^N_{t_{k-1},t_k}\Phi_k(\rho^N_{t_k-1}) \big] \big| \\
&\quad +  \big|\E^N\big[\Phi_0(\rho^N_{t_0}) \cdots \Phi_{k-1}(\rho^N_{t_{k-1}}) 
T_{t_{k-1},t_k}\Phi_k^m(\rho^N_{t_k-1})\big] 
- \E^N\big[\Phi_0(\rho^N_{t_0}) \cdots \Phi_{k-1}(\rho^N_{t_{k-1}}) 
T_{t_{k-1},t_k}\Phi_k(\rho^N_{t_k-1}) \big] \big|\\
&=: I_1 +I_2 + I_3
\end{align*} 
The term $I_1$ is bounded using \eqref{eq:conv_semiN}, which provides
\[
I_1 \leq \frac{C}{\sqrt N } [\Phi^m_k]_{C^2} 
||\Phi_0||_\infty \cdots  ||\Phi_{k-1}||_\infty.
\]
The term $I_2$ is bounded using the compact set as 
\begin{align*}
I_2 &= \big|\E^N\big[\big(\Phi_0(\rho^N_{t_0}) \cdots \Phi_{k-1}(\rho^N_{t_{k-1}}) 
\Phi_k^m(\rho^N_{t_k})- \Phi_0(\rho^N_{t_0}) \cdots \Phi_{k-1}(\rho^N_{t_{k-1}}) 
\Phi_k(\rho^N_{t_k-1}) \big) \big(\mathbbm{1}_{\rho^N_{t_k}\in K^\epsilon_{t_k}} + \mathbbm{1}_{\rho^N_{t_k}\notin K^\epsilon_{t_k}} \big) \big] \big| \\
&\leq ||\Phi_0||_\infty \cdots  ||\Phi_{k-1}||_\infty 
\Big( \sup_{f\in K^\epsilon_{t_k}} | \Phi^m_k(f) -\Phi_k(f)| 
+ 2 ||\Phi_{k}||_\infty  \P^N(\rho^N_{t_k}\notin K^\epsilon_{t_k} ) \Big)  \\
& \leq ||\Phi_0||_\infty \cdots  ||\Phi_{k-1}||_\infty 
 \sup_{f\in K^\epsilon_{t_k}} | \Phi^m_k(f) -\Phi_k(f)|  
 +2\epsilon ||\Phi_0||_\infty \cdots  ||\Phi_{k}||_\infty .
\end{align*} 
To bound the term $I_3$, we use the continuity of the flow of SPDE almost surely, given by \eqref{eq:LIP_flow}, to derive that, if the initial condition at time $t_{k-1}$ is $f\in K^\epsilon_{t_{k-1}}$, then for almost every $\omega$ there exists a compact $\widetilde K^\epsilon_{t_k}(\omega) \subset H^{-\lambda-2}(\T^d)$ such that  $\rho_{t_{k-1}, t_k}(f)(\omega) \in \widetilde K^\epsilon_{t_k}(\omega)$.  Thus we get 
\begin{align*}
I_3 &\leq ||\Phi_0||_\infty \cdots  ||\Phi_{k-1}||_\infty  \big(||T_{t_{k-1},t_k}\Phi_k^m ||_\infty+ ||T_{t_{k-1},t_k}\Phi_k||_\infty \big) \P^N(\rho^N_{t_{k-1}}\notin K^\epsilon_{t_k} ) \\
& \quad + ||\Phi_0||_\infty \cdots  ||\Phi_{k-1}||_\infty \sup_{f\in K^\epsilon_{t_{k-1}}} | T_{t_{k-1},t_k}(\Phi_k^m-\Phi_k)(f) | \\
&\leq 2\epsilon ||\Phi_0||_\infty \cdots  ||\Phi_{k}||_\infty \\
&\quad + ||\Phi_0||_\infty \cdots  ||\Phi_{k-1}||_\infty  \sup_{f\in K^\epsilon_{t_{k-1}}} \E \big| (\Phi_k^m-\Phi_k)   (\rho_{t_{k-1},t_k}(f)) \big|,
\end{align*}
where we used the contraction property of the semigroup, and we bound the latter term as 
\begin{align*}
\int_\Omega \sup_{f\in K^\epsilon_{t_{k-1}}} | \Phi_k^m-\Phi_k|( \rho_{t_{k-1},t_k}(f)(\omega) )\Id \P(\omega) 
\leq \int_\Omega \sup_{\tilde{f}\in \widetilde{K}^\epsilon_{t_{k}}(\omega)} | \Phi_k^m-\Phi_k|( \tilde{f} )\Id \P(\omega),  
\end{align*} 
which goes to zero, as $m\to\infty$, by the convergence on compacts and dominated convergence. Putting things together, sending $N\to\infty$ first, then $m\to\infty$ and finally $\epsilon\to 0$, we obtain \eqref{eq:conv_fdd}. 
\end{proof}

\section{Applications} \label{sec: application}
The aim of this section is to extend our Main Theorem~\ref{theorem: main_result} to more challenging models that do not satisfy the boundedness assumption on the flat derivatives stated in Assumption~\ref{ass: convergence_ass}. The starting point of our analysis is Proposition~\ref{prop: reminder}. By employing different techniques, in particular the methods developed for singular interaction kernels in the two seminal works~\cite{JabinWang2018,Serfaty2020}, we are able to estimate the remainder terms appearing in Proposition~\ref{prop: reminder} in two crucial cases: the first is when the interacting particle system is given by the point Vortex model, and the second is when the interaction kernel is the repulsive Coulomb potential.

\subsection{Vortex model}
\label{sec:vortex}
In this section we consider the Vortex model for approximating the 2D Navier-Stokes equation in the vorticity formulation, which is a guiding example in the literature~\cite{Osada1986,Fournier2014,JabinWang2018,Xianliang2023}. 
Let \(d=2\) and let \(b(t,x,m)= K*m\), where \(K\) is the Biot--Savart kernel 
\begin{equation} \label{eq: biot_savart_kernel}
K(x_1,x_2) = \frac{1}{2\pi}\frac{(-x_2,x_1)}{|x|^2} + K_0(x_1,x_2), \quad x=(x_1,x_2) \in \T^2, 
\end{equation}
where \(K_0\) is a smooth correction to periodize \(K\) on the torus. 

Notice that the kernel exhibits a singularity at the origin and is divergence-free. Consequently, we cannot directly apply our previous results, which rely on the regularity properties of the flat derivative \(\tfrac{\delta b}{\delta m}(t,x,m,v)\). As noted in Lemma~\ref{rem: convolution_estimate}, in the case of convolution-type interactions, this derivative reduces to \(k(x - v)\).
Our objective is to verify Assumptions~\ref{ass: inital_cond}--\ref{ass: coef_fokker} and to estimate the remainder term appearing in the generator expression in Proposition~\ref{prop: reminder}.

Recall that we assume \(\lambda > \tfrac{3}{2} d = 3\) and \(\lambda' > \lambda + 1\), and that the limiting measure \(\mu\) solves the Fokker--Planck equation~\eqref{eq: fokker--planck} with the Biot--Savart kernel \(K\) defined in~\eqref{eq: biot_savart_kernel}.
Unless stated otherwise, all relevant quantities such as \(\rho^N\), \(\mu^N\), and \(\rho\) are to be understood as solutions to their respective equations with the drift term given by the convolution \(b(t,x,m) = K * m(x)\). 

We begin by verifying that the underlying mathematical objects are well-defined. Our first step is to establish the existence of the interacting particle system. To this end, we recall the following result from Osada~\cite{Osada1985}.

\begin{theorem} \label{theorem: existence_vortex_particle}
Consider any family \((X^{i}_0,i=1,\dots,N)\) of \( \T^2\)-valued random variables, independent of a family \((B^{i},i=1,\dots,N)\) of independent and identically distributed two-dimensional Brownian motions, and such that almost surely,
\begin{equation} \label{eq: no_collision}
X^{i}_0 \neq X^{j}_0 \quad \text{for all } i \neq j.
\end{equation}
Then there exists a unique strong solution to system~\eqref{eq: interacting_particle_system}.
\end{theorem}
Hence, in the following we fix our probabilistic setting such that \((X^{i}_0,i=1,\dots,N)\) and \((B^{i},i=1,\dots,N)\) are independent. Note that the initial conditions are not assumed to be i.i.d. Then condition~\eqref{eq: no_collision} guarantee the existence of the interacting particle system~\eqref{eq: interacting_particle_system}.
Next, by Lemma~\ref{rem: convolution_estimate}, it suffices to verify that \(K \in L^1(\T^2)\) and \(\mu_t \in C([0,T]; H^{\lambda'}(\T^2))\) in order to ensure that Assumptions~\ref{ass: existence} and~\ref{ass: coef_fokker} are satisfied. 

The condition \(K \in L^1(\T^2)\) holds, as the singularity of the Biot--Savart kernel is integrable near the origin. 
For the second condition, namely \(\mu_t \in C([0,T]; H^{\lambda'}(\T^2))\), we may assume that the initial distribution \(\mu_0\) satisfies \(\mu_0 \in B_{\infty,\infty}^{\lambda'}(\T^2)\) and \(\inf_{x \in \T^2} \mu_0(x) > 0\). According to the proof of~\cite[Theorem~1.7]{Xianliang2023}, these regularity and positivity conditions on the initial data imply that \(\mu \in C([0,T]; H^{\lambda'}(\T^2))\) and that \(\inf_{x \in \T^2} \mu_t(x) > 0\) for all \(t \in [0,T]\). These properties ensure the validity of the two crucial estimates~\eqref{eq: linear_case_1} and~\eqref{eq: linear_case_2}.

Therefore, Assumptions~\ref{ass: inital_cond}--\ref{ass: coef_fokker} are dependent on the choice of initial data. Let us now formulate the corresponding assumptions specifically for the Vortex model.

\begin{assumption}[Vortex model]\label{ass: vortex_mode}
Suppose that \(\rho_0 \in L^2_{\cF_0}(H^{-\lambda-2}(\T^d))\) and that the initial positions \((X_0^{i})_{i \in \N}\) 
satisfy~\eqref{eq: no_collision}. Assume that \(\mu_0 \in B_{\infty,\infty}^{\lambda'}(\T^2)\), \(\inf_{x \in \T^2} \mu_0(x) > 0\) and \[ \sup\limits_{N \in \mathbb{N}} \mathcal{H}\big( \bar{\mu}_0^N \,\vert\, \mu_0^{\otimes N} \big) < \infty. \] 
%
\end{assumption}
\begin{remark} \label{remark: vortex}
By the preceding discussion, it is evident that Assumption~\ref{ass: vortex_mode} implies Assumptions~\ref{ass: inital_cond}–\ref{ass: coef_fokker}.
Furthermore, Assumption~\ref{ass: vortex_mode} in combination with~\cite[Theorem~1]{JabinWang2018} guarantees the key relative entropy estimate 
\begin{equation} \label{eq: vortex_relative_entropy_bound}
   \sup\limits_{N \in \mathbb{N}} \sup\limits_{0 \le t \le T } \mathcal{H}( \bar{\mu}_t^N \,\vert\, \mu_t^{\otimes N} ) < \infty,
\end{equation}
which plays a crucial role in estimating the remainder terms.
\end{remark}

As a consequence of Remark~\ref{remark: vortex}, Remark~\ref{remark: second_flat_vanish} and Proposition~\ref{prop: reminder}, together with Lemmas~\ref{lem:r0}-~\ref{lemma: vanish_first_order} and the limit $\rho^n\to\rho$, we have the following: 

\begin{corollary} \label{cor: vortex}
    Let Assumption~\ref{ass: vortex_mode} hold and \(\Phi \in FC^\infty(H^{-\lambda-2}(\T^2))\). We have 
    \begin{align*}
     &    |\E\big[ \Phi(\rho_t^N)- \Phi(\rho_t) \big] |
     \le \frac{1}{\sqrt{N}} \limsup\limits_{n \to \infty} \int\limits_0^t \E\bigg[    \frac{\sigma^2}{2} \sum\limits_{j=1}^2 \int_{\T^d}
       \partial_{x_j} \partial_{y_j}  ( \nabla^2 \hat T_{s,t}^n  \Phi(\rho_{s}^N))(x,y)_{\big| x=y } 
    \Id  \rho_s^N ( x)\bigg]  \\
     & \quad \quad  + \E \big[ \langle   (K*  \rho_{s}^N)(\cdot) \cdot D ( \nabla \hat T_{s,t}^n  \Phi(\rho_s^N))(\cdot)  , \rho_s^N  \rangle \big] \Id s + C [\Phi]_{C^1} W_{1, H^{-\lambda-2}(\T^d)}\Big( \P_{\rho^N_0}, \P_{\rho_0}\Big)  . 
    \end{align*}
\end{corollary}

Hence, it remains to estimate the rest terms. We follow similar strategies as in the proof of Theorem~\ref{theorem: main_result}. 
\begin{lemma}\label{lemma: vortex_first_order}
Let Assumption~\ref{ass: vortex_mode} hold and \(\Phi \in C_\ell^2(H^{-\lambda-2}(\T^d))\). Then,  
    \begin{equation}
        \E\big[\langle K*(\mu_t^N -\mu_t ), D ( \nabla \hat T_{s,t}^n \Phi (\rho_s^N))(\cdot) (\mu_t^N-\mu_t) \rangle \big] \le \frac{C(T,\mu_0) [\Phi]_{C^1} }{N}. 
    \end{equation}
\end{lemma}
\begin{proof}
Using the definition of \(\nabla \hat T_{s,t}^n \Phi (\rho_s^N))(\cdot)\) and the same computations as in Lemma~\ref{lemma: flat_reminder} we find 
\begin{align} \label{eq: vortex_first_order_aux}
\begin{split}
     &\E \big[ \langle K*(\mu_t^N -\mu_t ), D ( \nabla \hat T_{s,t}^n \Phi (\rho_s^N))(\cdot) (\mu_t^N-\mu_t) \rangle  \big]  \\
     &\quad \le C [\Phi]_{C^1} \sum\limits_{k \in \Z^2} \langle k \rangle^{-\lambda-1}  \E\big[  |\langle K*(\mu_t^N -\mu_t ),  e_k (\mu_t^N-\mu_t) \rangle| \big] . 
     \end{split}
\end{align}
    Applying the variational formula~\eqref{eq: variational_formula} for the relative entropy, we find 
    \begin{align*}
          &\E \big[ |\langle K*(\mu_t^N -\mu_t ), e_k (\mu_t^N-\mu_t) \rangle|\big] - \frac{1}{\kappa N}   \mathcal H(\bar\mu_t^N \vert \bar \mu_t^{\otimes N} ) \\
          &\quad \le   \frac{1}{\kappa N}  \log \bigg( \int_{\T^{2N}} \bar\mu_t^{\otimes N} \exp\bigg(\kappa N \bigg\langle K*\bigg (\frac{1}{N} \sum\limits_{i=1}^N \delta_{x_i}  -\mu_t \bigg), e_k(\cdot) \bigg(\frac{1}{N} \sum\limits_{i=1}^N \delta_{x_i} -\mu_t\bigg) \bigg \rangle\bigg) \Id \bm x \bigg) 
    \end{align*}
    for any \(\kappa >0\). 
    Let \(\Pi_t^N(A)=\tfrac{1}{N} \sum\limits_{i=1}^N \delta_{x_i}(A).\) for some measurable set \(A\). 
    Then, 
    \begin{align*}
       \langle K*(\Pi_t^N -\mu_t ),e_k(\cdot) (\Pi_t^N-\mu_t) \rangle  &= \langle K*\Pi_t^N ,e_k(\cdot) \Pi_t^N \rangle
       - \langle K*\mu_t ,e_k(\cdot) \Pi_t^N \rangle\\
       &\quad  -  \langle K*\Pi_t^N ,e_k(\cdot) \mu_t \rangle 
       +  \langle K*\mu_t ,e_k(\cdot) \mu_t \rangle. 
    \end{align*}
    For the first term, we apply the symmetrization trick 
    \begin{align*}
        &\langle K*\Pi_t^N,e_k(\cdot) \Pi_t^N \rangle \\
        &\quad = \int_{\T^{4}}  K(x-y) \cdot e_k(x) \Id \Pi_t^N (y) \Id \Pi_t^N (x) \\
        &\quad = \frac{1}{2}  \int_{\T^{4}}  K(x-y) ( e_k(x)- e_k(y))  \Id \Pi_t^N (y) \Id \Pi_t^N (x). 
    \end{align*}
   The same symmetrization applies to the last term. However, for the second and third terms, the symmetrization procedure above produces the corresponding term for the other one. Denote by 
   \begin{equation*}
       \mathds K_{k} (x,y) := \frac{1}{2} K(x-y) \cdot (e_k(x)-e_k(y))
   \end{equation*}
    Define the function 
    \begin{equation*}
        \varphi_k(t,x,y) = \mathds K_{k}- \langle \mathds K_{k }(x,\cdot),\mu_t \rangle -  \langle \mathds K_{k}(\cdot,y),\mu_t \rangle + \langle \mathds K_{k}(x,y),\mu_t \otimes \mu_t \rangle
    \end{equation*}
    and observe 
    \begin{equation*}
        \langle K*\Pi_t^N ,e_k(\cdot) \Pi_t^N \rangle 
        = \langle  \varphi_k(t,x,y), \Pi_t^N \otimes \Pi_t^N \rangle
        = \frac{1}{N^2} \sum\limits_{i,j=1}^N  \varphi_k(t,x_i,x_j)
    \end{equation*}
    and 
    \begin{equation*}
        \int_{\T^2} \varphi_k(t,x,y) \mu_t(x) \Id x = 0, \quad \forall y \in \T^2 , \quad \int_{\T^2} \varphi_k(t,x,y) \mu_t(y) \Id y = 0, \quad \forall x \in \T^2 . 
    \end{equation*}
    Hence, we have verified the cancellation conditions of Theorem~\ref{theorem: large_deviation}. Next, our goal is to demonstrate the inequality~\eqref{eq: deviation_p_bound}. 
    We start by demonstrating the boundedness of the function \(\mathds K_{k }\). We have 
    \begin{equation*}
        |\mathds K_{k } (x,y)|
        \le |x-y| |K(x-y)| |D e_k(x+a(x-y))| \le C |k| \sup\limits_{x \in \R^2} |x K(x)| \le C_{\mathrm{vortex}} |k|
    \end{equation*}
    for some fix constant \(C_{\mathrm{vortex}}\) depending on the periodic correction. 
Consequently, applying Lemma~\ref{cor: reg_derivative} we obtain
\begin{equation*}
    |\varphi_k(t,x,y)| \le 4 C_{\mathrm{vortex}} |k| . 
\end{equation*} 
Choosing \(\kappa = (8 \hat C C_{\mathrm{vortex}} (1+|k|))^{-1}\), where \(\hat C\) is the constant provided by Theorem~\ref{theorem: large_deviation}, we can verify \eqref{eq: deviation_p_bound}. Hence, we can use Theorem~\ref{theorem: large_deviation} to find 
\begin{equation*}
    \E \big[ |\langle K*(\mu_t^N -\mu_t ), e_k(\cdot) (\mu_t^N-\mu_t) \rangle | \big] \le  \frac{1}{\kappa N}   (\mathcal H(\bar\mu_t^N \vert \bar \mu_t^{\otimes N} ) +\log(4) ). 
\end{equation*}
Since \(\kappa\) depends on \(k\) we need to plug the above bound into~\eqref{eq: vortex_first_order_aux} to obtain 
\begin{align*}
     &\E\big[ \langle K*(\mu_t^N -\mu_t ), D ( \nabla \hat T_{s,t}^n \Phi (\rho_s^N))(\cdot) (\mu_t^N-\mu_t) \rangle \big]  \\
     &\quad \le \frac{C}{N} [\Phi]_{C^1} \big(\mathcal H(\bar\mu_t^N \vert \bar \mu_t^{\otimes N} ) +\log(4) \big) \sum\limits_{k \in \Z^d} \langle k \rangle^{-\lambda-1} (1+|k|)  \\
     &\quad \le  \frac{C}{N} [\Phi]_{C^1} \Big(\sup\limits_{0 \le t \le T} \mathcal H(\bar\mu_t^N \vert \bar \mu_t^{\otimes N} ) +1 \Big)
\end{align*}
for some new constant \(C\). We again utilized the fact that \(\lambda > d=2\) to bound the infinite series. 
Recall that Assumption~\ref{ass: vortex_mode} is sufficient to apply the relative entropy bound~\eqref{eq: vortex_relative_entropy_bound} from~\cite{JabinWang2018}, which bounds \(\sup\limits_{0 \le t \le T} \mathcal H(\bar\mu_t^N \vert \bar \mu_t^{\otimes N} )\) uniformly in \(N\in \N\) and proves our claim. 
\end{proof}

It remains to estimate the second order remaining term. 
We start by 
\begin{lemma} \label{lemma: vortex_second_order}
Let Assumption~\ref{ass: vortex_mode} hold and \(\Phi \in FC^\infty(H^{-\lambda-2}(\T^2))\).
Then,  
\begin{equation*} 
    \int\limits_0^t \E\bigg[ 
     \frac{\sigma^2}{2} \sum\limits_{j=1}^2 \int_{\T^d}
       \partial_{x_j} \partial_{y_j}  ( \nabla^2 \hat T_{s,t}^n  \Phi(\rho_{s}^N))(x,y)_{\big| x=y } 
    \Id \rho_s^N  ( x)  \bigg]  \Id s 
    \le C [\Phi]_{C^2}. 
\end{equation*}
\end{lemma}
\begin{proof}
Recall~\eqref{eq: second_order_rest_term_aux}, which states, without assuming boundedness of $\tfrac{\delta}{\delta m}$,  
    \begin{equation*}
    \int\limits_0^t \E\bigg[  \int_{\T^d}
       \partial_{x_j} \partial_{y_j}  ( \nabla^2 \hat T_{s,t}^n  \Phi(\rho_{s}^N))(x,y)_{\big| x=y } 
    \Id \rho_s^N  ( x)\bigg]  \Id s  \le C T [\Phi]_{C^2} \Big(\sup\limits_{ 0 \le t \le T}  \mathcal{H}(\bar \mu_t^N | \mu_t^{\otimes N} )^{\frac{1}{2}} +1 \Big). 
    \end{equation*}
Plugging the estimate~\eqref{eq: vortex_relative_entropy_bound} into the previous inequality proves the claim.  
\end{proof}

As a consequence of Corollary~\ref{cor: vortex}, Lemma~\ref{lemma: vortex_first_order} and Lemma~\ref{lemma: vortex_second_order}, and again the approximation by cylindrical functions, we obtain the main result of the section, which is a convergence rate for the weak fluctuation of the point Vortex model. 

\begin{theorem}
\label{thm:main:vortex}
    Let \(K\) be given by~\eqref{eq: biot_savart_kernel} and let Assumption~\ref{ass: vortex_mode} hold. 
Then, for all \(\Phi \in C_\ell^2(H^{-\lambda-2}(\T^2))\), we have 
\begin{equation*}
   \sup\limits_{0 \le t  \le T  } \big|\E\big[ \Phi(\rho_t^N) - \Phi(\rho_t)\big]\big|
\le  C [\Phi]_{C^2(H^{-\lambda-2}(\T^2))}  \bigg( \frac{1}{\sqrt N } +   W_{1, H^{-\lambda-2}(\T^d)}\Big( \P_{\rho^N_0}, \P_{\rho_0}\Big)   \bigg). 
\end{equation*}
\end{theorem}
\begin{proof}
    Since we have the bound on the relative entropy by~\cite[Theorem~1]{JabinWang2018}, we can follow the proof of Theorem~\ref{theorem: main_result}. 
\end{proof}

\subsection{Repulsive Coulomb potential}
\label{sec:couulomb}

We consider here the following interacting particle system
\begin{equation} \label{eq: repulsive_kernel}
    \Id X_t^{i} = \sum\limits_{\substack{j=1 \\ j \neq i}}^N \nabla G(X_t^{i}-X_t^{j}) \Id t + \sigma \Id B_t^{i},
\end{equation}
where \(G(x)=\ln(|x|) + G_0(x)\) for \(d = 2\), \(G(x)=|x|^{-(d-2)} + G_0(x)\) for \(d = 3\) and \(G_0\) is a smooth correction depending on the dimension \(d\). The potential \(G\) up to a normalizing constant is characterized as the solution of the Laplace equation \(-\Delta G = \delta_0\). 
We recall the result~\cite[Theorem~1.1]{Liu2016} on the strong well-posedness of the interacting particle system~\eqref{eq: repulsive_kernel}.  Note that such result is formulated for the Euclidean space \(\R^d\) but the proof can be adapted to the torus \(\T^d\). 
\begin{theorem}
    Let \(N \ge 2\) and \((X_0^{i}, i \in \N)\) be i.i.d. with common distribution \(\mu_0 \in L^{\frac{2d}{d+2}}(\T^d)\) and 
    \[
    H(\mu_0) := \int_{\T^d} \log (\mu_0(x)) \mu_0(x) \Id x
    < \infty,
    \] and independent of the Brownian motions \((B^{i}, i \in \N)\). Then there exists a unique global strong solution to~\eqref{eq: repulsive_kernel}. 
\end{theorem}
For the existence of a solution to the Fokker--Planck equation, we recall~\cite[Theorem~1.1]{Serfaty2025} with a combination of Sobolev's embedding~\cite[Proposition~4.6]{Triebel2006}. 
\begin{theorem}
    Let \(\mu_0 \in B_{\infty,\infty}^{\lambda'+\epsilon}(\T^d)  \) for some \(\epsilon > 0\). Then, there exists a solution \((\mu_t, 0 \leq t \leq T)\) of~\eqref{eq: fokker--planck} such that \(\mu \in C([0,T],B_{\infty,\infty}^{\lambda'}(\T^d))\). 
\end{theorem}
Based on the above Theorems we formulate the following Assumption for the initial condition in the repulsive Coulomb case. 

\begin{assumption}[Repulsive Coulomb]\label{ass: repulsive_coulomb}
    Suppose \((X_0^{i}, i \in \N)\) be i.i.d. with common distribution \(\mu_0 \in B_{\infty,\infty}^{\lambda'+\epsilon}(\T^d) \) for some \(\epsilon >0\), independent of the Brownian motions \((B^{i}, i \in \N)\), \(H(\mu_0) < \infty\) and \(\inf_{x \in \T^d} \mu_0(x) > 0\). 
    Moreover, assume 
    \begin{equation*}
        \int\limits_{\T^{dN}} \log \bigg (\bar \mu_0^N(x_1, \ldots, x_N) \exp\bigg( \frac{1}{N\sigma^2}\sum\limits_{1 \le i \neq j \le N} G(x_i-x_j)\bigg) \bigg ) \Id \bar \mu_0^{N}< \infty
    \end{equation*}
    and \(\rho_0 \in  L^2_{\cF_0}(H^{-\lambda-2}(\T^d)) \). 
\end{assumption}
Since \(\nabla G \in L^1(\T^d)\), we can apply Lemma~\ref{rem: convolution_estimate} to deduce that Assumption~\ref{ass: repulsive_coulomb} implies Assumptions~\ref{ass: inital_cond}–\ref{ass: coef_fokker}.
As for Corollary\ref{cor: vortex}, as consequence of Proposition~\ref{prop: reminder} we obtain the following: 
\begin{corollary} \label{cor: repulsive_coulomb}
    Let Assumption~\ref{ass: repulsive_coulomb} hold and \(\Phi \in FC^\infty(H^{-\lambda-2}(\T^d))\). We have 
    \begin{align*}
     &    \big|\E\big[ \Phi(\rho_t^N)- \Phi(\rho_t) \big] \big|
     \le \frac{1}{\sqrt{N}} \limsup\limits_{n \to \infty} \int\limits_0^t \E\bigg[   \frac{\sigma^2}{2} \sum\limits_{j=1}^d \int_{\T^d}
       \partial_{x_j} \partial_{y_j}  ( \nabla^2 \hat T_{s,t}^n  \Phi(\rho_{s}^N))(x,y)_{\big| x=y } 
    \Id  \rho_s^N ( x)\bigg] \\
     & \quad \quad  + \E \big[ \langle   (\nabla G*  \rho_{s}^N)(\cdot) \cdot D ( \nabla \hat T_{s,t}^n  \Phi(\rho_s^N))(\cdot)  , \rho_s^N  \rangle \big]  \Id s + C [\Phi]_{C^1} W_{1, H^{-\lambda-2}(\T^d)}\Big( \P_{\rho^N_0}, \P_{\rho_0}\Big) . 
    \end{align*}
\end{corollary}

In order to estimate the reminder terms, we utilize the the modulated energy given by 
\begin{equation} \label{eq: modulated_energy}
    F_N((x_1,\ldots,x_N),\mu) := \frac{1}{\sigma^2} \int\limits_{\T^{2d} \setminus D} G (x-y)  \Id  \bigg(\frac{1}{N}\sum\limits_{i=1}^N \delta_{x_i} - \mu\bigg)^{\otimes 2} (x,y) ,
\end{equation}
where \(D=\{(x,y) \in \T^d \times \T^d: \; x=y\}\) denotes the diagonal. 
The above quantity was utilized to demonstrate mean field limits for kernels, which are close to the Coulomb kernel. One can see that \(F_N\) can be understood as a renormalization of the negative-order homogeneous Sobolev norm corresponding to the Fokker--Planck eqaution~\eqref{eq: fokker--planck}.
More, presicely the mdoulated energy is coercive in the sense of~\cite[Proposition~3.6]{Serfaty2020}. Our goal is to connect \(F_N\) with the reminder terms. 
We start with the first order terms in Corollary~\ref{cor: repulsive_coulomb}. 
Notice that \(\nabla G\) is antisymmetric and we can apply a symmetrization trick to obtain 
\begin{align} \label{eq: riesz_aux2}
\begin{split}
    &\int\limits_{\T^{2d} \setminus D} \nabla G(x-y) \cdot D (\nabla \hat T^n_{s,t} \Phi(\rho_s^N)(x) \Id (\mu_s^N-\rho_s)^{\otimes 2}(x,y)\\
    &\quad = \frac{1}{2} \int\limits_{\T^{2d} \setminus D} \nabla G(x-y) \cdot \big( D (\nabla \hat T^n_{s,t} \Phi(\rho_s^N)(x)-D (\nabla \hat T^n_{s,t} \Phi(\rho_s^N)(y) \big)  \Id (\mu_s^N-\rho_s)^{\otimes 2}(x,y). 
    \end{split}
\end{align}
Terms of the above were highly analyzed by a series of works by Serfaty, Rosenzweig et al.~\cite{Serfaty2025}. 
The above expression arises naturally by pushing forward the empirical measure 
\(
\frac{1}{N} \sum_{i=1}^{N} \delta_{x_i}
\)
under the transport map \(\mathrm{Id} + t D (\nabla \hat T^n_{s,t} \Phi(\rho_s^N))\) in the modulated energy \(F_N((x_1,\ldots,x_N), \mu)\), and computing the first derivative at \(t = 0\); in other words, by evaluating the first variation of the modulated energy along the vector field \(D (\nabla \hat T^n_{s,t} \Phi(\rho_s^N))\).

We recall the crucial sharp functional inequality~\cite[Proposition~2.13]{Serfaty2025}. 
\begin{proposition}[Sharp functional inequality]\label{prop: riesz_func_inequ}
Assume \(\mu \in L^1(\T^d)\) satisfies \(\int_{\T^d} \mu = 1\). For any pairwise distinct configuration \((x_1,\ldots,x_N) \in (\T^d)^N\) and any Lipschitz map \(v \colon \T^d \to \R^d\), we have
\begin{align*}
&\left| 
\int_{(\T^d)^2\setminus D}  \nabla G(x - y) \cdot \left( v(x) - v(y) \right) \Id \left( \frac{1}{N} \sum_{i=1}^{N} \delta_{x_i} - \mu \right)^{\otimes 2}(x, y)
\right|\\
&\quad \leq \|\nabla v\|_{L^\infty(\T^d)} \left(
F_N((x_1,\ldots,x_N), \mu) + \frac{\log(N\|\nu\|_{L^\infty(\T^d)})}{4N} \indicator{d=2} + C \|\mu\|_{L^\infty(\T^d)}^{(d-2)/d} N^{-2/d}
\right),
\end{align*}
where \(C > 0\) depends only on \(d\).
\end{proposition}
\begin{remark}
The above bound holds not only in the Coulomb case but also in the super-Coulomb regime. However, to use known existence results of the interacting particle system we require the Riesz potential parameter \(\theta\) to lie in the range \([0, d - 2]\). Moreover, a weaker estimate is available for general Riesz potentials, where the corresponding rate is given by \(-\tfrac{d - \theta}{d(d + 1)}\).

For \(\theta < d - 2\), this convergence rate becomes too weak to compensate the fluctuation scaling. More precisely, the inequality
\[
\frac{1}{2} - \frac{d - \theta}{d(d + 1)} < 0
\]
can only hold if \(\theta < 0\), which falls outside the admissible range. Therefore, we are restricted to the critical case \(\theta = d - 2\), corresponding to the Coulomb interaction, where the stronger functional inequality stated in Proposition~\ref{prop: riesz_func_inequ} applies.

In dimension \(d = 2\), this yields the favorable rate \(\log(N)/N\), while for \(d = 3\), the rate \(N^{-2/3}\) is sufficient to balance the fluctuation scaling. This highlights the special role of the Coulomb case.
\end{remark}
 To proceed, we require an estimate on the modulated energy. Fortunately, such an estimate is provided in~\cite[Theorem~1.2]{Serfaty2025}. Combining it with~\cite[Theorem~1.2 and inequality~(6.10)]{Serfaty2025} we also obtain a relative entropy estimate necessary for the other term in Corollary~\ref{cor: repulsive_coulomb}. 
\begin{theorem}\label{theorem: repulsive_coloumb}
    Let Assumption~\ref{ass: repulsive_coulomb} hold. Then there exists a constant \(C=C(T,\mu_0)\) depending on \(T\) and the initial condition \(\mu_0\) such that 
    \begin{equation*}
        \sup\limits_{0 \le t \le T} \mathcal H(\bar \mu_t^N \vert \mu_t^{\otimes N} )
    \le  C\Big|\log\Big(N \sup\limits_{0 \le t \le T} \|\mu_t\|_{L^\infty(\T^d)}\Big)\Big| \indicator{d=2} + C \sup\limits_{0 \le t \le T} \|\mu_t\|_{L^\infty(\T^d)}^{(d-2)/d} N^{-2/d+1}
    \end{equation*}
    and 
    \begin{equation*}
        \sup\limits_{0 \le t \le T} \E \big[  F_N((X^1_t,\ldots,X^N_t), \mu) \big]  
    \!\le  C\Big|\log\Big(N \sup\limits_{0 \le t \le T} \!\|\mu_t\|_{L^\infty(\T^d)}\Big)\Big| N^{-1} \indicator{d=2} 
    + C \sup\limits_{0 \le t \le T} \!\|\mu_t\|_{L^\infty(\T^d)}^{(d-2)/d} N^{-2/d}. 
    \end{equation*}
\end{theorem}
\begin{remark}
    Theorem~1.2 in~\cite{Serfaty2025} is stated under additional assumptions, such as the existence of an entropy solution. We ensure that this condition is satisfied by selecting sufficiently regular initial data as specified in Assumption~\ref{ass: repulsive_coulomb}. See, for instance,~\cite[Lemma~6.2]{Serfaty2025}.
\end{remark}

Combining Proposition~\ref{prop: riesz_func_inequ} and Theorem~\ref{theorem: repulsive_coloumb} leads to the following Lemma. 
\begin{lemma}\label{lemma: repulsive_first_order}
There exists a \(\gamma(d,N) > 0\) given by 
\begin{equation} \label{eq: vortex_rate}
    \gamma(d,N):= \begin{cases}
        \frac{\log(N)}{\sqrt{N}} \quad \mathrm{if} \; d=2,  \\
        N^{-\tfrac{1}{6} } \quad \mathrm{if} \; d=3,
    \end{cases}
\end{equation}
such that 
\begin{equation*}
    \frac{1}{\sqrt{N}} \int\limits_0^t \E \big[  \langle   (\nabla G*  \rho_{s}^N)(\cdot) \cdot D ( \nabla \hat T_{s,t}^n  \Phi(\rho_s^N))(\cdot)  , \rho_s^N  \rangle \big]  \Id s  \le C  [\Phi]_{C^1} \gamma(d,N). 
\end{equation*}
\end{lemma}

\begin{proof}
Rewriting the left hand side by~\eqref{eq: riesz_aux2}, we observe that the conclusion of Proposition~\ref{prop: riesz_func_inequ} applies, provided that for each \(\omega \in \Omega \) the map  \( x \mapsto D (\nabla \hat T_{s,t}^n \Phi(\rho_s^N(\omega)))(x)\) is Lipschitz continuous. 
By~\eqref{cor: besov_reg_first_order_function}, the function \(\nabla \hat T_{s,t}^n \Phi(\rho_s^N(\omega))(\cdot)\) belongs to the Besov space \(B_{\infty,\infty}^{\lambda+2 - d/2}(\T^d)\), with a uniform bound on its norm independent of \(n \in \mathbb{N}\) and \(\omega \in \Omega\). Since \(\lambda + 2 - d/2 > 2\) the derivative of the function is Lipschitz continuous and the claim follows by applying Proposition~\ref{prop: riesz_func_inequ}.
\end{proof}

Following the same steps as in the proof of Lemma~\ref{lemma: vortex_second_order}, we obtain a similar estimate with adjusted convergence rates, which are given by Theorem~\ref{theorem: repulsive_coloumb}. 
\begin{lemma} \label{lemma: repulsive_second_order}
Let Assumption~\ref{ass: repulsive_coulomb} hold and \(\Phi \in FC^\infty(H^{-\lambda-2}(\T^d))\).
Then,  
\begin{align*} 
    &\int\limits_0^t \E\bigg[  
     \frac{\sigma^2}{2} \sum\limits_{j=1}^d \int_{\T^d}
       \partial_{x_j} \partial_{y_j}  ( \nabla^2 \hat T_{s,t}^n  \Phi(\rho_{s}^N))(x,y)_{\big| x=y } 
    \Id \rho_s^N ( x)  \bigg]  \Id s  \le C(T,\mu_0) [\Phi]_{C^2} \bigg( \sqrt{\log(N)} \indicator{d=2} + N^{1/3} \bigg)  
\end{align*}
\end{lemma}
\begin{proof}
    The proof, follows by similar arguments as in Lemma~\ref{lemma: vortex_second_order}, where the sharp relative entropy estimate~\cite[Theorm~1]{JabinWang2018} needs to be replaced by the relative entropy estimate given by Theorem~\ref{theorem: repulsive_coloumb}. 
\end{proof}

Combining Lemma~\ref{lemma: repulsive_first_order} and Lemma~\ref{lemma: repulsive_second_order} with Corollary~\ref{cor: repulsive_coulomb}, we obtain the weak Gaussian fluctuation estimate in the case of a repulsive Coulomb interaction.

\begin{theorem}\label{theorem: repulsive_fluc}
Let the interacting particle system be given by~\eqref{eq: repulsive_kernel} with the repulsive Coulomb kernel in dimensions \(d = 2, 3\), and suppose that Assumption~\ref{ass: repulsive_coulomb} holds. Then, for all \(\Phi \in C_\ell^2(H^{-\lambda-2}(\T^d))\), we have
\begin{equation*}
   \sup_{0 \le t \le T} \left| \mathbb{E} \big [ \Phi(\rho_t^N) - \Phi(\rho_t) \big]  \right|
\le C [\Phi]_{C^2(H^{-\lambda-2}(\T^2))} \left( \gamma(d, N) + W_{1, H^{-\lambda-2}(\T^d)}\Big( \P_{\rho^N_0}, \P_{\rho_0}\Big)  \right),
\end{equation*}
where \(\gamma(d,N)\) is given by~\eqref{eq: vortex_rate}. 
\end{theorem}
\begin{proof}
The Theorem follows similar to Theorem~\ref{theorem: main_result}. Notice that by Theorem~\ref{theorem: repulsive_coloumb} the relative entropy is bounded for each \(N \in \N\) and we can apply the Vitali convergence theorem in a similar fashion as in Theorem~\ref{theorem: main_result}.
\end{proof}

\begin{remark}
     Notice that, unlike in the previous example of the Vortex model, we do not obtain optimal convergence rates, but rather rates strictly smaller than \(N^{-1/2}\). However, due to the coercivity of the modulated energy approach~\cite{Serfaty2020}, the bounds on the modulated energy are sharp. Hence, Lemma~\ref{lemma: repulsive_first_order} cannot be improved. Therefore, it remains unclear whether, at least via the method presented in this paper and the modulated energy framework, Gaussian fluctuations with quantitative rates can be obtained for other Riesz kernels or in dimensions \(d > 3\).                     
\end{remark}

\appendix

\section{}
\label{sec:appendix}

\renewcommand{\thesection}{A}
\setcounter{theorem}{0}

\subsection{Sobolev and Besov spaces}

We recall some properties of Sobolev and Besov spaces, and also show a regularity result for the restriction to the diagonal of Sobolev functions. 

A consequence of Maurins theorem is the following embedding. 
\begin{lemma}[Hilbert--Schmidt Embedding]
\label{lem:HS_embedding}
    Let \(s,\tilde s \in \R\) and \(  s-\tilde s > d/2  \), then the embedding \(H^{s}(\T^d) \hookrightarrow H^{\tilde s}(\T^d)\) is Hilbert--Schmidt. 
\end{lemma}

We recall the multiplication inequalities for Besov spaces~\cite[Corollary~1 and Corollary~2]{Weber2017} 

\begin{lemma} \label{lemma: product_distr}
Let \(s_1 > 0 > s_2\) and \(p,q,p_1,p_2 \in [1, \infty]\) such that 
\begin{equation*}
    \frac{1}{p} = \frac{1}{p_1}+\frac{1}{p_2}.
\end{equation*}
Then, the map \((f,g) \mapsto fg\) extends to a continuous linear map from \(B_{p_1,q}^{s_1}(\T^d) \times B_{p_2,q}^{s_1}(\T^d) \) to \(B_{p,q}^{s_1}(\T^d)\) and 
\begin{equation*}
    \norm{fg}_{B_{p,q}^{s_1}(\T^d)} \le C \norm{f}_{B_{p_1,q}^{s_1}(\T^d)}
    \norm{g}_{B_{p_2,q}^{s_1}(\T^d)}.
\end{equation*}
If, in addition \(s_1+s_2 > 0\), then the map \((f,g) \mapsto fg\) extends to a continuous linear map from \(B_{p_1,q}^{s_1}(\T^d) \times B_{p_2,q}^{s_2}(\T^d) \) to \(B_{p,q}^{s_2}(\T^d)\) and 
\begin{equation*}
    \norm{fg}_{B_{p,q}^{s_2}(\T^d)} \le C \norm{f}_{B_{p_1,q}^{s_1}(\T^d)}
    \norm{g}_{B_{p_2,q}^{s_2}(\T^d)}.
\end{equation*}
\end{lemma}
Next, we also require Young's inequality~{\cite[Theorem 2.]{Schilling2022}}. 

\begin{lemma}[Young's inequality for Besov spaces] \label{lemma: young}
Let \( s \in \mathbb{R} \), \( q,q_1 \in (0, \infty] \), and \( p, p_1, p_2 \in [1, \infty] \) be such that:
\[
1 + \frac{1}{p} = \frac{1}{p_1} + \frac{1}{p_2} \quad \mathrm{and} \quad \frac{1}{q} \le \frac{1}{q_1} + \frac{1}{2}.
\]
If \( f \in B^{\alpha}_{p_1,q}(\T^d) \) and \( g \in L_{p_2}(\T^d) \), then \( f * g \in B^{s}_{p,q}(\T^d) \) and
\[
\|f * g\|_{B^{s}_{p,q}(\T^d)} \leq C \|f\|_{B^{s}_{p_1,q_1}(\T^d)} \cdot \|g\|_{L^{p_2}(\T^d)},
\]
where \( C > 0 \) is a constant independent of \( f \) and \( g \). 
\end{lemma}

\noindent The following analyzes the  regularity of a Sobolev function on the diagonal. This is a kind of Trace Theorem for Sobolev spaces onto a subspace, for which there might be some results in the literature that, however, we haven't found; we hence provide a proof based on the Fourier expansion.  

\begin{lemma} \label{lemma: diagonal}
Let \(f \in H^{s}(\T^d\times \T^d)\) for \(s > d\) and let \(g(x):=f(x,x)\). Then \(g\in H^{s'}(\T^d)\) for \( s' \leq s-d/2\) and 
\[\norm{g}_{H^{s'}(\T^d) } \le C \norm{f}_{H^{s}(\T^d\times \T^d ) } .  \]
    
\end{lemma}
\begin{proof}
    First by Sobolev embedding the function \(f\) is continuous and therefore \(g\) is well-defined. Let \(l=(l_1,l_2) \in \Z^d \times \Z^d\), \(\tilde l=(\tilde l_1,\tilde l_2) \in \Z^d \times \Z^d\).
    We have 
    \begin{align*}
    \norm{g}_{H^{s'}(\T^d)}^2 &= 
    \sum\limits_{k \in \Z^d} \langle k \rangle^{2s'} |\langle g ,e_k\rangle|^2  \\
    &=  \sum\limits_{k \in \Z^d} \langle k \rangle^{2s'} \bigg| \sum\limits_{l \in \Z^d \times \Z^d}  \langle f, e_l \rangle   \langle e_{l_1+l_2} ,e_k\rangle \bigg|^2 \\
    &\le \sum\limits_{k \in \Z^d}  \sum\limits_{l \in \Z^d \times \Z^d}  \sum\limits_{\tilde l \in \Z^d \times \Z^d} \langle l \rangle^{2s}  |\langle f, e_l \rangle|^2    \langle k \rangle^{2s'}  \langle \tilde  l \rangle^{-2s}  \langle e_{ l_1+  l_2} ,e_k\rangle \langle e_{\tilde l_1+ \tilde l_2} ,e_k\rangle   \\
    &=  \sum\limits_{ n \in \Z^d} \sum\limits_{\substack{ l_1, l_2 \in \Z^d \\  l_1 + l_2 =n}}   \sum\limits_{ \tilde n \in \Z^d} \sum\limits_{\substack{\tilde l_1,\tilde l_2 \in \Z^d \\ \tilde l_1 +\tilde l_2 =\tilde n}} \sum\limits_{k \in \Z^d}  \langle l \rangle^{2s}  |\langle f, e_l \rangle|^2   \langle k \rangle^{2s'}  \langle \tilde  l \rangle^{-2s} \langle e_{ l_1+  l_2} ,e_k\rangle  \langle e_{\tilde l_1+ \tilde l_2} ,e_k\rangle   \\ 
     &\le  \sum\limits_{ n \in \Z^d} \sum\limits_{\substack{ l_1, l_2 \in \Z^d \\  l_1 + l_2 =n}}  \sum\limits_{\substack{\tilde l_1,\tilde l_2 \in \Z^d \\ \tilde l_1 +\tilde l_2 =n}}  \langle l \rangle^{2s}  |\langle f, e_l \rangle|^2   \langle n \rangle^{2s'}  \langle \tilde  l \rangle^{-2s}      
    \end{align*}
    For the last series we find 
    \begin{align*}
        \sum\limits_{\substack{\tilde l_1,\tilde l_2 \in \Z^d \\ \tilde l_1 +\tilde l_2 =n}}   \langle \tilde  l \rangle^{-2s} 
        &\le \sum\limits_{\tilde l_1 \in \Z^d}   (1+|\tilde l_1|^2+|n-\tilde l_1|^2)^{-s}
        = \sum\limits_{\tilde l_1 \in \Z^d}   (1+\tfrac{3}{4}|n_1|^2+2\big|\tfrac{n}{2}-\tilde l_1 \big|^2)^{-s} \\
        &\le C \sum\limits_{\tilde l_1 \in \Z^d}   (1+|n_1|^2+\big|\tfrac{n}{2}-\tilde l_1 \big|^2)^{-s} = C \langle n \rangle^{-2s} \sum\limits_{\tilde l_1 \in \Z^d} \bigg( 1+  \frac{\big|\tfrac{n}{2}-\tilde l_1 \big|}{1+|n|^2} \bigg)^{-s}  . 
    \end{align*}
    Using the integral criteria for the series and a change of variables, we obtain the bound 
    \begin{equation*}
         \sum\limits_{\substack{\tilde l_1,\tilde l_2 \in \Z^d \\ \tilde l_1 +\tilde l_2 =n}}   \langle \tilde  l \rangle^{-2s} 
         \le C \langle n \rangle^{d-2s} .
    \end{equation*}
Consequently, we obtain 
\begin{align*}
    \norm{g}_{H^{s'}(\T^d)} &\le C \sum\limits_{ n \in \Z^d} \langle n \rangle^{2s'+d-2s} \sum\limits_{\substack{ l_1, l_2 \in \Z^d \\  l_1 + l_2 =n}}  
     \langle l \rangle^{2s}  |\langle f, e_l \rangle|^2 \\
     &\le C \sum\limits_{ n \in \Z^d} \sum\limits_{\substack{ l_1, l_2 \in \Z^d \\  l_1 + l_2 =n}}  
     \langle l \rangle^{2s}  |\langle f, e_l \rangle|^2 \\
     &\le C \norm{f}_{H^s(\T^d \times \T^d)}^2,
\end{align*}
where we used the fact that \(2s'-d-2s \leq 0 \). 
\end{proof}

\subsection{Gy\"ongy--Krylov criterion}
We recall the Gy\"ongy--Krylov criterion~\cite{Gyongy1996}.
If \((E,d)\) is a metric space we denote by \((E^2,d^2)\) the product space with the metric given by \((d((x,y),(x',y'))=d(x,y)+d(x',y')\) and equipped with the Borel sigma algebra. 
Let \(D = \{(x, x) \in E^2 \, ; \, x \in E\}\) be the diagonal.

\begin{lemma} \label{lemma: gyongy}
Let \((X_n)_{n \in \N}\) be a sequence of random variables from a probability space \((\Omega, \cF, \P)\) to a complete separable metric space \((E, d)\). Assume that, for every pair of subsequences \((n_1(k), n_2(k))_{k \in \N}\), with \(n_1(k) \geq n_2(k)\) for every \(k \in \N\), there exists a subsequence \((k(h))_{h \in \N}\) such that the random variables
\[
(X_{n_1(k(h))}, X_{n_2(k(h))})_{h \in \N} \colon (\Omega, \cF, \P) \to (E^2, d_2)
\]
converge in law to a probability measure \(\mu\) on \(E^2\) such that \(\mu(D) = 1\). Then there exists a random variable
\[
X \colon (\Omega, \cF, \P) \to (E, d)
\]
such that \(X_n \to X\) in probability.    
\end{lemma}

\subsection{Approximation with cylindrical functions}
We provide the main approximation result for a class of continuous  functions. 
\begin{definition} \label{def: cylindrical_class}
    Let \(s \ge 0 \). We define the set of \emph{cylindrical functions} as  
    \begin{align*}
     FC^\infty( H^{-s}(\T^d)) := \{  \Phi \colon H^{-s}(\T^d) \to \R \, : \,  \exists m \in \N , \varphi_1, \ldots, \varphi_m \in C^{\infty}(\T^d) , g \in C^\infty_c(\R^m) \\
    \mathrm{such \; that} \; \Phi(f) = g( \langle f, \varphi_1\rangle_{H^{-s}(\T^d)} , \ldots,\langle f, \varphi_n  \rangle_{H^{-s}(\T^d)} ), \forall f \in   H^{-s}(\T^d))\}. 
    \end{align*}
\end{definition}
Note that functions in $FC^\infty( H^{-s}(\T^d))$ are in  $C^\infty( H^{-s}(\T^d))$ with derivatives all bounded. The following Lemma is a key ingredient to compute the weak error in a rigorous manner, and is a variation of the result in~\cite[Lemma~B.78]{Gozzi2017}.  Note that convergence is just on compact sets, as it is known that approximation results by smooth functions can not hold on the whole space with the uniform norm; 
see~\cite{nemirovski1973polynomial}.
\begin{lemma} \label{lemma: approximation_hilbert_smooth}
   Let \(s\ge 0 \) and \(\Phi \in  C^2_{\ell}(H^{-s}(\T^d)) \). Then there exists a sequence \((\Phi_n, n \in \N)\) in \( FC^\infty( H^{-s}(\T^d))\) such that for each compact set \(K \subset H^{-s}(\T^d)\) we have 
 \begin{align}
  \lim\limits_{n \to \infty}   \sup\limits_{f \in K} | \Phi_n(f) -\Phi(f) | &= 0,
  \label{eq:Phi_n_1}
  \\
  \| \Phi_n \|_{C_\ell(H^{-s}(\T^d))} &\leq 2\|\Phi\|_{C_\ell(H^{-s}(\T^d))} 
  \label{eq:Phi_n_2}\\
  [\Phi_n]_{C^2(H^{-s}(\T^d))} &\le  [\Phi]_{C^2(H^{-s}(\T^d))}.
  \label{eq:Phi_n_3}
 \end{align}
 Further, if $\Phi$ is just in $C_b(H^{-s}(\T^d))$ then \eqref{eq:Phi_n_1} holds and 
 \begin{equation}
 \label{eq:Phi_n_4}
\| \Phi_n \|_{C_b(H^{-s}(\T^d))} \leq \|\Phi\|_{C_b(H^{-s}(\T^d))}. 
 \end{equation} 
\end{lemma}

\begin{proof}
    We will reduce the problem. 
We claim that we can actually choose \(\varphi_1,\ldots, \varphi_m  \in H^{s}(\T^d) \). Choose a sequence \((h_n^i,n \in \N)\) of smooth functions converging towards \(\varphi_i\) in \( H^{s}(\T^d) \). Then, we find 
\begin{align*}
        &|g( \langle f, \varphi_1\rangle_{H^{-s}(\T^d)} , \ldots,\langle f, \varphi_m  \rangle_{H^{-s}(\T^d)} )- g( \langle f, h_n^1 \rangle_{H^{-s}(\T^d)} , \ldots,\langle f, h_n^m \rangle_{H^{-s}(\T^d)} )|^2 \\
    &\quad \le \norm{g}_{C^1(\R^m)} \sum\limits_{j=1}^m |\langle f, \varphi_j - h_n^j \rangle_{H^{-s}(\T^d)}|^2 \\
    &\quad \le \norm{g}_{C^1(\R^m)} \sum\limits_{j=1}^m \norm{f}_{H^{-s}(\T^d)}^2 \norm{ \varphi_j -  h_n^j}_{H^{-s}(\T^d)}^2 \\
&\quad \to 0 , \quad \mathrm{as} \; n \to \infty.
\end{align*}
Therefore, it remains to show the claim for the new set 
     \begin{align*}
    \widetilde{ FC}^\infty( H^{-s}(\T^d)) := \{  \Phi \colon H^{-s}(\T^d) \to \R \, : \,  \exists m \in \N , \varphi_1, \ldots, \varphi_m \in H^{s}(\T^d) , g \in C^\infty(\R^m) \\
    \mathrm{such \; that} \; \Phi(f) = g( \langle f, \varphi_1\rangle_{H^{-s}(\T^d)}, \ldots,\langle f, \varphi_m  \rangle_{H^{-s}(\T^d)}), \,\, \forall f \in   H^{-s}(\T^d))\}. 
    \end{align*}
    The convergence claim on compact sets \eqref{eq:Phi_n_1}  as well as the preservation of linear growth norm \eqref{eq:Phi_n_2} and uniform norm \eqref{eq:Phi_n_4} follow immediately by~\cite[Lemma~B.78]{Gozzi2017}, which provides an explicit approximation based of finite dimensional projection. 
    For a given $\Phi$, denoting $P_n:H^{-s}(\T^d)\rightarrow H^{-s}(\T^d)$ the orthogonal projection into the the linear span of the first $n$ elements of a orthonormal basis and $Q_n:\R^n \rightarrow H^{-s}(\T^d)$ the corresponding embedding, such approximation is defined by
    \[
    \Psi^n_k (f) = \int_{\R^n} \Phi(P_n f -Q_n y) \eta_k(y) dy, 
    \]
    where $\eta_k:\R^n\rightarrow \R$ is a smooth mollifier with support on the ball of radius $1/k$. Such approximation immediately gives the seminorm bounds \eqref{eq:Phi_n_3}. 
    %
\end{proof}

The derivatives of functions in this class can be easily computed:

\begin{lemma}\label{lemma: derivative}
Let \(s >0\), \(\Phi \in  FC^\infty(H^{-s}(\T^d)\) with the representation 
\[
\Phi(f) = g( \langle f, \varphi_1\rangle_{H^{-s}(\T^d)}  , \ldots,\langle f, \varphi_m   \rangle_{H^{-s}(\T^d)}  ) 
\] 
for \( g \in C^\infty_c(\R^m), \varphi_1, \ldots, \varphi_m \in C^\infty(\T^d)  \).
Then, we have the following formulas for the Fréchet derivatives: 
\begin{align*}
    \nabla \Phi: H^{-s}(\T^d) &\to H^{-s}(\T^d)   \\
    f &\mapsto  \sum\limits_{i=1}^m \partial_{x_i} g( \langle f, \varphi_1\rangle_{H^{-s}(\T^d)}  , \ldots,\langle f, \varphi_m  \rangle_{H^{-s}(\T^d)}  ) \varphi_i,
%
\end{align*}
and 
\begin{align*}
    \nabla^2 \Phi: H^{-s}(\T^d) &\to L(H^{-s}(\T^d),H^{-s}(\T^d))\\
    f &\mapsto \bigg( h \mapsto \sum\limits_{i,j=1}^m \partial_{x_i} \partial_{x_j} g( \langle f, \varphi_1\rangle_{H^{-s}(\T^d)}  , \ldots,\langle f, \varphi_m  \rangle_{H^{-s}(\T^d)}  ) \langle \varphi_j,h \rangle_{H^{-s}(\T^d)} 
    \varphi_i \bigg).
\end{align*}
\end{lemma}

\providecommand{\bysame}{\leavevmode\hbox to3em{\hrulefill}\thinspace}
\providecommand{\MR}{\relax\ifhmode\unskip\space\fi MR }
\providecommand{\MRhref}[2]{%
	\href{http://www.ams.org/mathscinet-getitem?mr=#1}{#2}
}
\providecommand{\href}[2]{#2}

\medskip

\end{document}